%% file: KhovanovLeeVirtuals.tex
\documentclass[11pt]{article} 
\usepackage{latexsym}
\usepackage{amsmath,amsthm,amssymb,amsfonts,amscd}
\usepackage{graphicx}
\usepackage{epstopdf}
\usepackage{longtable}
\usepackage{multirow}
\usepackage{float}
\usepackage{hyperref}
\usepackage{stmaryrd}
\usepackage{fancyhdr}
\usepackage{color}
\usepackage{url}
\usepackage{subcaption}
\usepackage[all]{xy}
\usepackage{endnotes}
\usepackage{mathtools}
\usepackage{epstopdf}

\theoremstyle{plain}
\newtheorem{thm}{Theorem}[section]
\newtheorem{cor}[thm]{Corollary}
\newtheorem{lem}[thm]{Lemma}
\newtheorem*{lem*}{Lemma}
\newtheorem{prop}[thm]{Proposition}

\theoremstyle{remark}
\newtheorem{exa}{Example}[section]

\theoremstyle{remark}
\newtheorem{rem}{Remark}[section]

\theoremstyle{remark}
\newtheorem{defn}{Definition}[section]

\newcommand{\degree}{\ensuremath{^\circ}}
\newcommand{\Across}{\raisebox{-0.25\height}{\includegraphics[width=0.5cm]{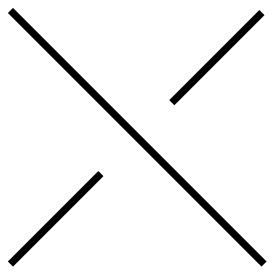}}}
\newcommand{\Asmooth}{\raisebox{-0.25\height}{\includegraphics[width=0.5cm]{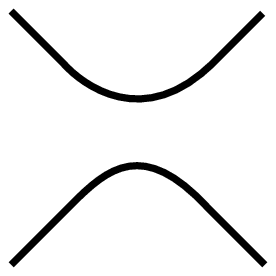}}}
\newcommand{\Bsmooth}{\raisebox{-0.25\height}{\includegraphics[width=0.5cm]{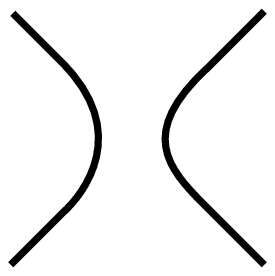}}}

\newcommand{\Rcurl}{\raisebox{-0.25\height}{\includegraphics[width=0.5cm]{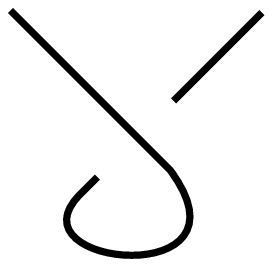}}}
\newcommand{\Lcurl}{\raisebox{-0.25\height}{\includegraphics[width=0.5cm]{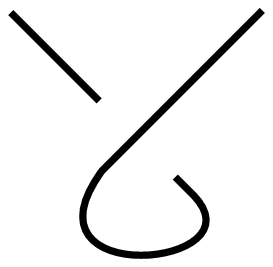}}}
\newcommand{\Arc}{\raisebox{-0.25\height}{\includegraphics[width=0.5cm]{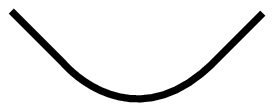}}}
\newcommand{\rgA}{\raisebox{-0.45\height}{\includegraphics[width=0.6cm]{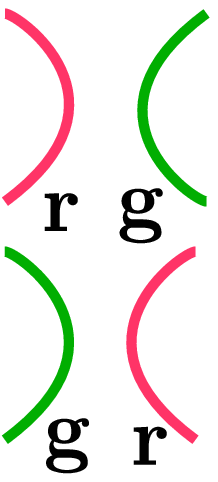}}}
\newcommand{\rgB}{\raisebox{-0.4\height}{\includegraphics[width=0.7cm]{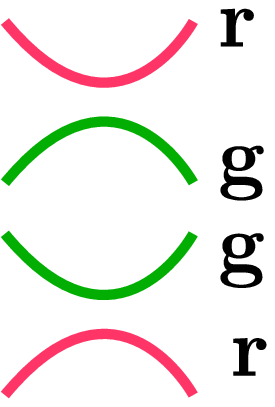}}}
\newcommand{\rrggA}{\raisebox{-0.45\height}{\includegraphics[width=0.6cm]{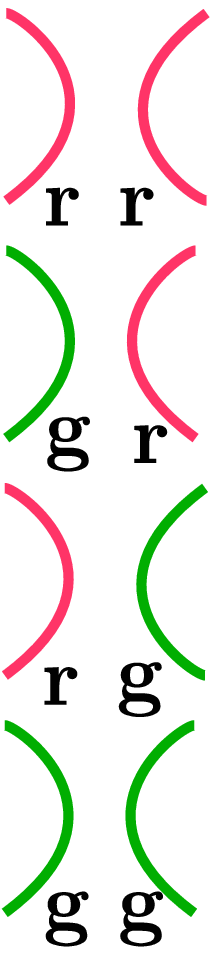}}}
\newcommand{\rrggB}{\raisebox{-0.4\height}{\includegraphics[width=0.7cm]{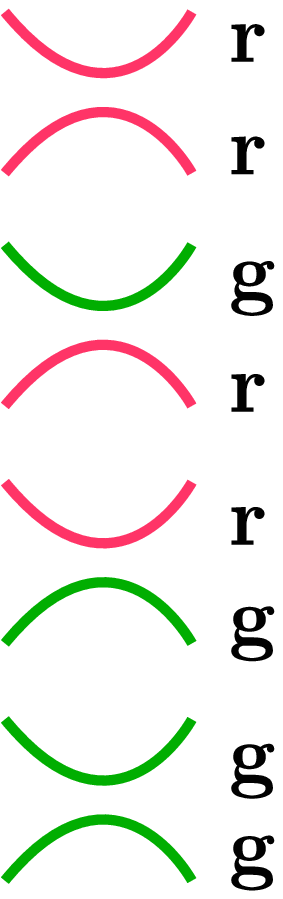}}}
\newcommand{\saddler}{\raisebox{-0.45\height}{\includegraphics[width=0.6cm]{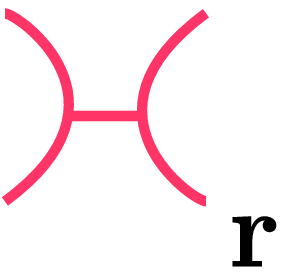}}}
\newcommand{\saddleg}{\raisebox{-0.45\height}{\includegraphics[width=0.6cm]{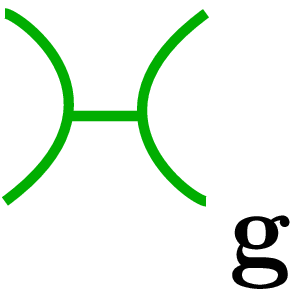}}}
\newcommand{\RIb}{\raisebox{-0.25\height}{\includegraphics[height=.9cm]{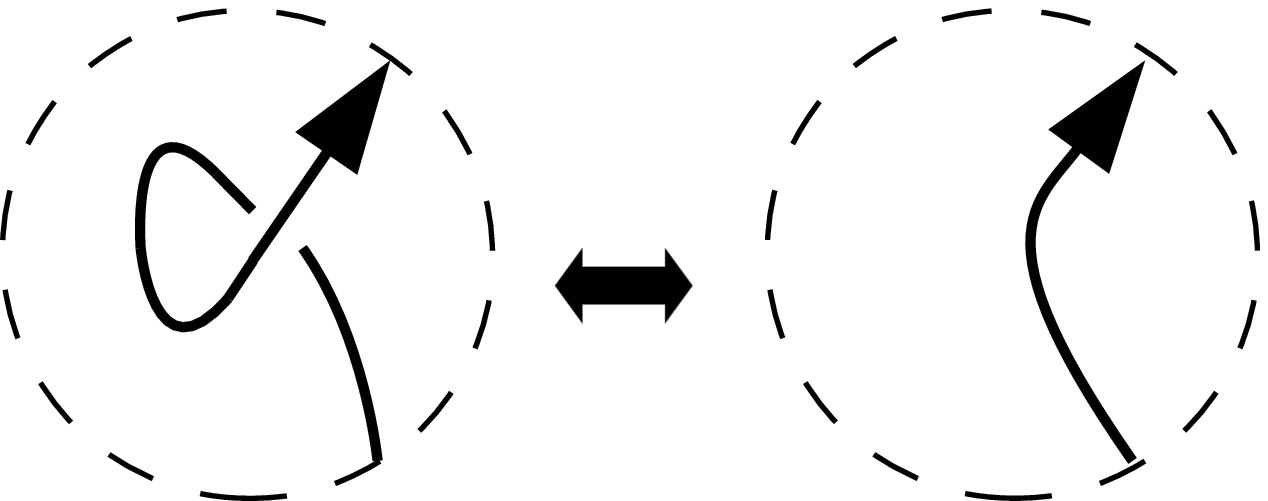}}}
\newcommand{\RId}{\raisebox{-0.25\height}{\includegraphics[height=.9cm]{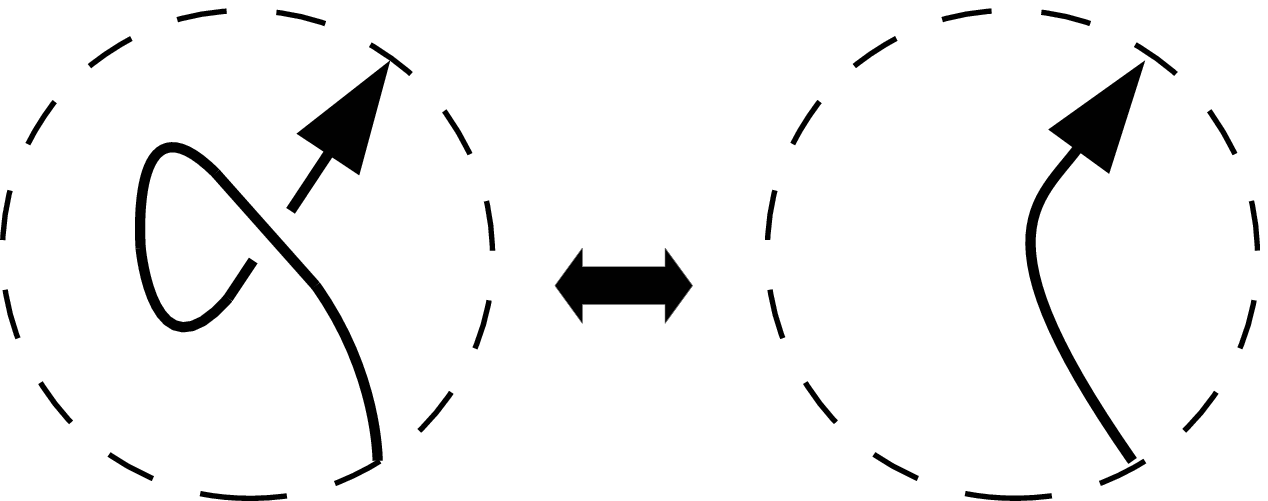}}}
\newcommand{\RIIc}{\raisebox{-0.25\height}{\includegraphics[height=.9cm]{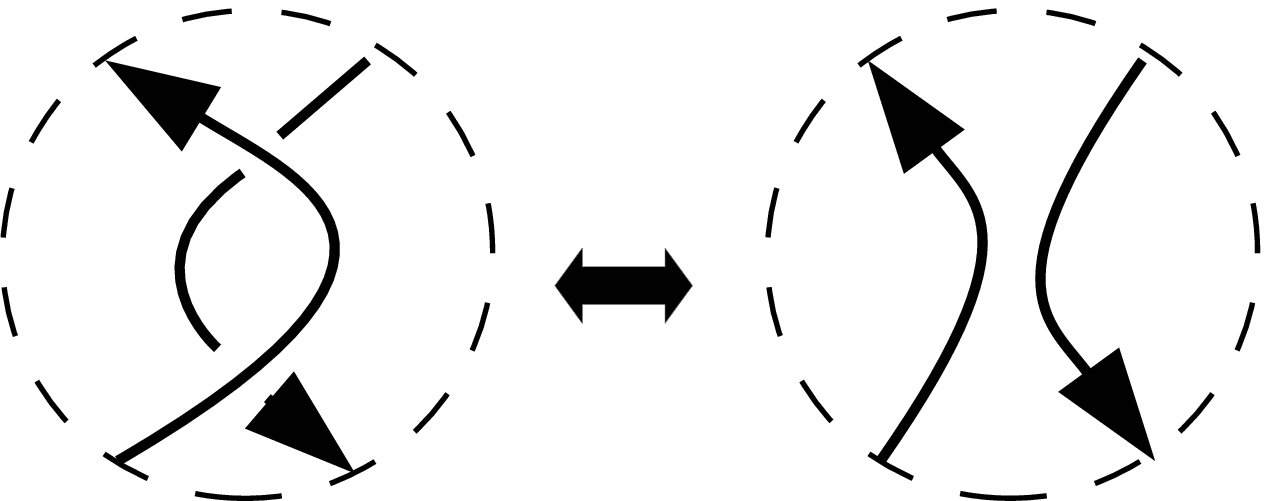}}}
\newcommand{\RIId}{\raisebox{-0.25\height}{\includegraphics[height=.9cm]{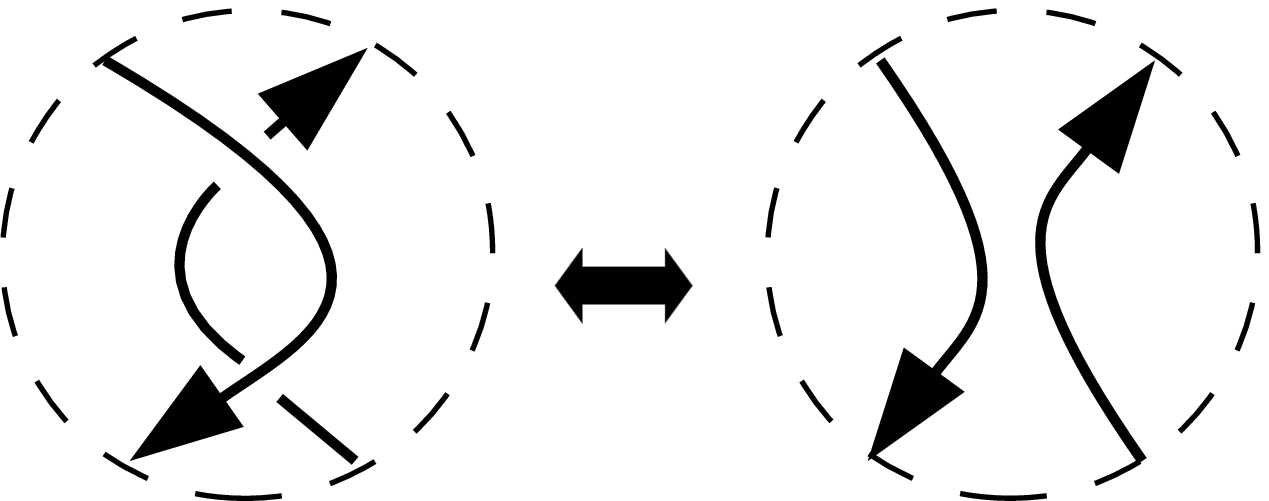}}}
\newcommand{\RIIIb}{\raisebox{-0.25\height}{\includegraphics[height=.9cm]{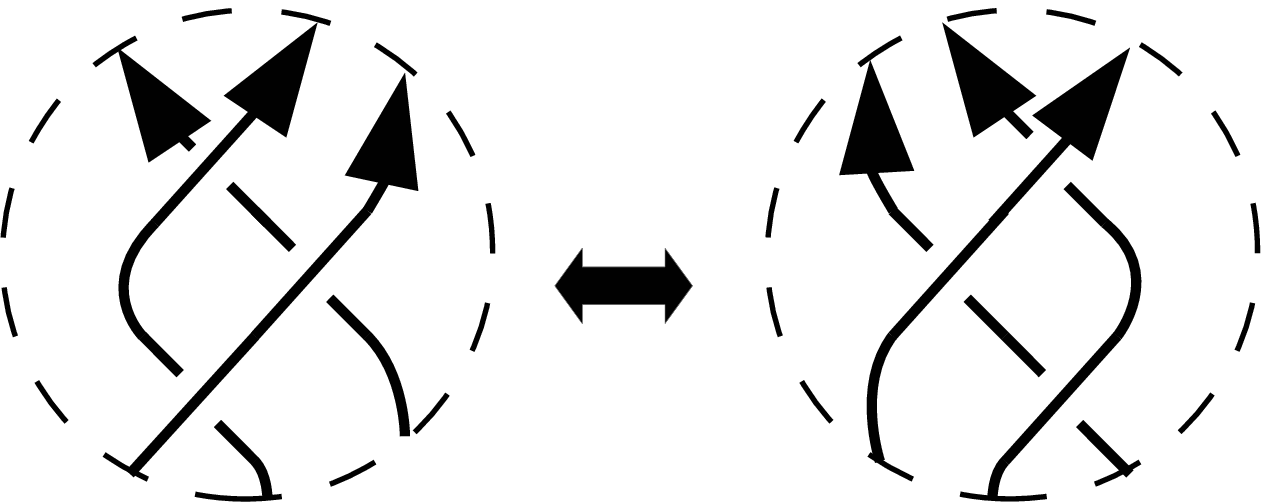}}}
\newcommand{\RIbLeft}{\raisebox{-0.25\height}{\includegraphics[height=.9cm]{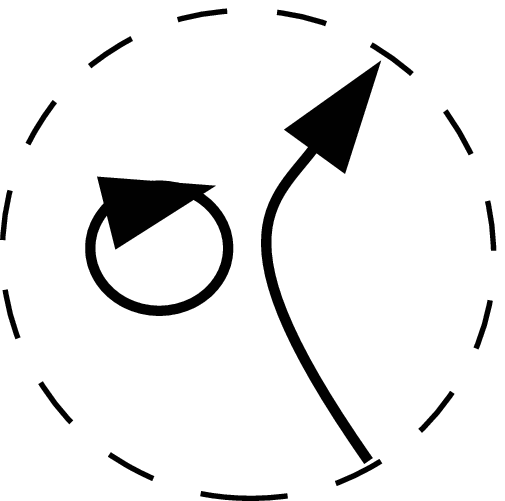}}}

\newcommand{\RIIcLeft}{\raisebox{-0.25\height}{\includegraphics[height=.9cm]{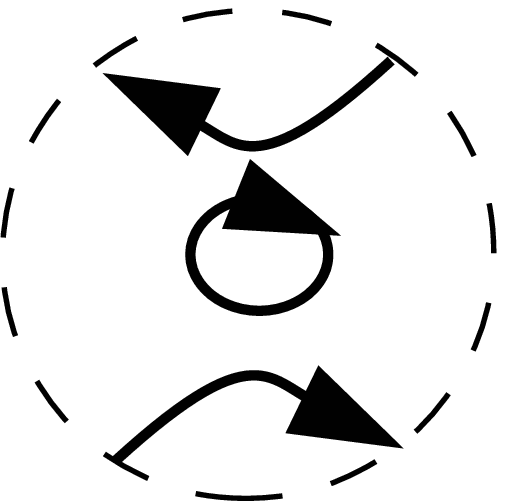}}}

\newcommand{\RIIIbLeft}{\raisebox{-0.25\height}{\includegraphics[height=.9cm]{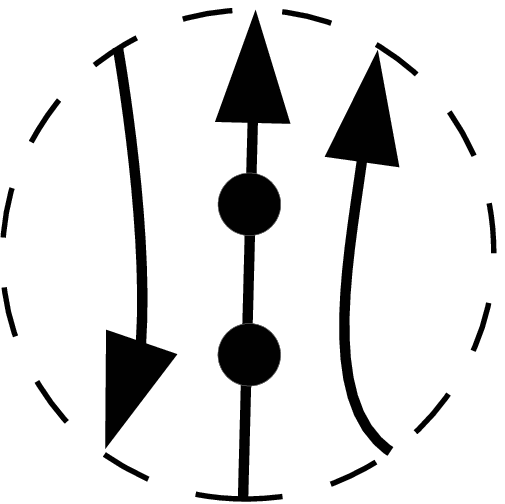}}}
\newcommand{\RIbRight}{\raisebox{-0.25\height}{\includegraphics[height=.9cm]{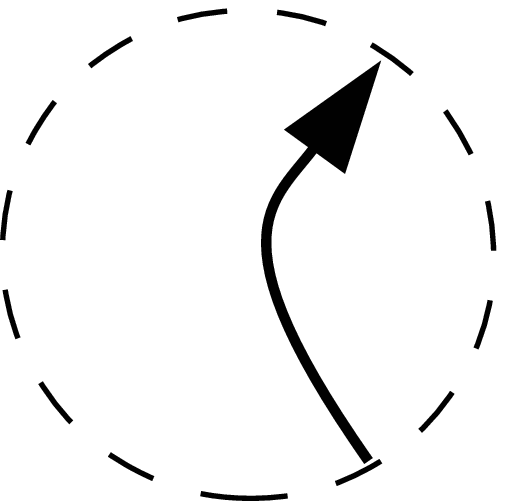}}}

\newcommand{\RIIcRight}{\raisebox{-0.25\height}{\includegraphics[height=.9cm]{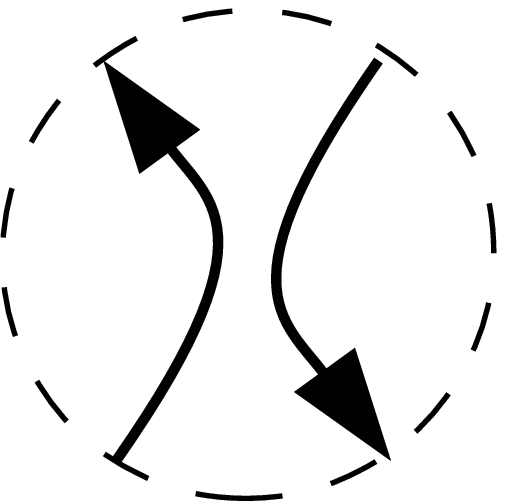}}}

\newcommand{\RIIIbRight}{\raisebox{-0.25\height}{\includegraphics[height=.9cm]{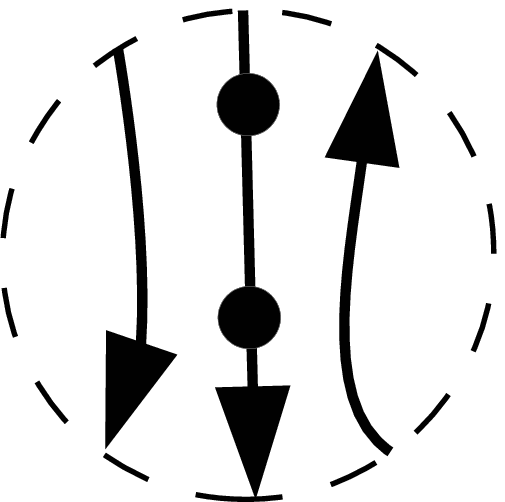}}}

\title{Khovanov Homology, Lee Homology and a Rasmussen Invariant for Virtual Knots\footnote{AMS classification codes: 57M27 , 57M99}}
\renewcommand\footnotemark{}

\author{Heather A. Dye\endnote{Division of Science and Mathematics, McKendree University, 701 College Rd, Lebanon, Illinois 62254, USA}, Aaron Kaestner\endnote{Department of Mathematics, 3225 W Foster Ave, Box 57, North Park University, Chicago, Illinois 60625, USA}, Louis H. Kauffman\endnote{Department of Mathematics, Statistics and Computer Science (m/c 249), 851 South Morgan Street, University of Illinois at Chicago, Chicago, Illinois 60607-7045, USA}}
\date{\today}

\begin{document}

\maketitle

\begin{abstract}
The paper contains an essentially self-contained treatment of Khovanov homology, Khovanov-Lee homology as well as the Rasmussen invariant for virtual knots and virtual knot cobordisms which directly applies as well  to classical knot and classical knot cobordisms. We give an alternate formulation for the Manturov definition \cite{ArbitraryCoeffs} of Khovanov homology \cite{Khovanov1} \cite{Khovanov2} for virtual knots and links with arbitrary coefficients. This approach uses cut loci on the knot diagram to induce a conjugation operator in the Frobenius algebra. We use this to show that a large class of virtual knots with unit Jones polynomial are non-classical, proving a conjecture in \cite{VKT} and \cite{OpenProbsVKT}. We then discuss the implications of the maps induced in the aforementioned theory to the universal Frobenius algebra \cite{KhovanovFrobenius} for virtual knots.  Next we show how one can apply the Karoubi envelope approach of Bar-Natan and Morrison \cite{BarNatanMorrisonKaroubi} on abstract link diagrams \cite{KamadaALD} with cross cuts to construct the canonical generators of the Khovanov-Lee homology \cite{Endomorphism}. Using these canonical generators we derive a generalization of the Rasmussen invariant \cite{Rasmussen} for virtual knot cobordisms and generalize Rasmussen's result on the slice genus for positive knots to the case of positive virtual knots. It should also be noted that this generalization of the Rasmussen invariant provides an easy to compute obstruction to knot cobordisms in $S_g \times I \times I$ in the sense of Turaev \cite{Turaev}.
\end{abstract}

\section{Introduction}
This paper gives a self-contained introduction to Khovanov homology and a generalization of Khovanov homology to virtual knot theory. Virtual knot theory is a natural generalization of classical knot theory to the study of knots in thickened orientable surfaces where embeddings of curves in surfaces are taken up to handle stabilization. We prove a number of new results for virtual knot theory with implications for classical knots. We have made every effort to make the paper self-contained, and hence an introduction to both Khovanov homology, Khovanov-Lee Homology and to the subject of virtual knots and links. The material in this paper will be of interest to topologists as an introduction to these subjects and to specialists for the particular results that we prove.\\

Khovanov homology was discovered by Michail Khovanov \cite{Khovanov1} as a {\it categorification} of the Jones polynomial \cite{JO}. This means that the Jones polynomial can be retrieved from Khovanov homology as a graded Euler characteristic, and indeed the process of forming this homology theory involves up-leveling many aspects of the Kauffman Bracket (state summation) model of the Jones polynomial to objects and morphisms in appropriate categories. The entire collection of states of the bracket polynomial for a given knot diagram becomes a category (see the next section of this paper) and one wants to measure how this category changes under isotopy of the given diagram of the knot. The homology theory is a way of measuring this category of states in such a way that the homology is invariant under isotopies of the knot diagrams (i.e. invariant under the Reidemeister moves). The resulting construction of Khovanov has a long reach and has led to other categorifications of other link invariants and to connections with other areas of mathematics.  {\it Khovanov homology detects the classical unknot.} That is, if $K$ is a classical knot diagram and $K$ has trivial Khovanov homology, then $K$ is isotopic to an unknotted circle. This result was proved by Kronheimer and Mrowka \cite{UnknotDetector} by relating Khovanov homology with a version of Floer homology for knots and hence is a proof using aspects of gauge theory. It remains an open question whether the original Jones polynomial detects classical unknots.\\

Virtual knot theory studies knots in thickened surfaces. It is a natural extension of classical knot theory which studies knots and links in the thickened two dimensional sphere (one can remove a point from Euclidean three-space without disturbing its knot theory). In fact, virtual knot theory uses  a natural extension of the diagrammatic representation of classical knots. Virtual knots can be represented
by diagrams with one extra type of crossing, called virtual. See Figure~\ref{virtseifert} for examples of diagrams with virtual crossings. One can think of a virtual crossing in the following way: Imagine drawing a knot or link diagram on a closed surtace. There will be ordinary knot diagrammatic crossings on this surface. However if you project this surface to a plane then curves that wind around the surface will sometimes create shadow crossings when there is no corresponding crossing on the surface itself. Such crossings in the projection are labeled virtual crossings. Taking diagrams with virtual crossings, one can define a set of moves (See Figure~\ref{fig:VRMs}) for them so that two virtual diagrams are equivalent via these extended moves if and only if the corresponding knots or links in thickened surfaces are stably equivalent (meaning equivalent up to ambient isotopy in the thickened surfaces plus adding or subtracting 1-handles from the thickened surface in the complement of the embedding for the knot or link). The diagrammatic interpretation of virtual knot theory is convenient for many purposes since one can analyze invariants of virtual knots by working with the moves. In particular,
since Khovanov homology is defined in terms of bracket states for the classical diagrams, one can make generalizations of it for virtual knots and links by working with the virtual diagrams. We accomplish this construction in the body of the paper. Our construction of Khovanov homology for virtual knots is a reformulation of the theory developed by Vassily Manturov in \cite{ArbitraryCoeffs}. The reader interested in seeing more background about virtual knot theory can consult \cite{OpenProbsVKT,VKTI,ManturovVirtualKnots}.\\

One of the simplest extensions of an invariant of classical knots to virtual knots is the extension of the bracket polynomial \cite{VKTI}. Here one uses the usual expansion of the bracket polynomial at the classical crossings and evaluates the resulting loops with virtual crossings by counting the number of loops just as in the classical case. This gives rise to a definition of the Jones polynomial for virtual knots. Remarkably there is a construction \cite{VKTI} that yields infinitely many non-trivial virtual knots with unit Jones polynomial. The construction chooses a subset $S$ of crossings from a classical knot diagram $K$ such that the classical knot diagram $S(K)$ obtained by switching all the crossings in $S$ is unknotted (this can always be done with a classical diagram). Then we apply a construction called {\it virtualization} to the crossings in $S$ (replacing a crossing by adding a virtual crossing on either side), as shown in Figure~\ref{f7}, to obtain a virtual knot $Virt(K).$ One can show that $Virt(K)$ is non-trivial and has unit Jones polynomial. This leads the question of finding out whether $Virt(K)$ is classical or non-classical, and if non-classical one would like find out its {\it virtual genus} where the virtual genus of a virtual knot is the least genus thickened surface in which it can be represented. The virtual knots of genus zero are the classical knots.
We have conjectured that when $K$ is a non-trivial classical knot, then all the knots of type $Virt(K)$ have genus greater than zero, that is, they are not classical. If $Virt(K)$ were classical for some $K$, then $Virt(K)$ would be a classical non-trivial knot with unit Jones polynomial.  In this paper we prove that the knots of the form $Virt(K)$ are all of virtual genus greater than one, by using our version of Khovanov homology for virtual knots, coupled with the theorem of Kronheimer and Mrowka that standard Khovanov homology detects the classical unknot. From our point of view this is a good resolution of the problem. It remains to find the genus and other properties of the knots $Virt(K)$ but we now know that they are all non-classical.\\

Rasmussen \cite{Rasmussen}, by using Khovanov homology and a variant due to Lee \cite{Endomorphism} was able to create a new invariant of knots and links and to prove that positive classical links have four-ball genus equal to the genus of the Seifert spanning surface of the link in three dimensional space. Recall that the four-ball genus of a classical link is the least genus of a tamely embedded surface in the four-ball that bounds the link in three-space. This work of Rasmussen gives an elementary proof of a conjecture of Milnor about torus links. In this paper, we give a self-contained treatment of the Rasmussen invariant in the classical case and we generalize it to virtual knots and links. We give a definition of the {\it virtual four-ball genus} of virtual link and a notion of the 
{\it virtual Seifert surface} of a virtual link. We prove an analogue to the Rasmussen Theorem to the effect that a positive virtual link has virtual four-ball genus equal to the genus of its virtual Seifert surface.
This marks the beginning of an investigation of Rasmussen invariants for virtual links that will continue in other papers.\\

The rest of the paper is organized as follows. Section 2 discusses background material, including the bracket polynomial, classical Khovanov homology and the basics of virtual knot theory.
Section 3 constructs Khovanov homology for virtual knots and links and proves the theorem alluded to above about the non-classicality of $Virt(K)$ examples of virtual knots with unit Jones polynomial.
Section 4 constructs Lee Homology for virtual knots and our generalization of the Rasmussen invariant. Section 5 discusses virtual knot cobordism starting from the formulation given in \cite{VKC}.
In Section 6 we then combine this with the work of Section 4 to prove our result about positive virtual links. The appendix handles certain technical points about well-definition of the generalized Khovanov homology.\\ 

\noindent {\bf Remark.} It is curious that the original verision of Khovanov homology for classical knots is tied to the plane in such a way that it must be subtlety modified in order to extend to virtual knots.
This is a property not shared by the $sl(n)$ generalization of Khovanov and Rozansky. Here we have made the best compromise that we know in extending Khovanov homology to virtual and hence to knots embedded in thickened surfaces.\\

\section{Background}

\subsection{Bracket and Jones Polynomials}\label{subsec:BracketJonesPoly}
The bracket polynomial \cite{KaB} model for the Jones polynomial \cite{JO,JO1,JO2,Witten} is usually described by the expansion
$$\langle \Across \rangle=A \langle \Asmooth \rangle + A^{-1}\langle \Bsmooth \rangle$$
Here the small diagrams indicate parts of an otherwise identical larger knot or link diagrams. The two types of smoothing (local diagram with no crossing) in
this formula are said to be of type $A$ ($A$ above) and type $B$ ($A^{-1}$ above).

$$\langle \bigcirc \rangle = -A^{2} -A^{-2}$$
$$\langle K \, \bigcirc \rangle=(-A^{2} -A^{-2}) \langle K \rangle $$
$$\langle \Rcurl \rangle=(-A^{3}) \langle \Arc \rangle $$
$$\langle \Lcurl \rangle=(-A^{-3}) \langle \Arc \rangle $$
One uses these equations to normalize the invariant and make a model of the Jones polynomial.
In the normalized version we define $$f_{K}(A) = (-A^{3})^{-wr(K)} \langle K \rangle / \langle \bigcirc \rangle $$
where the writhe $wr(K)$ is the sum of the oriented crossing signs for a choice of orientation of the link $K.$ Since we shall not use oriented links
in this paper, we refer the reader to \cite{KaB} for the details about the writhe. One then has that $f_{K}(A)$ is invariant under the Reidemeister moves
(again see \cite{KaB}) and the original Jones polynomial $V_{K}(t)$ is given by the formula $$V_{K}(t) = f_{K}(t^{-1/4}).$$ The Jones polynomial has been of great interest since its discovery in 1983 due to its relationships with statistical mechanics, its ability to often detect the difference between a knot and its mirror image, and the many open problems and relationships of this invariant with other aspects of low dimensional topology.
\bigbreak

\subsubsection{The State Summation} In order to obtain a closed formula for the bracket, we now describe it as a state summation.
Let $K$ be any unoriented link diagram. Define a {\em state}, $S$, of $K$  to be the collection of planar loops resulting from  a choice of
smoothing for each  crossing of $K.$ There are two choices ($A$ and $B$) for smoothing a given  crossing, and
thus there are $2^{c(K)}$ states of a diagram with $c(K)$ crossings.
In a state we label each smoothing with $A$ or $A^{-1}$ according to the convention
indicated by the expansion formula for the bracket. These labels are the  {\em vertex weights} of the state.
There are two evaluations related to a state. The first is the product of the vertex weights,
denoted $\langle K|S \rangle .$
The second is the number of loops in the state $S$, denoted  $||S||.$

\noindent Define the {\em state summation}, $\langle K \rangle $, by the formula

$$\langle K \rangle  \, = \sum_{S} <K|S> \delta^{||S||}$$
where $\delta = -A^{2} - A^{-2}.$
This is the state expansion of the bracket. It is possible to rewrite this expansion in other ways. For our purposes in
this paper it is more convenient to think of the loop evaluation as a sum of {\it two} loop evaluations, one giving $-A^{2}$ and one giving
$-A^{-2}.$ This can be accomplished by letting each state curve carry an extra label of $+1$ or $-1.$ We describe these {\it enhanced states} in Section \ref{subsec:EnhancedStates}.

\begin{figure}[H]
\centering
    \includegraphics[width=.5\textwidth]{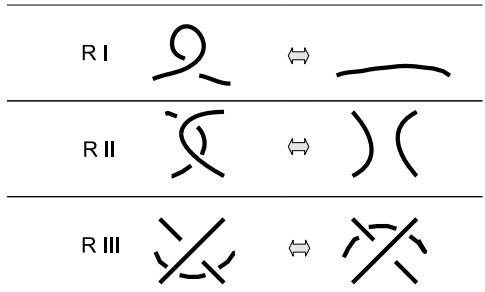}
\caption{Reidemeister moves}
\label{fig:RMs}
\end{figure}


\subsubsection{Changing Variables}

Letting $c(K)$ denote the number of crossings in the diagram $K$. If we replace $\langle K \rangle$ by $A^{-c(K)} \langle K \rangle$ and $A^2$ by $-q^{-1}$ the bracket is then rewritten in the
following form: $$\langle \Across \rangle=\langle \Asmooth \rangle-q\langle \Bsmooth \rangle $$
with $\langle \bigcirc\rangle=(q+q^{-1})$.
It is useful to use this form of the bracket state sum
for the sake of the grading in the Khovanov homology (to be described below). We shall
continue to refer to the smoothings labeled $q$ (or $A^{-1}$ in the
original bracket formulation) as {\it $B$-smoothings}.
\bigbreak

We catalog here the resulting behaviour of this modified bracket under the Reidemeister moves.
\begin{align*}
\langle \bigcirc \rangle & = q + q^{-1}\\
\langle K \, \bigcirc \rangle & =(q + q^{-1}) \langle K \rangle \\
\langle \Rcurl \rangle & =q^{-1} \langle \Arc \rangle \\
\langle \Lcurl \rangle & = -q^{2} \langle \Arc \rangle \\
\langle \raisebox{-0.25\height}{\includegraphics[width=0.5cm]{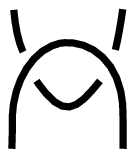}}\rangle  & =-q \langle  \raisebox{-0.50\height}{\includegraphics[width=0.5cm]{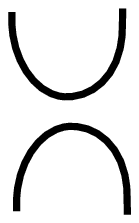}}\rangle\\
\langle \raisebox{-0.25\height}{\includegraphics[width=0.8cm]{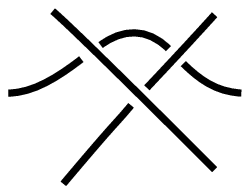}}\rangle  & =
\langle  \raisebox{-0.25\height}{\includegraphics[width=0.8cm]{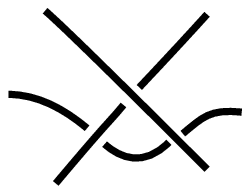}}\rangle
\end{align*}

It follows that if we define $$J_{K}(q) = (-1)^{n_{-}} q^{n_{+} - 2n_{-}} \langle K \rangle,$$
where $n_{-}$ denotes the number of negative crossings in $K$ and $n_{+}$ denotes the number
of positive crossings in $K$, then $J_{K}$ is invariant under all three Reidemeister moves.
Thus $J_{K}$ is a version of the Jones polynomial taking the value $q + q^{-1}$ on an unknotted cycle.
\bigbreak

\subsubsection{Using Enhanced States}\label{subsec:EnhancedStates}
We now use the convention of {\it enhanced
states} where an enhanced state has a label of $1$ or $-1$ on each of
its component loops. We then regard the value of the loop $q + q^{-1}$ as
the sum of the value of a cycle labeled with a $1$ (the value is
$q$) added to the value of a cycle labeled with an $-1$ (the value
is $q^{-1}).$ We could have chosen the less neutral labels of $+1$ and $X$ so that
$$q^{+1} \Longleftrightarrow +1 \Longleftrightarrow 1$$
and
$$q^{-1} \Longleftrightarrow -1 \Longleftrightarrow X,$$
since an algebra involving $1$ and $X$ naturally appears later in relation to Khovanov homology. It does no harm to take this form of labeling from the
beginning. The use of enhanced states for formulating Khovanov homology was pointed out by Oleg Viro \cite{ViroKhoHo}.
\bigbreak

Consider the form of the expansion of this version of the bracket polynomial in enhanced states. We have the formula as a sum over enhanced states $s:$
$$\langle K \rangle = \sum_{s} (-1)^{i(s)} q^{j(s)} $$
where $i(s)$ is the number of $B$-type smoothings in $s$ and $j(s) = i(s) + \lambda(s)$, with $\lambda(s)$ the number of loops  labeled $1$ minus the number of loops labeled $X$ in the enhanced state $s.$
\bigbreak

One advantage of the expression of the bracket polynomial via enhanced states is that it is now a sum of monomials. We shall make use of this property throughout the rest of the paper.

\subsection{Khovanov Homology for Classical Knots}
\subsubsection{The Jones polynomial as the Euler Characteristic of Khovanov homology}

In this section, we describe Khovanov homology along the lines of \cite{Khovanov1} \cite{DrorCat}, and we tell the story so that the gradings and the structure of the differential emerge in a natural way.
This approach to motivating the Khovanov homology uses elements of Khovanov's original approach, Viro's use of enhanced states for the bracket polynomial \cite{ViroKhoHo}, and Bar-Natan's emphasis on tangle cobordisms \cite{DrorCat}\cite{DrorCob}. The first and third named author's used similar considerations in their paper \cite{DKM} and the third author used this approach in his paper \cite{KauKhoHo}. Here we include a similar exposition as in \cite{KauKhoHo} to provide a complete picture and allow us to compare the standard approach to Khovanov homology for classical knots to that of virtual knots.  As we shall see in Section \ref{sec:KhovanovforVirtuals}, the standard approach fails for virtual knots but we will be able to recover the homology theory via an alternate method. Namely the introduction of local and global orientation and orderings, which act analogously to local coefficient systems.
\bigbreak

A key motivation in finding the Khovanov invariant is that one would like to {\it categorify} a link polynomial such as $\langle K \rangle.$ There are many meanings to the term categorify, but here the quest is to find a way to express the link polynomial as a {\it graded Euler characteristic} $\langle K \rangle = \chi_{q} \langle Kh(K) \rangle$ for some homology theory associated with $\langle K \rangle.$
\bigbreak

We will use the bracket polynomial and its enhanced states as described in the previous sections of this paper.
To see how the Khovanov grading arises, consider the form of the expansion of this version of the
bracket polynomial in enhanced states. We have the formula as a sum over enhanced states $s:$
$$\langle K \rangle = \sum_{s} (-1)^{i(s)} q^{j(s)} $$
where $i(s)$ is the number of $B$-type smoothings in $s$, $\lambda(s)$ is the number of loops in $s$ labeled $1$ minus the number of loops
labeled $X,$ and $j(s) = i(s) + \lambda(s)$.
This can be rewritten in the following form:
$$\langle K \rangle  =  \sum_{i \,,j} (-1)^{i} q^{j} dim({\mathcal C}^{ij}) $$
where we define ${\mathcal C}^{ij}$ to be the linear span (over the complex numbers for the purpose of this paper, but over the integers or the integers modulo two for other contexts) of the set of enhanced states with
$i(s) = i$ and $j(s) = j.$ Then the number of such states is the dimension of ${\mathcal C}^{ij}$ denoted:  $dim({\mathcal C}^{ij}).$
\bigbreak

\noindent We would like to have a  bigraded complex composed of the ${\mathcal C}^{ij}$ with a
differential
$$\partial:{\mathcal C}^{ij} \longrightarrow {\mathcal C}^{i+1 \, j}.$$
The differential should increase the {\it homological grading} $i$ by $1$ and preserve the
{\it quantum grading} $j.$
Then we could write
$$\langle K \rangle = \sum_{j} q^{j} \sum_{i} (-1)^{i} dim({\mathcal C}^{ij}) = \sum_{j} q^{j} \chi({\mathcal C}^{\bullet \, j}),$$
where $\chi({\mathcal C}^{\bullet \, j})$ is the Euler characteristic of the subcomplex ${\mathcal C}^{\bullet \, j}$ for a fixed value of $j.$
\bigbreak

\noindent This formula would constitute a categorification of the bracket polynomial. Below, we
shall see how {\it the original Khovanov differential $\partial$ is uniquely determined by the restriction that $j(\partial s) = j(s)$ for each enhanced state
$s$.} Since $j$ is
preserved by the differential, these subcomplexes ${\mathcal C}^{\bullet \, j}$ have their own Euler characteristics and homology. We have
$$\chi(Kh({\mathcal C}^{\bullet \, j})) = \chi({\mathcal C}^{\bullet \, j}) $$ where $Kh({\mathcal C}^{\bullet \, j})$ denotes the Khovanov homology of the complex
${\mathcal C}^{\bullet \, j}$. We can write
$$\langle K \rangle = \sum_{j} q^{j} \chi(Kh({\mathcal C}^{\bullet \, j})).$$
The last formula expresses the bracket polynomial as a {\it graded Euler characteristic} of a homology theory associated with the enhanced states
of the bracket state summation. This is the categorification of the bracket polynomial. Khovanov proves that this homology theory is an invariant
of knots and links (via the Reidemeister moves of Figure \ref{fig:RMs}), creating a new and stronger invariant than the original Jones polynomial.
\bigbreak

We will construct the differential in this complex first for mod-$2$ coefficients. What we do for mod-$2$ immediately generalizes to give the differentials over the integers by adding appropriate signs. The differential is based on regarding two states as {\it adjacent} if one differs from the other by a single smoothing at some site.
Thus if $(s,\tau)$ denotes a pair consisting in an enhanced state $s$ and site $\tau$ of that state with $\tau$ of type $A$, then we consider
all enhanced states $s'$ obtained from $s$ by smoothing at $\tau$ and relabeling only those loops that are affected by the resmoothing.
Call this set of enhanced states $S'[s,\tau].$ Then we shall define the {\it partial differential} $\partial_{\tau}(s)$ as a sum over certain elements in
$S'[s,\tau],$ and the differential by the formula $$\partial(s) = \sum_{\tau} \partial_{\tau}(s)$$ with the sum over all type $A$ sites $\tau$ in $s.$
It then remains to see what are the possibilities for $\partial_{\tau}(s)$ so that $j(s)$ is preserved.
\bigbreak

\noindent Note that if $s' \in S'[s,\tau]$, then $i(s') = i(s) + 1.$ Thus $$j(s') = i(s') + \lambda(s') = 1 + i(s) + \lambda(s').$$ From this
we conclude that $j(s) = j(s')$ if and only if $\lambda(s') = \lambda(s) - 1.$ Using this notation we have
$$\lambda(s) = [s:+] - [s:-]$$ where $[s:+]$ is the number of loops in $s$ labeled $1,$ $[s:-]$ is the number of loops
labeled $X$ and $j(s) = i(s) + \lambda(s)$.
\bigbreak

\begin{prop} The partial differentials $\partial_{\tau}(s)$ are uniquely determined by the condition that $j(s') = j(s)$ for all $s'$
involved in the action of the partial differential on the enhanced state $s.$ This unique form of the partial differential can be described by the
following structures of multiplication and co-multiplication on the algebra $V = k[X]/(X^{2})$ where
 $k = \mathbb{Z}/2\mathbb{Z}$ for mod-2 coefficients, $k = \mathbb{Z}$ for integral coefficients, or $k = \mathbb{C}$ for arbitrary coefficients.

\begin{enumerate}
\item The element $1$ is a multiplicative unit and $X^2 = 0.$
\item $\Delta(1) = 1 \otimes X + X \otimes 1$ and $\Delta(X) = X \otimes X.$
\end{enumerate}
These rules describe the local relabeling process for loops in a state. Multiplication corresponds to the case where two loops merge to a single loop,
while co-multiplication corresponds to the case where one loop bifurcates into two loops.
\end{prop}

\begin{proof}
Using the above description of the differential, suppose that
there are two loops at $\tau$ that merge in the smoothing. If both loops are labeled $1$ in $s$ then the local contribution to $\lambda(s)$ is $2.$
Let $s'$ denote a smoothing in $S[s,\tau].$ In order for the local $\lambda$ contribution to become $1$, we see that the merged loop must be labeled $1$.
Similarly if the two loops are labeled $1$ and $X,$ then the merged loop must be labeled $X$ so that the local contribution for $\lambda$ goes from
$0$ to $-1.$ Finally, if the two loops are labeled $X$ and $X,$ then there is no label available for a single loop that will give $-3,$ so we define
$\partial$ to be zero in this case. We can summarize the result by saying that there is a multiplicative structure $m$ such that
$m(1 \otimes 1) = 1, m(1\otimes X) = m(X\otimes 1) = X, m(X\otimes X) = 0,$ and this multiplication describes the structure of the partial differential when two loops merge.
Since this is the multiplicative structure of the algebra $V = k[X]/(X^{2}),$ we take this algebra as summarizing the differential.
\bigbreak

Now consider the case where $s$ has a single loop at the site $\tau.$ Smoothing produces two loops. If the single loop is labeled $X,$ then we must label
each of the two loops by $X$ in order to make $\lambda$ decrease by $1$. If the single loop is labeled $1,$ then we can label the two loops by
$X$ and $1$ in either order. In this second case we take the partial differential of $s$ to be the sum of these two labeled states. This structure
can be described by taking a coproduct structure with $\Delta(X) = X \otimes X$ and
$\Delta(1) = 1 \otimes X + X \otimes 1.$
We now have the algebra $V = k[X]/(X^{2})$ with product $m: V \otimes V \longrightarrow V$ and coproduct
$\Delta: V \longrightarrow V \otimes V,$ describing the differential completely.
\end{proof}

Partial differentials are defined on each enhanced state $s$ and a site $\tau$ of type
$A$ in that  state. We consider states obtained from the given state by  smoothing the given site $\tau$. The result of smoothing $\tau$ is to
produce a new state $s'$ with one more site of type $B$ than $s.$ Forming $s'$ from $s$ we either amalgamate two loops to a single loop at $\tau$, or
we divide a loop at $\tau$ into two distinct loops. In the case of amalgamation, the new state $s$ acquires the label on the amalgamated cycle that
is the product of the labels on the two cycles that are its ancestors in $s$. This case of the partial differential is described by the
multiplication in the algebra. If one cycle becomes two cycles, then we apply the coproduct. Thus if the cycle is labeled $X$, then the resultant
two cycles are each labeled $X$ corresponding to $\Delta(X) = X \otimes X$. If the original cycle is labeled $1$ then we take the partial boundary
to be a sum of two enhanced states with  labels $1$ and $X$ in one case, and labels $X$ and $1$ in the other case,  on the respective cycles. This
corresponds to $\Delta(1) = 1 \otimes X + X \otimes 1.$ Modulo two, the boundary of an enhanced state is the sum, over all sites of type $A$ in the
state, of the partial boundaries at these sites. It is not hard to verify directly that the square of the  boundary mapping is zero (this is the identity of mixed partials!) and that it behaves
as advertised, keeping $j(s)$ constant. There is more to say about the nature of this construction with respect to Frobenius algebras and tangle cobordisms.

\begin{figure}[h!]
     \begin{center}
     \begin{tabular}{c}
     \includegraphics[width=6cm]{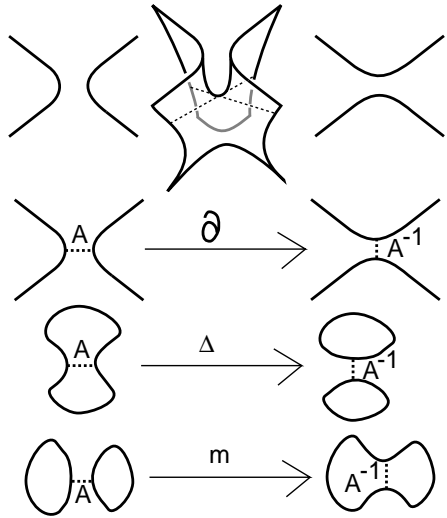}
     \end{tabular}
     \caption{Saddle points and state smoothings}
     \label{fig:SaddlePointStateSmoothing}
\end{center}
\end{figure}

\begin{figure}[h!]
     \begin{center}
     \begin{tabular}{c}
     \includegraphics[width=8cm]{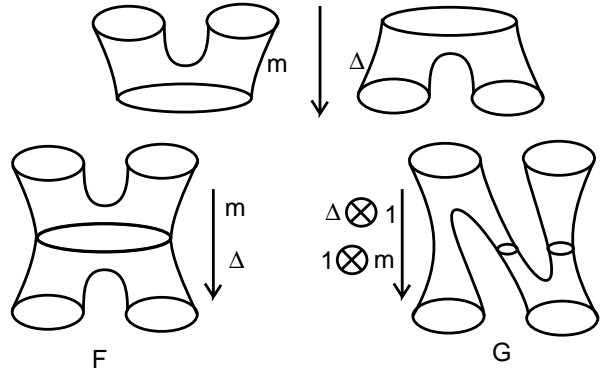}
     \end{tabular}
     \caption{Surface cobordisms}
     \label{fig:SurfaceCobordism}
\end{center}
\end{figure}


\begin{figure}[h!]
\centering
    \includegraphics[width=.7\textwidth]{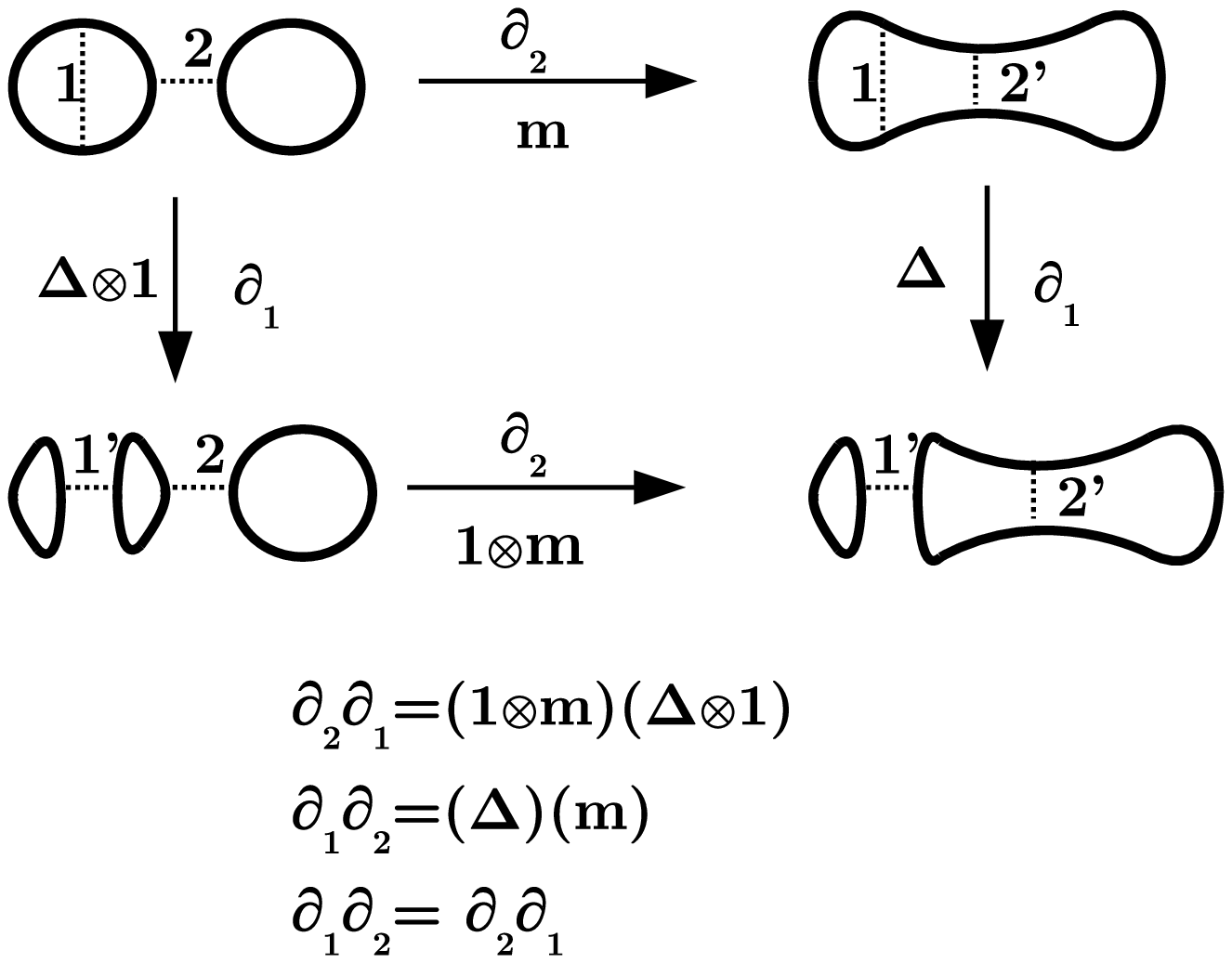}
\caption{Local boundaries commute}
\label{fig:LocalBoundariesCommute}
\end{figure}

In Figures \ref{fig:SaddlePointStateSmoothing}, \ref{fig:SurfaceCobordism}, and \ref{fig:LocalBoundariesCommute} we illustrate how the partial boundaries can be conceptualized in terms of surface cobordisms. Figure \ref{fig:SaddlePointStateSmoothing} shows how the partial boundary corresponds to a saddle point and illustrates the two cases of fusion and fission of cycles. The equality of mixed
partials corresponds to topological equivalence of the corresponding surface cobordisms, and to the relationships between Frobenius algebras \cite{Kock} and the
surface cobordism category. In particular, in Figure \ref{fig:LocalBoundariesCommute} we show how in a key case of two sites (labeled 1 and 2 in that Figure) the two orders of partial
boundary are $$\partial_{2} \partial_{1} = (1 \otimes m) \circ (\Delta \otimes 1)$$ and
$$\partial_{1} \partial_{2} = \Delta \circ m.$$ In the Frobenius algebra $V = k[X]/(X^{2})$ we have the identity
$$(1 \otimes m) \circ (\Delta \otimes 1) = \Delta \circ m.$$ Thus the Frobenius algebra implies the identity of the mixed partials.
Furthermore, in Figure \ref{fig:SurfaceCobordism} we see that this identity corresponds to the topological equivalence of cobordisms under an exchange of saddle points.
\bigbreak

\begin{figure}[h!]
\centering
    \includegraphics[width=.7\textwidth]{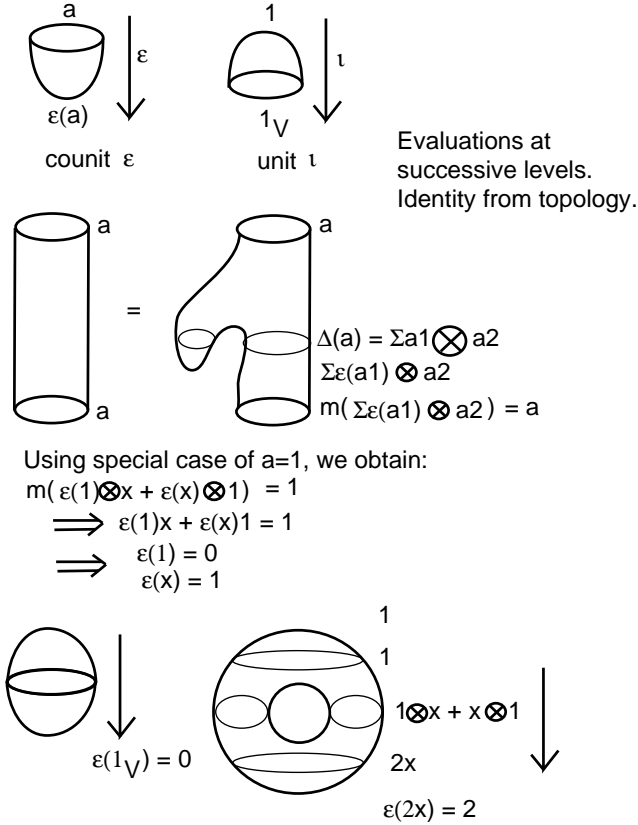}
\caption{Unit and counit as cobordisms}
\label{fig:UnitCounitCobordism}
\end{figure}


\begin{figure}[h!]
     \begin{center}
     \begin{tabular}{c}
     \includegraphics[width=8cm]{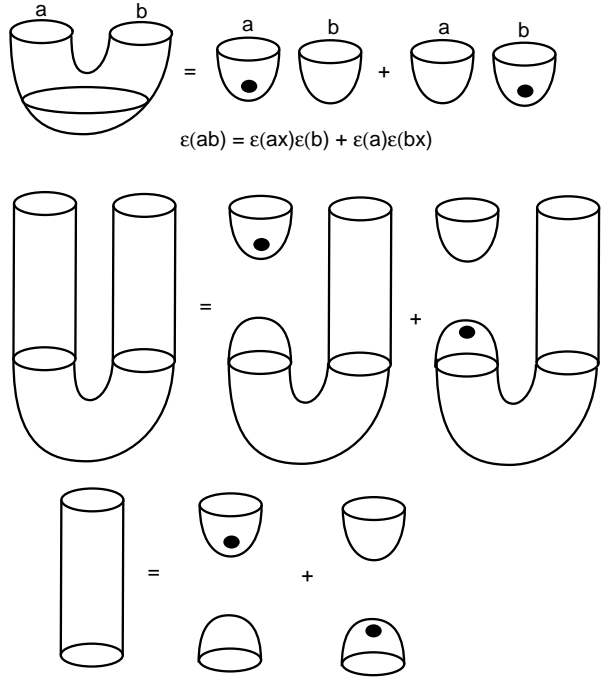}
     \end{tabular}
     \caption{The tube cutting relation}
     \label{fig:TubeCuttingRelation}
\end{center}
\end{figure}

In Figures \ref{fig:UnitCounitCobordism} and \ref{fig:TubeCuttingRelation} we show another aspect of this algebra. As Figure \ref{fig:UnitCounitCobordism} illustrates, we can consider cup (minimum)  and cap (maximum)  cobordisms that go between the empty set and a single cycle.
With the categorical arrow going down the page, the cap is a mapping from the base ring $k$ to
the module $V$ and we denote this mapping by $\iota: k \longrightarrow V$. It is the {\em unit} for the algebra $V$ and is defined by $\iota(1) = 1_{V},$ taking $1$ in $k$ to $1_{V}$ in $V.$ The cup is a mapping from $V$ to $k$ and is denoted by $\epsilon: V \longrightarrow k.$ This is the {\em counit}.
As Figure \ref{fig:UnitCounitCobordism} illustrates, we need a  basic identity about the counit which reads
$$\Sigma \epsilon(a_{1})a_{2} = a$$ for any $a \in V$ where
$$\Delta(a) =  \Sigma a_{1} \otimes a_{2}.$$

We will often supress the summation symbol and write $\Delta(a) = a_{1} \otimes a_{2}.$ The summation is over an appropriate set of elements in $V \otimes V$ as in our specific formulas for the algebra $k[X]/(X^{2}).$ Of course we also demand
$$\Sigma a_{1}\epsilon(a_{2}) = a$$ for any $a \in V.$ With these formulas about the counit and unit in
place, we see that cobordisms will give equivalent algebra when one cancels a maximum or a minimum with a saddle point, again as shown in Figure \ref{fig:UnitCounitCobordism}.
\bigbreak

Note that for our algebra $V = k[X]/(X^{2}),$ it follows from the counit identities of the last paragraph that
$$\epsilon(1) = 0$$ and $$\epsilon(X) = 1.$$ In fact, Figure \ref{fig:TubeCuttingRelation} shows a formula that holds in this special algebra. The formula reads $$\epsilon(ab) = \epsilon(aX)\epsilon(b) + \epsilon(a)\epsilon(bX)$$
for any $a,b \in V.$ As the rest of Figure \ref{fig:TubeCuttingRelation} shows, this identity means that a single tube in any cobordism can be cut, replacing it by a cups and a caps in a linear combination of two terms. The tube-cutting
relation is shown in its most useful form at the bottom of Figure \ref{fig:TubeCuttingRelation}. In Figure \ref{fig:TubeCuttingRelation}, the black dots are symbols standing for the special element $X$ in the algebra.
\bigbreak

It is important to note that we have a nonsingular pairing $$\langle ~~|~~ \rangle: V \otimes V \longrightarrow k$$ defined by the equation $$\langle a|b \rangle = \epsilon(ab).$$ One can define a
Frobenius algebra by starting with the existence of a nonsingular bilinear pairing.  In fact, a finite dimensional associative algebra with unit defined over a unital commutative ring  $k$ is said to be
a {\em Frobenius algebra} if it is equipped with a non-degenerate bilinear form
$$\langle ~~|~~ \rangle: V \otimes V \longrightarrow k$$ such that
$$\langle ab | c \rangle = \langle a | bc \rangle$$ for all $a,b,c$ in the algebra. The other mappings and
the interpretation in terms of cobordisms can all be constructed from this definition. See \cite{Kock}.
\bigbreak
\begin{rem}
Now that the mixed partials commute, they will continue to commute over the integers. By adding signs we can make the squares anti-commute.
\end{rem}

\subsubsection{Remark on Grading and Invariance} In Section \ref{subsec:BracketJonesPoly} we showed how the
bracket, using the variable $q$, behaves under Reidemeister moves. These formulas correspond to how
the invariance of the homology works in relation to the moves. We have that
$$J_{K}(q) = (-1)^{n_{-}} q^{n_{+} - 2n_{-}} \langle K \rangle,$$
where $n_{-}$ denotes the number of negative crossings in $K$ and $n_{+}$ denotes the number
of positive crossings in $K.$ $J(K)$ is invariant under all three Reidemeister moves. The corresponding formulas for Khovanov homology are as follows
\begin{align*}
J_{K}(q) &= (-1)^{n_{-}} q^{n_{+} - 2n_{-}} \langle K \rangle\\
& = (-1)^{n_{-}} q^{n_{+} - 2n_{-}} \sum_{i,j} (-1)^{i}q^{j} dim(Kh^{i,j}(K))\\
& = \sum_{i,j} (-1) ^{i + n_{-} }q^{j + n_{+} - 2n_{-} }dim(Kh^{i,j}(K))\\
& = \sum_{i,j} (-1)^{i} q^{j} dim(Kh^{i - n_{-}, j - n_{+} + 2n_{-}}(K)).
\end{align*}
It is often more convenient to define the {\em Poincar\'e polynomial} for Khovanov homology via
\[P_{K}(t, q)
= \sum_{i,j} t^{i} q^{j} dim(Kh^{i - n_{-}, j - n_{+} + 2n_{-}}(K)).\]
The Poincar\'e polynomial is a two-variable polynomial invariant of knots and links, generalizing the Jones polynomial. Each coefficient $$ dim(Kh^{i - n_{-}, j - n_{+} + 2n_{-}}(K))$$ is an
invariant of the knot, invariant under all three Reidemeister moves. In fact, the homology groups
$$Kh^{i - n_{-}, j - n_{+} + 2n_{-}}(K)$$ are knot invariants. The grading compensations show how the grading of the homology can change from diagram to diagram for diagrams that represent the same knot. For notational convenience this is often denoted by $Kh(K) = \llbracket K \rrbracket [-n_{-}]\{ n_{+} - 2n_{-}\}$ \cite{DrorCat}.
\bigbreak

\subsection{Virtual Knot Theory}

Kauffman \cite{VKT} introduced virtual knots and links as a natural extension of classical knot theory.  The following formulations of virtual knot theory can be though of as equivalent:
\begin{enumerate}
\item Equivalence classes of embeddings of closed curves in a thickened surface (possibly non-orientable) up to ambient isotopy and handle stabilization on the surface.
\item Signed oriented Gauss codes taken up to the equivalence relations generated by the abstract Reidemeister moves on such codes.
\end{enumerate}

Diagrammatically we can represent virtual knots as projections of embeddings of closed curves in a thickened surface modulo the Reidemeister moves in Figure \ref{fig:RMs} and virtual Reidemeister moves in Figure \ref{fig:VRMs}.

\begin{figure}[h!]
\centering
    \includegraphics[width=.5\textwidth]{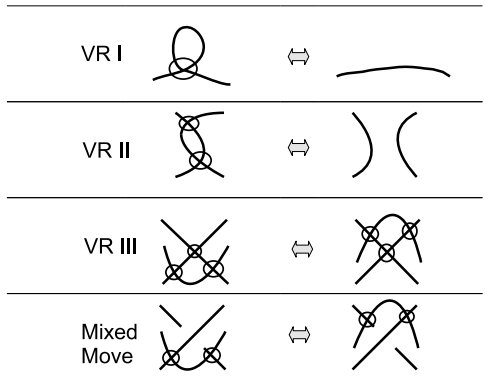}
\caption{Virtual Reidemeister moves}
\label{fig:VRMs}
\end{figure}

\section{Khovanov Homology for Virtual Knots}\label{sec:KhovanovforVirtuals}
\subsection{Single Cycle Smoothings}

Extending Khovanov homology to virtual knots for arbitrary coefficients is complicated by the single cycle smoothing as depicted in Figure \ref{fig:oneonebifurcation}.

\begin{figure}[h!]
\begin{center}
    \includegraphics[width=.4\textwidth]{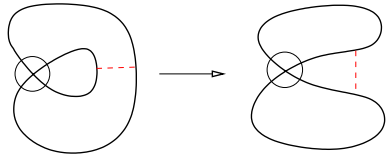}
    \caption{Single cycle smoothing}
\label{fig:oneonebifurcation}
\end{center}
\end{figure}

We define a map for this smoothing $\eta:V \longrightarrow V$. In order to preserve the quantum grading ($j(s') = j(s)$)  $\eta$ is the zero map.

\begin{rem}
Note that algebraic system for Khovanov homology for virtual knots is, strictly speaking, no longer a Frobenius system due to presence of the single-cycle smoothing. Rather it is a Frobenius system plus an $\eta$ map satisfying the property that the resulting homology theory for virtual knots is well-defined. Here we consider only the homology theory related to $\eta(a)=0$ for all $a$ and thus will still call such a system a Frobenius system and the associated algebra a Frobenious algebra. We will be discussing the general case in an upcoming paper.
\end{rem}

Consider the following complex arising from the 2-crossing virtual unknot:

\begin{figure}[h!]
\centering
    \includegraphics[width=.4\textwidth]{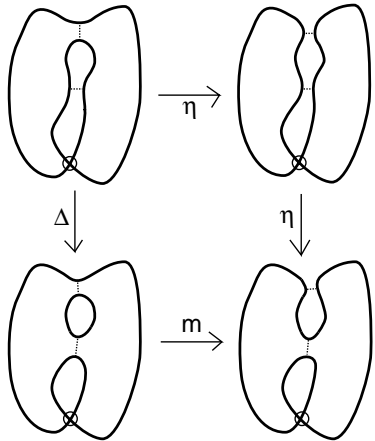}
    \caption{Khovanov complex for the two-crossing virtual unknot}
\label{fig:problemsquare}
\end{figure}

Composing along the top and right we have $$\eta \circ \eta  = 0.$$

But composing along the opposite sides we see $$m \circ \Delta(1) = m(1 \otimes X + X \otimes 1) = X + X = 2X.$$ Hence the complex does not naturally commute or anti-commute.


When the base ring is $\mathbb{Z}_2$ Manturov \cite{ManturovKhovanovZ2} \cite{ManturovVirtualKnots} has shown that the definition of Khovanov homology given in the previous section goes through unchanged. Using this construction, Manturov also showed that one can always take the double of the virtual knot or link to get a definition of Khovanov homology \cite{ManturovKhovanovZ2}. Additionally,  Manturov \cite{ArbitraryCoeffs} introduced a definition of Khovanov homology for (oriented) virtual knots with arbitrary coefficients that does not require doubling. The following is a reformulation of the latter definition using cut loci. In the remainder of this section we provide a detailed explanation of this construction.

\subsection{Source-Sink Orientations and Cut Loci}

Given an oriented virtual knot diagram, assign a \emph {source-sink orientation} at each classical crossing, where a source-sink orientation is defined by the diagram in Figure \ref{fig:sourcesink}. For more on source-sink structures and virtual knots see Maturov \cite{ManturovVassilievConj} \cite{ManturovVirtualKnots}.

\begin{figure}[H]
\centering
    \includegraphics[width=.15\textwidth]{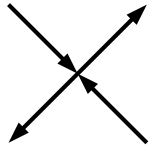}
    \caption{Source-sink orientation}
\label{fig:sourcesink}
\end{figure}

We will use the {\emph canonical source-sink orientation} given in Figure \ref{fig:canonicalsourcesink}.

\begin{figure}[H]
\centering
    \includegraphics[width=.4\textwidth]{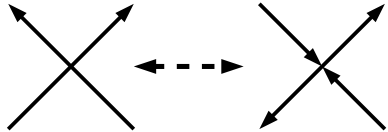}
    \caption{Canonial source-sink orientation}
\label{fig:canonicalsourcesink}
\end{figure}

Using the canonical source-sink orientations we place cut loci on the semi-arcs of the knot diagram whenever neighboring source-sink orientations disagree. For classical knots, the canonical source-sink orientations will never produce cut loci. The proof of this is straightforward when the knot is viewed as a chord diagram.  However for virtual knots, cut loci arise quite naturally.

Figure \ref{fig:examplecutloci} illustrates this process for a two-crossing virtual knot. Note that the number of cut loci is fixed for any particular diagram and the canonical source-sink orientations on the crossings. However, as we will see, the number of cut loci may change when a Reidemeister move is performed.

\begin{figure}[h!]
\centering
    \includegraphics[width=.35\textwidth]{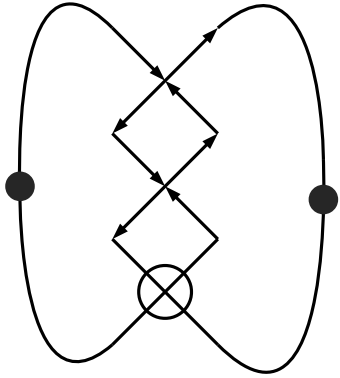}
    \caption{Cut loci for a two-crossing virtual knot}
\label{fig:examplecutloci}
\end{figure}

Source-sink orientations were originally introduced by Naoko Kamada \cite{MR1914297} \cite{MR2100692} and determine a version of checkerboard coloring for virtual knots. We note that the source-sink orientations we use can be regarded as a translation of the oriented chord diagrams of Oleg Viro \cite{ViroKhoHo}.

The source-sink orientations define a \emph{local orientation} on the semi-arcs of each enhanced state in the Khovanov complex. Furthermore, for each enhanced state we select a \emph{global orientation} by (arbitrarily) choosing a marked point, not a cut loci, on the cycle. When necessary we will keep track of the global orientations by placing a ($\star$) on the arcs and allowing the global order to agree with the local order on those arcs. For instance see Figures \ref{fig:orientedproblemsquare} and \ref{fig:orientedorderedproblemsquare}. The proof that this can be done arbitrarily follows from Lemma \ref{lem:cutlocicancel}. It should be noted that this is similar in spirit to reducing matrix factorizations in Khovanov-Rozansky $sl_2$ homology. For more about matrix factorizations and Khovanov-Rozansky $sl_2$ homology see Hughes \cite{Hughessl2}.

When the local and global orientations disagree we introduce a new operation that we call the \emph{bar operation}. The bar operation is an operation on the algebra, not an isomorphism of the algebra,  which acts similarly to complex conjugation and is defined by:
$$ \overline{X} = -X \; , \; \overline{1} = 1. $$

Hence, in analogy to complex conjugation,

$$\overline{a+bX} = a - bX.$$

Furthermore this operation respects multiplication,

$$ \overline{(a + bX)(c+dX)}= (\overline{a + bX})  (\overline{c+dX}).  $$

Considering comultiplication, we recall that $\Delta$ respects scalar multiplication but is not an algebra morphism. Considering the standard basis we see that
\begin{align*}
\Delta ( \overline{1}) & = \Delta (1)
&  \Delta ( \overline{X}) &=  - \Delta (X).
\end{align*}

Furthermore, suppressing summation signs, if $\Delta (a) = a_1 \otimes a_2$ and we denote $\Delta (\overline{a})$ by $\overline{a_1 \otimes a_2}$ then for all choices of $a_1$ and $a_2$

 \begin{align*}
\Delta (\overline{a})=\overline{a_1 \otimes a_2} = -(\overline{a_1} \otimes \overline{a_2}).
 \end{align*}

 For example,
  \begin{align*}
\Delta (\overline{1})=1 \otimes X + X \otimes 1 = -(1 \otimes -X + -X \otimes 1 )= -(\overline{1} \otimes \overline{X} + \overline{X} \otimes \overline{1})
 \end{align*}
 and
  \begin{align*}
\Delta (\overline{X})=-\Delta (X)= -(X \otimes X)= -(\overline{X} \otimes \overline{X}).
 \end{align*}

For each differential $\partial_{\tau}(s)$ we pre- and post-compose the state with the barring operation.
This allows us to take into account sites that disagree with the global orientation.

The number of times the barring operation is performed equals the number of cut loci crossed when traversing a path starting at the smoothing corresponding to the partial differential and ending at the marked point corresponding the global orientation.

\begin{figure}[h!]
\centering \def\svgwidth{3.00 in}
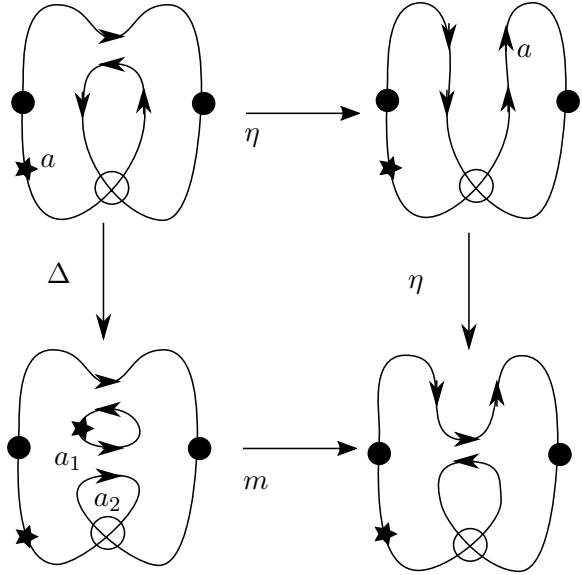
    \caption{Oriented Khovanov complex for the two-crossing unknot}
\label{fig:orientedproblemsquare}
\end{figure}

Now recall our problem square. In Figure \ref{fig:orientedproblemsquare} the global orientation of each component is marked by a star. Traversing a cut loci results in the application of the barring operation. Each component is marked with a global algebra element (for example, $ a $). We compare the global orientation with the local orientation inherited from the source-sink orientation at each smoothing to determine when to apply the barring operation. Following along the top and right hand sides of the diagram we have that $\eta \circ \eta (a) =0$ for all choices of $a$. (Note the choice of global orientation is not important along this composition.) Following along the left and bottom we need to show that here we have $m \circ \Delta (a) = 0$. Looking at $\Delta$ we see the lower smoothing site in the upper left state has local orientation which agrees with the global orientation. Similarly, in the lower left state of the diagram the local orientations agree with the global orientations on both $a_1$ and $a_2$.  Hence in the $\Delta$ map no compositions with the barring operator are performed. Looking at $m$ in the lower left state of the diagram we note that the local orientation agrees with the global orientation on $a_1$ and but disagrees on $a_2$ (i.e. we must transport $a_2$ past one cut loci in order for it to be located at the smoothing site corresponding to $m$).  Hence we pre-compose $m$ with the barring operation on the second element (assuming a lexicographical ordering on the elements in that state). Finally, in the lower right hand state, to get from the smoothing site at $m$ to the star representing the global orientation we must cross a cut loci. Thus we must also post-compose $m$ with the barring operation.

More precisely $$m \circ \Delta(1) = m(1 \otimes \overline{X} + X \otimes \overline{1}) = -\overline{X} + \overline{X} = 0.$$

\begin{rem}
We remind the reader that this result is independent of the placement of stars in Figure \ref{fig:orientedproblemsquare}. Had the stars been placed at a different location the only change would be the pre- and post-composition with the bar operation. These changes will occur in pair corresponding to incoming and outgoing differentials. Hence the final result will remain unchanged.
\end{rem}



Recall that it would be enough to show that all faces commute. In fact this is the standard approach taken with Khovanov homology for classical knots. In the classical knot case one can then sprinkle signs appropriately through out the cube complex to produce a well defined homology theory. In Figure \ref{fig:problemsquare} we have shown that this approach is not possible for virtual knots due to the presence of the single-cycle smoothing.

\begin{rem} Guided by the observation mentioned in Figure \ref{fig:orientedproblemsquare} we no longer want to show that faces commute, but rather we want to show directly that all faces anti-commute (i.e. $\partial^2 = 0$), negating the need for the sprinkling of signs. We proceed by attempting to show that all faces of the complex anti-commute through the introduction of the bar operation. This  procedure will fail, and hence our use of the term ``attempt'' in the upcoming lemma. It is insightful to see how and why this procedure fails in order to motivate the additional structure introduced in the following section, namely local and global bases.
\end{rem}

The set of all two crossing diagrams (without over-under markings) is shown in Figure \ref{fig:essentialknots}.

\begin{figure}[h!]
\centering
\def\svgwidth{2.75 in} 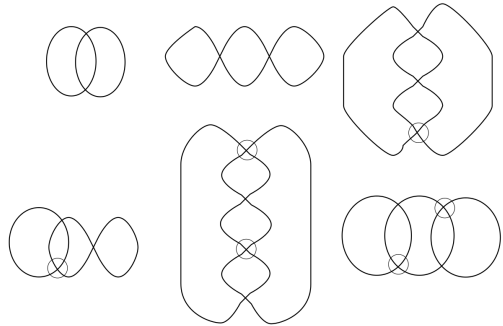
    \caption{Knot diagrams with two flat crossings}
\label{fig:essentialknots}
\end{figure}

 First, consider the oriented versions of these two crossing diagrams. We may reduce the number of cases we need to check by noting that many of the two crossing complexes commute due to the appearance of $\eta$ maps along both paths in the complex. Only the oriented versions of the 6 essential atoms in Figure \ref{fig:essentialatoms} result in squares where the composition of maps in the complexes do not commute due to the presence of these zero maps. (For more on atoms see Manturov \cite{ManturovKnotTheory}.) We have already shown one case in Figure \ref{fig:orientedproblemsquare}. It is left as an exercise to check the remaining five cases.

\begin{figure}[h!]
\centering
\def\svgwidth{3.00 in} 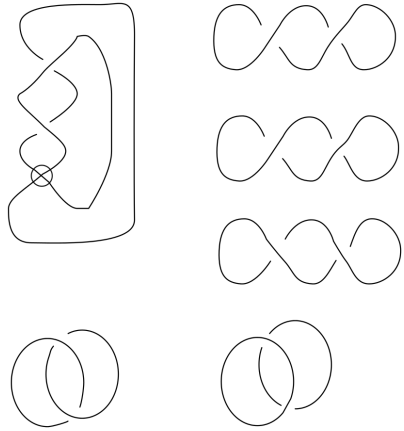
    \caption{The 6 essential atoms}
\label{fig:essentialatoms}
\end{figure}

Now recall that not all of the faces on a Khovanov complex arise from oriented two crossing diagrams. Many contra-oriented faces are obtained through smoothing in a larger complex as, for example, in Figure \ref{fig:contraface}.

\begin{figure}[h!]
\centering
    \includegraphics[width=.35\textwidth]{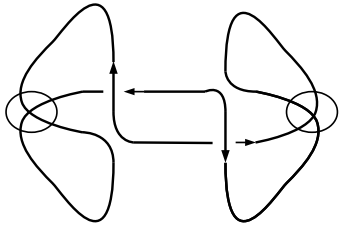}
\caption{Contra-oriented face arising from smoothing a Kishino's knot}
\label{fig:contraface}
\end{figure}

Rather than tackling the remaining cases individually we consider the action of a $90 \degree$ rotation at a crossing as in Figure \ref{fig:clockwiserotation}.

\begin{figure}[H] \centering
\[ \begin{array}{c}
 \def\svgwidth{0.50 in}
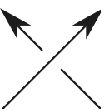 \end{array}
   \rightarrow
   \begin{array}{c} \def\svgwidth{0.50 in} 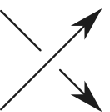 \end{array} \]
    \caption{$90 \degree $  rotation of a crossing}
\label{fig:clockwiserotation}
\end{figure}

Here we attempt to show that a $90 \degree$ rotation at a crossing does not effect the commutativity of the face of the Khovanov complex. With the current structure this method will fail. However it is informative to see why this fails. This attempt at a proof will motivate the additional structure introduced in the following section.

\begin{lem*}[$90 \degree$ Lemma]\label{lem:90firsttry}

A $90 \degree$ rotation on the orientation of a crossing does not change the anti-commutativity of a face of a Khovanov complex. Furthermore every face of a Khovanov complex is anti-commutative.
\end{lem*}

\noindent{ \textbf{Attempt of Proof:} (The full proof is given in Section \ref{subsec:LGOrder})}

Consider the action of a $90\degree$ rotation on multiplication at a positive crossing. In Figure \ref{fig:RotateMultPositiveUnordered} the top row corresponds to the pre-rotated complex and the lower row corresponds to the complex obtained by a $90 \degree$ rotation. We would like construct an isomorphism sending the complex corresponding to the upper row to the complex in the lower row. Thus since the top complex anti-commutes we have that the lower complex also anti-commutes. Assume the global orientations in both rows agree with the local orientations in the pre-rotated complex so that the maps $i_1$ and $i_2$ are induced by the change in local orientations.  Finally let $a$ and $b$ denote arbitrary algebra elements.

\begin{figure}[H]
\centering
    \includegraphics[width=.7\textwidth]{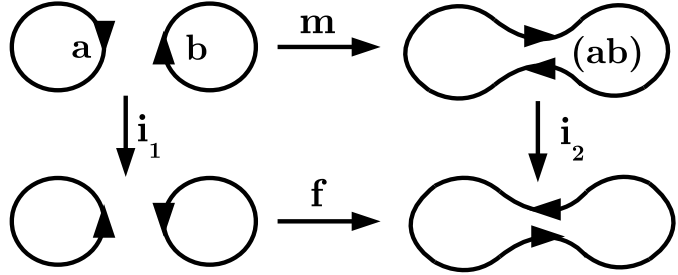}
    \caption{$90^\circ$ rotation on multiplication}
\label{fig:RotateMultPositiveUnordered}
\end{figure}

Here  $i_1(a \otimes b) = (\overline{a} \otimes \overline{b})$ and $i_2 (ab) = \overline{ab}$. To determine the map denoted by $f$ (i.e. the image under the isomorphism of $m$) we note that $m(\overline{a} \otimes \overline{b}) = (\overline{a}\overline{b})$. Since $(\overline{a}\overline{b}) = \overline{ab}$ for all $a$ and $b$ the isomorphism sends $m$ in the upper complex to $m$ in the lower complex.

Similarly, consider the action of a $90\degree$ rotation on co-multiplication at a positive crossing as in Figure \ref{fig:RotateComultPositiveUnordered}.

\begin{figure}[H]
\centering
    \includegraphics[width=.7\textwidth]{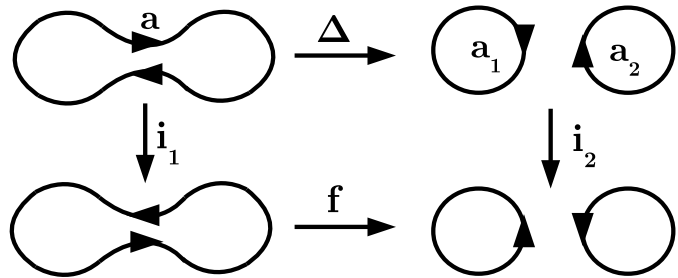}
    \caption{$90^\circ$ rotation on co-multiplication}
\label{fig:RotateComultPositiveUnordered}
\end{figure}

Here $i_1(a)=\overline{a}$ and $i_2(a_1 \otimes a_2) = \overline{a_1} \otimes \overline{a_2}$. To determine the map denoted by $f$ (i.e. the image under the isomorphism of $\Delta$) we recall an earlier observation that $\Delta(\overline{a}) = \overline{a_1 \otimes a_2} = -(\overline{a_1} \otimes \overline{a_2})$. Hence the isomorphism sends $\Delta$ in the upper complex to $-\Delta$ in the lower complex.

Thus if we only use orientations we see that a $90\degree$ rotation sends some anti-commutative faces to anti-commutative faces while sending other anti-commutative faces to commutative faces. While one may be able to find additional methods  allowing for a mix of commutative and anti-commutative faces (the second author thanks Adam Lowrance for pointing out such a method in \cite{OddKhoHo}) we prefer to add a bit more structure to ensure uniformity. We restate and prove this lemma in as Lemma \ref{lem:anticommutativity} in the following section.

\subsection{Local and Global Order}\label{subsec:LGOrder}

%

We introduce a local order on the cycles for each smoothing determined by the labels 1 and 2 in Figure \ref{fig:CanonicalOrientation}. As with local and global orientation we will choose an global order for the cycles in each state in the Khovanov complex and pre- and post-compose with maps after a local-to-global comparison. In doing so, we are effectively moving from the tensor product to a Grassmann algebra.

\begin{figure}[H]
\centering
    \includegraphics[width=.25\textwidth]{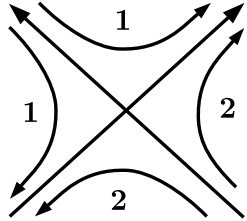}
    \caption{Canonical local orientation and order}
\label{fig:CanonicalOrientation}
\end{figure}

More precisely, the \emph{global order} is an arbitrary choice of labels ${1, 2, \ldots, n}$ for the $n$ cycles in each state of the Khovanov complex. To define the \emph{local order} associated to each smoothing in a state we have two situations to consider.

\begin{itemize}
  \item If the labels 1 and 2 in Figure \ref{fig:CanonicalOrientation} appear on the same cycle we permute the global order so that this cycle corresponds to the first element in the local order and set the the remaining elements of the local order to be in the same relative order as they appeared in the global order.

    \item If the labels 1 and 2 appear on different cycles we permute the global order so that the first and second elements in the local order are those labeled 1 and 2 in Figure \ref{fig:CanonicalOrientation} and set the remaining elements of the local order to be  in the same relative order as they appeared in the global order.
\end{itemize}

We can then, in addition to the pre- and post-compositions induced by the local and global orientations (i.e. the pre- and post- composition by the barring operation resulting from the transport of the algebraic elements across the cut loci), pre- and post-compose the maps $m$ and $\Delta$ with multiplication by the sign of the permutation taking the local order to the global order. (For a global order which is convenient for computation see Appendix \ref{App:AppendixA}.) It should be noted that this could alternately be formalized as a Grassmann algebra on the cycles in each state as was done by Manturov \cite{ArbitraryCoeffs}.

\begin{thm}\label{thm:Anticommute}
Using the pre- and post-compositions induced by orientation and order, all faces of the Khovanov complex for a virtual link diagram \bf{anti-}commute.
\end{thm}

Before we prove Theorem \ref{thm:Anticommute} we revisit and prove the $90 \degree$ Lemma from the previous section.

\begin{lem}[$90 \degree$ Lemma]\label{lem:anticommutativity}

A $90 \degree$ rotation on the orientation of a crossing does not change the anti-commutativity of a face of a Khovanov complex. Furthermore every face of a Khovanov complex anti-commutes.
\end{lem}

\begin{proof}

As before in Figure \ref{fig:RotateMultPositive} the top row corresponds to the pre-rotated complex and the lower row corresponds to the complex obtained by a $90 \degree$ rotation. We construct an isomorphism between the complex corresponding to the top row and the complex corresponding to the bottom row. Assume the global orientations and orders in both rows agree with the local orientations and orders in the pre-rotated complex so that the maps $i_1$ and $i_2$ are induced by the change in local orientations and orders.  Finally let $a$ and $b$ denote arbitrary algebra elements.

\begin{figure}[H]
\centering
    \includegraphics[width=.7\textwidth]{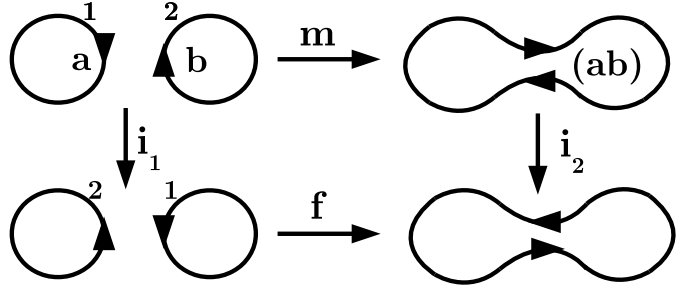}
    \caption{$90^\circ$ rotation on multiplication}
\label{fig:RotateMultPositive}
\end{figure}

Then we define $i_1(a \otimes b) = -(\overline{a} \otimes \overline{b})$ and $i_2 (ab) = \overline{ab}$, where the negative sign in $i_1$ appears due to the change in order. To determine the map denoted by $f$ (i.e. the image under the isomorphism of $m$)  we note that $-m(\overline{a} \otimes \overline{b}) = m(-(\overline{a} \otimes \overline{b})) = -(\overline{a}\overline{b}) = -\overline{ab}$ for all $a$ and $b$. Thus the isomorphism sends $m$ in the upper complex to $-m$ in the lower complex.

Similarly, consider the action of a $90\degree$ rotation on co-multiplication at a positive crossing as in Figure \ref{fig:RotateComultPositive}.

\begin{figure}[H]
\centering
    \includegraphics[width=.7\textwidth]{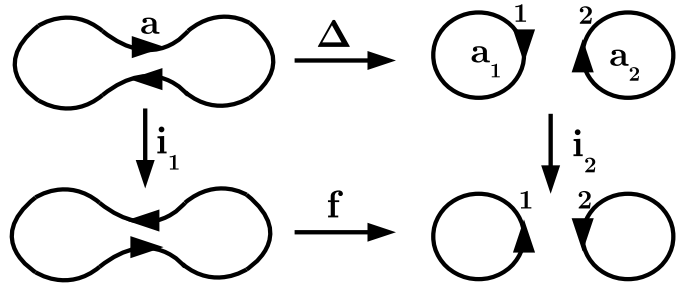}
    \caption{$90^\circ$ rotation on co-multiplication}
\label{fig:RotateComultPositive}
\end{figure}

We note that the local orders on the right hand side do not induce any sign changes and thus the algebra from the previous attempt remains unchanged and the isomorphism sends $\Delta$ in the upper complex to $-\Delta$ in the lower complex.

The cases for the negative crossing can be handled in a similar fashion. (Or one can reverse the horizontal arrows in the above proofs for positive crossings and consider $90 \degree$ counter-clockwise rotations on negative crossings.)
\end{proof}

\begin{figure}[h!]
\centering \def\svgwidth{3.00 in}
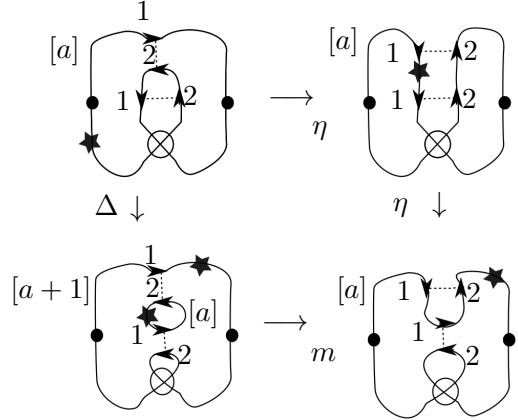
    \caption{Oriented, ordered Khovanov complex for the two-crossing unknot}
\label{fig:orientedorderedproblemsquare}
\end{figure}

\begin{proof}[Proof of Theorem \ref{thm:Anticommute}]

First consider the oriented versions of the diagrams in Figure \ref{fig:essentialatoms}. It is an exercise to show that all of the corresponding Khovanov complexes anti-commute.  Note that the oriented, ordered diagram for the 2 crossing unknot is given in Figure \ref{fig:orientedorderedproblemsquare}. In this figure we use brackets to denote the global order so as to distinguish it from the local orders used in previous diagrams. If the global order in the lower left state is $[a] \otimes [a+1]$ we see the only sign change occurs as a pre-composition in the multiplication map and hence,
$$m \circ \Delta(1) = m(-(1 \otimes \overline{X} + X \otimes \overline{1})) = X - X = 0.$$

Finally by applying multiple applications of the $90 \degree$ Lemma we see the contra-oriented faces must also anti-commute.
\end{proof}

Finally one must show that this construction is invariant under an oriented set of Reidemeister moves. While care must be taken due to the additional maps produced by the source-sink orientation and ordering of cycles, the general details are identical to the classical proofs. We refer the reader to \cite{DrorCat} \cite{DrorCob} \cite{ArbitraryCoeffs}.

\begin{rem}
As one might expect there are issues of sign to consider when one examines the functoriality of this construction.
To define the Rasmussen invariant one only needs functoriality up to multiplication by an element of $\mathbb{Q}$. For both this reason and to simplify the exposition we choose not to move in that direction. However, one can get precise functoriality by applying constructions analogous to \cite{CMW} or \cite{CaprauWebsFoams} on top of the previous structure.
\end{rem}

\subsection{Non-classicality of Virtual Knots with Unit Jones Polynomial}

There is a construction that produces infinitely many non-trivial virtual knot diagrams that have unit Jones polynomial. We shall prove here, that all of these examples of knots with unit Jones polynomial are non-classical. This settles a question raised in \cite{VKT} and \cite{VKTI}.  To clarify our proof we first restate a theorem of Manturov \cite{ArbitraryCoeffs} regarding the invariance of Khovanov homology for virtual knots with arbitrary coefficients under Z-equivalence.

\begin{thm} Let $K$, $K'$ be two knot diagrams obtained one from another by Z-equivalence as depicted in Figure \ref{fig:ZEquiv}. Then there is a grading-preserving chain isomorphism $C(K) \rightarrow C(K')$ that agrees with the local differentials. In particular, if $C(K)$ is a well-defined complex, then so is $C(K')$ and their homology groups are isomorphic.
\end{thm}

\begin{proof} The proof is an explicit construction of the isomorphism. We refer the reader to Lemma 1 of Manturov \cite{ArbitraryCoeffs}.
\end{proof}

\begin{figure}
     \begin{center}
     \begin{tabular}{c}
     \includegraphics[width=8cm]{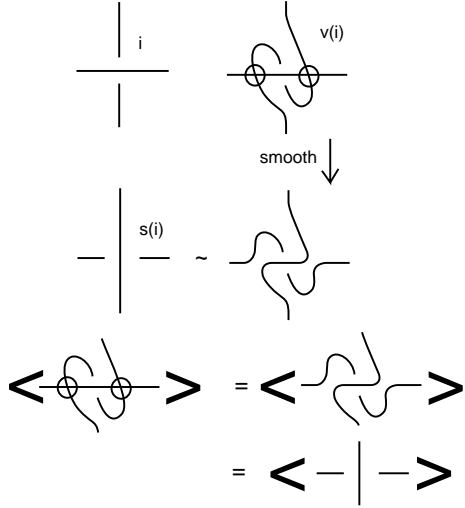}
     \end{tabular}
     \caption{\bf Switch and Virtualize}
     \label{f7}
\end{center}
\end{figure}

\begin{figure}[h!]
\centering
    \includegraphics[height=.6in]{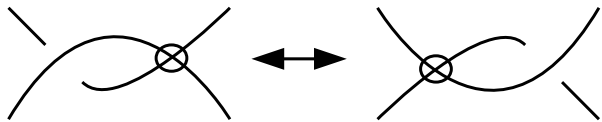}
    \caption{Z-equivalence}
\label{fig:ZEquiv}
\end{figure}

We say that a classical crossing has been {\it virtualized} if it is replaced by a crossing with the opposite orientation that is flanked by two virtual crossings. See the illustrations for this construction in \cite{VKT} and \cite{VKTI} and see Figure~\ref{f7} in this paper.  
In the Gauss diagram, {\it a crossing is virtualized by reversing its sign and leaving its arrow unchanged.} The reader should note this very specific convention for the term virtualization.  One checks, as in Figure~\ref{f7},  that the Jones polynomial (via the bracket model \cite{KaB}) of a knot with a virtualized crossing is the same as the Jones polynomial of that same knot with the crossing {\it switched} (a switch interchanges over and under-crossing lines at the site of the crossing.). Thus, given a classical knot diagram $K,$ one can choose a subset $S$ of the crossings so that switching  all of them gives an unknot diagram $U = S(K)$ where $S(K)$ denotes the diagram that results from the diagram $K$ by switching all the crossings in the subset $S.$ Instead of switching, we virtualize all the crossings in $S$ to form a virtual knot diagram $Virt(K).$ It then follows that $Virt(K)$ has unit Jones polynomial, and is a non-trivial knot due to the fact that its un-oriented Gauss code has not been changed (See \cite{VKT} for the proof of non-triviality).

\begin{thm}If $K$ is a non-trivial classical knot diagram and $Virt(K)$ is the virtual diagram described above with unit Jones polynomial, then $Virt(K)$ is a non-trivial and non-classical virtual diagram.
\end{thm}

\begin{proof}The proof follows a suggestion of Adam Lowrance to examine the virtual Khovanov homology of $Virt(K)$. We know that $Virt(K)$ is a non-trivial knot. We also know, using the description of virtual integral Khovanov homology given above, that the Khovanov homology of $Virt(K)$ is isomorphic to the Khovanov homology of $S(K) = U$ where $S(K)$ is the classical unknot obtained by switching crossings in $K$ from the subset $S$ of crossings  used to virtualize $K$ to $Virt(K).$ That is, Khovanov homology of a knot with a virtualized crossing is the same as the Khovanov homology of the corresponding knot with that crossing switched. {\it Thus $Virt(K)$ has the Khovanov homology of the unknot.} If $Virt(K)$ were equivalent to a classical knot $K'$, then $K'$ would have the Khovanov homology of the unknot. Since by Kronheimer and Mrowka \cite{UnknotDetector}, Khovanov homology detects the classical unknot, it would follow that $K'$  and hence $Virt(K)$, is unknotted. Since we know that $Virt(K)$ is not trivial, we know that $Virt(K)$ is not classical.
\end{proof}

This argument shows that all examples of virtual diagrams with unit Jones polynomial,  obtained by the virtualization construction are non - classical. There are diagrams to which the argument in the Theorem does not apply. These are virtual diagrams with unit Jones polynomial that are not $Z$-equivalent to the unknot. See \cite{VKTI} for such an example and for references to previous constructions of this type due to V. O. Manturov.   Previous attempts to prove such results involve using other invariants of virtual knots and have gone on a case-by-case basis. In \cite{DKMinimalSurface} \cite{SilverWilliamsUnitJones} it was previously shown that $Virt(K)$ is non-classical for classical knots $K$ of unknotting number one. Knowing that all the virtualization examples are definitely non-trivial and non-classical provides a big challenge for  the search for combinatorial proofs of these properties.

\section{Lee's Homology for Virtual Knots}

\subsection{Khovanov's Universal Algebra}\label{subsec:UniversalAlg}

Before we consider calculations of Lee's homology we first investigate the universal Frobenius system for virtual knots. A similar investigation was performed by Manturov \cite{ManturovVirtualKnots}.

Khovanov \cite{KhovanovFrobenius} showed that for classical knots the Frobenius system he labels as $\mathcal{F}_5$ is universal in that all other rank two Frobenius systems can be obtained from it via compositions of base changes and twists. Let us first recall the definition of $\mathcal{F}_5$ and then apply it to the complex in Figure \ref{fig:orientedorderedproblemsquare}.

Following Khovanov we have $\mathcal{F}_5$ given by:
\begin{align*}
\mathcal{R}_5 & = \mathbb{Z}[h, t], \\
\mathcal{A}_5 & = \mathcal{R}_5[X]/(X^2 - hX - t),\\
deg(h) & = 2,\\
deg(t) & = 4
\end{align*}
where

\begin{align*}
 \epsilon(1) & = 0,\\
 \Delta(1) & = 1\otimes X + X \otimes 1 - h 1\otimes 1,\\
 \epsilon(X) & = 1, \\
 \Delta(X) & = X\otimes X + t 1\otimes 1\\
 \iota(1)& =1
\end{align*}
and set

\[ \eta(X) = 0 = \eta(1). \]

Applying this Frobenius system to the complex in Figure \ref{fig:orientedorderedproblemsquare} we get
\begin{align*}
m \circ \Delta(1) & = m (-(1\otimes \overline{X} + X \otimes \overline{1} - h 1 \otimes \overline{1})) \\
 & = m (1\otimes X - X \otimes 1 + h 1\otimes 1) = X - X + h = h
\end{align*}
and
\begin{align*}
m \circ \Delta(X) & = m(-(X\otimes \overline{X} + t 1 \otimes \overline{1})) \\
& = m(X\otimes X - t 1\otimes 1) = hX + t - t = hX.
\end{align*}

Hence we must set $h=0$ in order to have a well-defined complex. And so the universal Frobenius system for virtual knots corresponds to the system introduced by Bar-Natan \cite{DrorCob} and labeled by Khovanov \cite{KhovanovFrobenius} as $\mathcal{F}_3$. For completeness we restate the definition of $\mathcal{F}_3$  given by Khovanov as:
\begin{align*}
\mathcal{R}_3 & = \mathbb{Z}[t],\\
\mathcal{A}_3 & = \mathcal{R}_3[X]/(X^2 - t)\\
deg(t) & = 4
\end{align*}

where

\begin{align*}
 \epsilon(1) & = 0, \\
 \Delta(1) & = 1\otimes X + X \otimes 1,  \\
 \epsilon(X) & = 1, \\
 \Delta(X) & = X\otimes X + t 1\otimes 1\\
 \iota(1)& =1
\end{align*}

and set

\[ \eta(X) = 0 = \eta(1). \]

We shall call $\mathcal{F}_3$ the universal rank 2 Frobenius system for Khovanov homology for virtual knots. Note that $\mathcal{F}_3$ is a slight generalization of the (co)homology theory defined by Lee \cite{Endomorphism} which can be arrived at by setting $t=1$ in $\mathcal{F}_3$. Furthermore, note that $\mathcal{F}_3$ is no longer a bi-graded homology theory as the differentials $m$ and $\Delta$ are not homogeneous with respect to the quantum grading. Never the less, one may apply a filtration to the former quantum gradings to arrive at a filtered theory where the quantum degree can be reinterpreted as a filtered degree. More formally, we say a (finite length) \emph{filtration} of a complex $C$ is a sequence of subcomplexes \[0 = C_n \subset C_{n - 1} \subset  \cdots \subset C_m = C.\] One can associate a grading on the filtration by saying $x \in C$ has grading $i$ whenever $x \in C_i$ but $x \notin C_{i - 1}$. Given a map $f : C \rightarrow C'$ between two filtered chain complexes we say $f$ respects the filtration if $f(C_i) \subset C'_{i}$ and $f$ is a \emph{filtered map of degree k} if $f(C_i) \subset C'_{i+k}$. A filtration ${C_i}$ on $C$ induces a filtration ${S_i}$ on $H^{*}(C)$  where a class $[x] \in H^{*}(C)$ is in $S_i$ if and only if it has a representative which is an element of $C_i$.

Rasmussen showed that when $m$ and $\Delta$ are viewed as maps between chain complex, there exists the following short exact sequence (\cite{Rasmussen} Lemma 3.8)
\[0 \rightarrow \widehat{Kh}(K_1 \sharp K_2) \xrightarrow{\Delta}  \widehat{Kh}(K_1) \otimes \widehat{Kh}(K_2) \xrightarrow{m} \widehat{Kh}(K_1 \sharp K_2) \rightarrow 0 ,\] where $\widehat{Kh}$ represents Khovanov-Lee homology. Using this short exact sequence, one can directly compute the induced filtered degrees of $m$ and $\Delta$ as chain maps. In particular, Rasmussen notes that when $K_1$ and $K_2$ are both unknots it is a quick exercise to show that the induced filtered degree of $m$ and $\Delta$ as chain maps are $-1$. Computing directly from $\mathcal{F}_3$, one get that the filtered degrees for $\epsilon$ and $\iota$ as chain maps are 1.

Before we can say more about the Rasmussen invariant, we must return to calculating the Khovanov-Lee homology. To do so we follow the construction of Bar-Natan and Morrison \cite{BarNatanMorrisonKaroubi} using the Karoubi envelope of $\mathcal{F}_3$ to guarantee projections (i.e. maps $p$ in the category satisfying $p^2=p$) exist. We then use the projections locally at a saddle morphism to create a surface algebra which is non-zero on the canonical generators of Lee's homology. As in Bar-Natan and Morrison, we see that the alternately colored smoothings are in one-to-one correspondence with orientations of the original diagram. Since, unlike classical knots, our diagrams are not planar we must first describe the surface on which we will apply this algebra.

\subsection{Checkerboard Shading on Abstract Link Diagrams}\label{sec:CheckerboardALD}

Given a virtual knot diagram $D$, Kamada and Kamada \cite{KamadaALD} have shown how to construct an abstract link diagram (ALD), a (thickened) surface with boundary, which contains $D$. The process for creating an ALD given a knot diagram is summarized in Figures \ref{fig:ASDCrossing} and \ref{fig:ASDVirtual}. One can then obtain an embedding of a knot in a (thickened) orientable surface by placing (thickened) discs along the boundary curves of the ALD. A classical knot has a diagram with associated ALD that results in an embedding of the knot diagram in $S^2$. Kamada and Kamada \cite{KamadaALD} have also defined equivalence relations on ALDs which yield equivalence classes called abstract links and furthermore shown that there is a bijection between abstract links and virtual links.

\begin{figure}[H]
\centering
    \includegraphics[width=.6\textwidth]{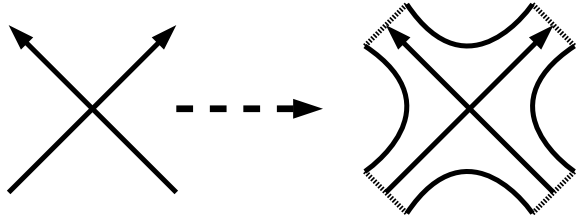}
    \caption{ALD for a classical crossing}
\label{fig:ASDCrossing}
\end{figure}

\begin{figure}[H]
\centering
    \includegraphics[width=.6\textwidth]{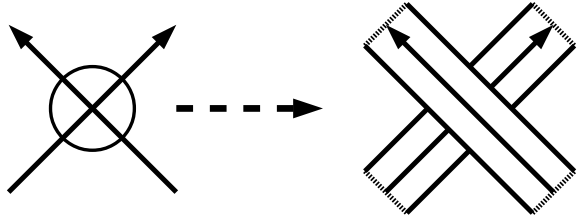}
    \caption{ALD for a virtual crossing}
\label{fig:ASDVirtual}
\end{figure}

Here we extend the definition of ALDs to a map sending locally oriented virtual knot diagrams with cut loci to (possibly non-orientable) surfaces by placing Mobius bands at the cut loci. We denote the location of the Mobius band by cross cuts on the abstract link diagram as shown in Figure \ref{fig:ALDCutLoci}. This construction is similar to the twisted link theory constructed by Bourgoin \cite{BourgoinTwistedLinks}. Manturov has constructed a version of Khovanov homology which extends to twisted links \cite{ManturovAddlGradings}. We believe our construction for Lee's theory and the Rasmussen invariant should also extend naturally to this situation and plan to return to this case in a future paper.

\begin{figure}[H]
\centering
    \includegraphics[width=.4\textwidth]{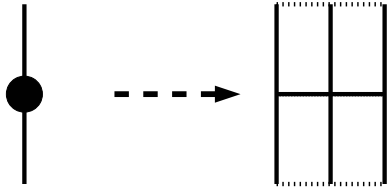}
    \caption{Cross cut ALD at a cut locus}
\label{fig:ALDCutLoci}
\end{figure}

Using the source-sink orientations of the virtual knot the ALD with cross cuts can be checkerboard shaded using the right-hand (or left-hand) rule as shown in Figures \ref{fig:ASDCheckerboardSourceSink} and \ref{fig:ASDCheckerboardCutLoci}.

\begin{figure}[H]
\centering
    \includegraphics[width=.6\textwidth]{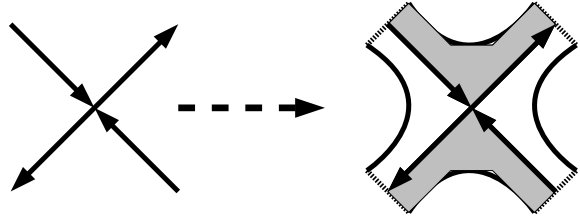}
    \caption{ALD checkerboard shading at a locally oriented crossing}
\label{fig:ASDCheckerboardSourceSink}
\end{figure}

\begin{figure}[H]
\centering
    \includegraphics[width=.4\textwidth]{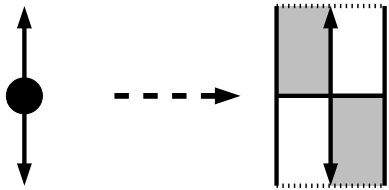}
    \caption{ALD checkerboard shading at a cross cut}
\label{fig:ASDCheckerboardCutLoci}
\end{figure}

\subsection{Lee's Homology and Alternately Colored Smoothings}\label{sec:LeeALDColoring}

For the remainder of the paper we specialize $\mathcal{F}_3$ by setting $t=1$ and hence $X^2 = 1$ as in Bar-Natan and Morrison \cite{BarNatanMorrisonKaroubi}. We denote the resulting homology theory, Khovanov-Lee theory, by $\widehat{Kh}$ and chain complexes by $\widehat{\mathcal{C}}$. Note that this theory is no longer bi-graded theory as the maps are no longer homogeneous with respect to the quantum grading. However a filtration exists based on the quantum grading. At this point one could pass to spectral sequences as was first noted by Rasmussen \cite{Rasmussen}.  Instead we take an alternate approach and only return to the filtration after computing the Khovanov-Lee homology.

Let $Cat(\widehat{\mathcal{C}})$ be the category of chain complexes up to chain homotopies. Recall from Bar-Natan and Morrison that we may pass from the homotopy category of complexes $Kom(Cat(\widehat{\mathcal{C}}))$  to the homotopy category of complexes over the Karoubi envelope $Kom(Kar(Cat(\widehat{\mathcal{C}})))$ without changing any equivalencies of complexes (\cite{BarNatanMorrisonKaroubi} Proposition 3.3).  The Karoubi envelope has the advantage that every projection has an image and furthermore that, given a projection, the objects of the original category decompose into a direct sum. More specifically, given a projection $p: \mathcal{O} \rightarrow \mathcal{O}$ in an additive category then $(1-p)$ is also a projection and $\mathcal{O} \cong Im (p) \oplus Im (1-p)$(\cite{BarNatanMorrisonKaroubi} Proposition 3.2).

Following the conventions of Bar-Natan and Morrison we set
\[``red" = \mathbf{r} =  \frac{1+X}{2}\]
and
\[``green" = \mathbf{g} =  \frac{1-X}{2}.\]
Then
\[(1-\mathbf{r}) = 1 -  \frac{1+X}{2} = \frac{1-X}{2} = \mathbf{g}\]
 and so $\mathcal{O} \cong \mathbf{r} \oplus \mathbf{g}$. Diagrammatically  this can be depicted by coloring the arcs of a state as in Figure \ref{fig:singleprojdecomp}.

\begin{figure}[H]
\centering
    \includegraphics[width=.3\textwidth]{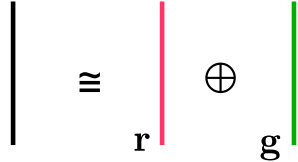}
    \caption{Red-green decomposition of a single strand}
\label{fig:singleprojdecomp}
\end{figure}

\begin{lem}[1. - 4. from \cite{BarNatanMorrisonKaroubi} Lemma 4.1]\label{lem:redgreenproperties}
If one defines $\mathbf{r}$ and $\mathbf{g}$ as above then:
\begin{enumerate}
  \item $\mathbf{r}$ and $\mathbf{g}$ are projections: $\mathbf{r}^2 = \mathbf{r}$ and $\mathbf{g}^2 = \mathbf{g}$
  \item $\mathbf{r}$ and $\mathbf{g}$ are complementary: $\mathbf{r} + \mathbf{g} = 1$
  \item $\mathbf{r}$ and $\mathbf{g}$ are disjoint: $\mathbf{r} \cdot \mathbf{g} = 0$
  \item $\mathbf{r}$ and $\mathbf{g}$ are eigenprojections of $X$: $X \cdot \mathbf{r} = \mathbf{r}$ and $X \cdot \mathbf{g} = -\mathbf{g}$
  \item $\mathbf{r}$ and $\mathbf{g}$ are conjugates: $\overline{\mathbf{r}} = \mathbf{g}$ and $\overline{\mathbf{g}} = \mathbf{r}$
  \item $\Delta(\mathbf{r}) = 2\mathbf{r \otimes r}$ and $\Delta(\mathbf{g}) = 2\mathbf{g \otimes g}$.
\end{enumerate}
\end{lem}

Using Lemma \ref{lem:redgreenproperties} we can begin to prove the following theorems:

\begin{thm}[Generalization of \cite{BarNatanMorrisonKaroubi} Theorem 1.2]\label{thm:GeneratorsforLee}
Within the Karoubi envelope the Khovanov-Lee complex of a virtual knot or link $K$ is homotopy equivalent to a complex with one generator for each alternately coloured smoothing of $K$ {\rm on an ALD with cross cuts} and with vanishing differentials.
\end{thm}

\begin{thm}[Generalization of \cite{BarNatanMorrisonKaroubi} Proposition 1.3]\label{thm:AltColoredSmoothingsforLee}
A virtual link $K$ with $c$-components has exactly $2^c$ alternately coloured smoothings {\rm on an ALD with cross cuts}. These smoothings are in a bijective correspondence with the $2^c$ possible orientations of the $c$ components of $K$.
\end{thm}

\begin{figure}[H]
\centering
    \includegraphics[width=.5\textwidth]{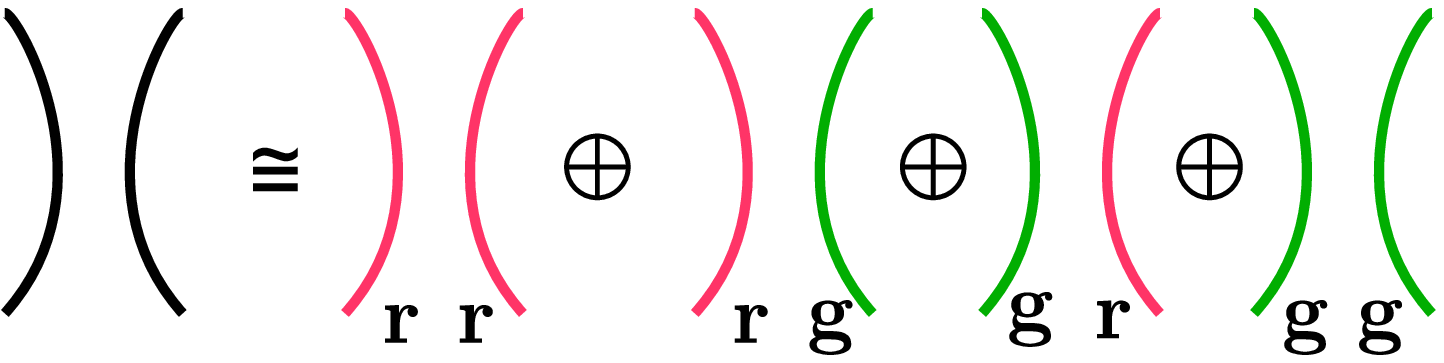}
    \caption{Red-green decomposition at a local differential}
\label{fig:doubleprojdecomp}
\end{figure}

\begin{proof}[Proof of Theorem \ref{thm:GeneratorsforLee}]
As Bar-Natan and Morrison have shown, Theorem \ref{thm:GeneratorsforLee} follows from Lemma \ref{lem:redgreenproperties} by decomposing the states of the Khovanov-Lee complex into smaller complexes at the crossings and then reassembling the states of the Khovanov-Lee complex in the method of a planar (or in our case surface) algebra. In particular, consider the complex arising from a single crossing. The A-state and B-state can be decomposed as in Figure \ref{fig:doubleprojdecomp}. It follows from Lemma \ref{lem:redgreenproperties} that the map between the A-state and B-State can be represented by a $4 \times 4$ matrix which has only two non-zero entries:

\[  [\![ \Across ]\!] \cong \left[\rrggA \right] \xrightarrow{\left(
                                                                \begin{array}{cccc}
                                                                  \saddler & 0 & 0 & 0 \\
                                                                  0 & 0 & 0 & 0 \\
                                                                  0 & 0 & 0 & 0 \\
                                                                  0 & 0 & 0 & \saddleg \\
                                                                \end{array}
                                                              \right)
} \left[\rrggB \right].\]

The two non-zero entries are contractible maps, one between the all-$red$ coloring and the other between the all-$green$ coloring. Thus, up to homotopy, we have $[\![ \Across ]\!] \cong \left[\rgA \right] \xrightarrow{0} \left[\rgB \right].$

No consider each state in the Khovanov-Lee complex as the tensor product of tangles on an ALD with cross cuts.  After applying the above differential along with Lemma \ref{lem:redgreenproperties}, we are left with a collection of states with alternating colorings at each smoothing. Finally, reassembling the Khovanov-Lee complex and applying Lemma \ref{lem:redgreenproperties} reduces the remaining collection of states to those stated in the lemma.
 \end{proof}

 \begin{proof}[Proof of Theorem \ref{thm:AltColoredSmoothingsforLee}]

Figures \ref{fig:ALDandColors} and \ref{fig:ALDCutLociColors} describe the bijection between checkerboard colored ALDs with cross cuts and alternately colored smoothings. Starting with an oriented virtual knot diagram on an ALD with cross cuts if one places a clockwise orientation on the shaded regions, an arc is labeled $red$ if its oriented smoothing agrees with the clockwise orientation and otherwise is labeled $green$.

Going the other direction an alternately colored smoothing induces an alternating shading on the regions of the ALD with cross cuts. Placing a clockwise orientation on the shaded regions we can orient the $red$ arcs so that they agree with the clockwise orientation and $green$ arcs disagree. This in turn induces an orientation on the virtual knot diagram which agrees with the $red$ arcs. These equivalences can be depicted diagrammatically as in Figures \ref{fig:ALDandColors} and \ref{fig:ALDCutLociColors}.
\end{proof}

  \begin{figure}[h!]
\centering
    \includegraphics[width=.5\textwidth]{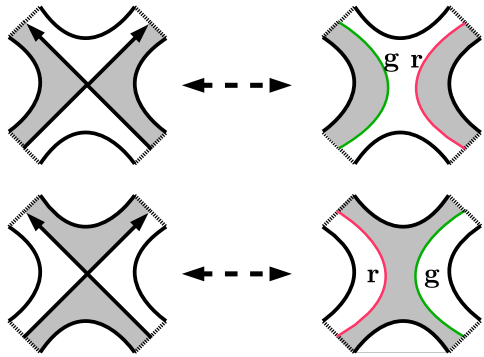}
    \caption{Red-green coloring and checkerboard coloring on an ALD}
\label{fig:ALDandColors}
\end{figure}

\begin{figure}[h!]
\centering
    \includegraphics[width=.4\textwidth]{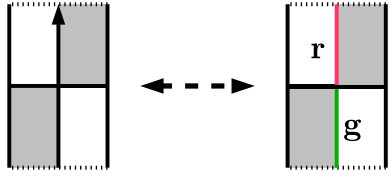}
    \caption{Red-green coloring of an ALD at a cross cut}
\label{fig:ALDCutLociColors}
\end{figure}

 As an example of an application of this approach to calculating the canonical generators of the Khovanov-Lee homology, see Figure \ref{fig:ExampleCannonicalStates} where we show the canonical generators for a two crossing virtual knot arising from the source-sink diagram with cut loci as given in Figure \ref{fig:examplecutloci}.

\begin{figure}[h!]
\centering
    \includegraphics[width=.5\textwidth]{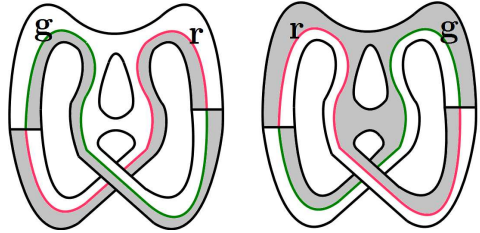}
    \caption{The canonical generators for a two crossing virtual knot}
\label{fig:ExampleCannonicalStates}
\end{figure}

\begin{rem}
The one-to-one correspondence between checkerboard colorings of the ALD with cross cuts and the $red$ and $green$ labels described above allows one to work solely with knot diagrams having the appropriate $red$ and $green$ labels. This provides an analogous correspondence between the Khovanov-Lee construction for virtual knots as described here and the Khovanov-Lee constructions for classical knots found in the literature. The latter constructions are typically described solely with labeled circle in the plane. To emphasize this correspondence we provide Figure \ref{fig:ExampleCannonicalStatesNoALD} which provides an equivalent representation of Figure \ref{fig:ExampleCannonicalStates} without the ALD.
\end{rem}

\begin{figure}[h!]
\centering
    \includegraphics[width=.4\textwidth]{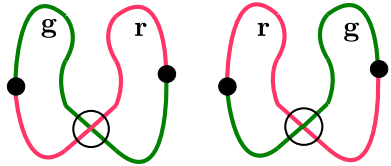}
    \caption{An equivalent representation for the canonical generators for a two crossing virtual knot}
\label{fig:ExampleCannonicalStatesNoALD}
\end{figure}

\begin{rem}
After posting a previous version of this paper on the arXiv it was pointed out to the authors that, earlier the same month, Tubbenhauer \cite{Tubbenhauer} posted a similar derivation of Khovanov-Lee homology for virtual knots using the Karoubi envelope. Comparing the two constructions, there are a few technical differences. In particular the signs appearing our Lemma \ref{lem:redgreenproperties} and the related lemma in his paper do not always coincide. These difference appear to arise from the differences between our construction of Khovanov homology following Manturov and the cobordism approach taken in his paper.
\end{rem}

\section{Virtual Knot Cobordisms }

\subsection{Virtual Knot Cobordisms and the Virtual Slice Genus}
The definitions and basic material from this section are from \cite{VKC}.

 Two oriented knots or links $K$ and $K'$ are {\it virtually cobordant}  if one may be obtained from the other by a sequence of virtual isotopies (Reidemeister moves plus detour moves) plus births, deaths and oriented saddle points, as illustrated in  Figure~\ref{saddle}. A {\it birth} is the introduction into the diagram of an isolated unknotted cycle. A {\it death} is the removal from the diagram of an isolated unknotted cycle. A saddle point move results from bringing oppositely oriented arcs into proximity and resmoothing the resulting site to obtain two new oppositely oriented arcs. See Figure \ref{saddle} for an illustration of the process. Figure~\ref{saddle} also illustrates the {\it schema} of surfaces that are generated by  cobordism process. These are abstract surfaces with well defined genus in terms of the sequence of steps in the cobordism. In  Figure \ref{saddle} we illustrate two examples of genus zero, and one example of genus 1. We say that a cobordism has genus $g$ if its schema has that genus. Two knots are {\it concordant} if there is a cobordism of genus zero  connecting them. A virtual knot is said to be a {\it slice} knot if it is virtually concordant to the unknot, or equivalently if it is virtually concordant to the empty knot (The unknot is concordant to the empty knot via one death.) Finally we define the {\it virtual slice genus} (or simply {\it slice genus} when in the virtual category) of a knot or link to be the minimal genus of all such schema between the knot or link and the unknot.
\bigbreak

\begin{figure}[h!]
     \begin{center}
     \begin{tabular}{c}
     \includegraphics[width=5cm]{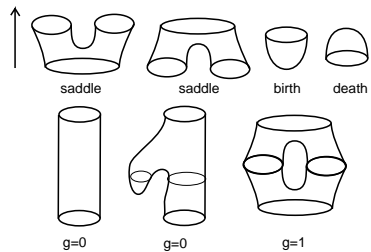}
     \end{tabular}
     \caption{Saddles, births and deaths}
     \label{saddle}
\end{center}
\end{figure}

In Figure~\ref{vstevedore} we illustrate the {\it virtual Stevedore's knot, VS,} and show that it is a slice knot in the sense of the above definition. We will use this example to illustrate our theory of virtual knot cobordism, and the questions that we are investigating.
\bigbreak
\bigbreak

\begin{figure}[h!]
     \begin{center}
     \begin{tabular}{c}
     \includegraphics[width=7cm]{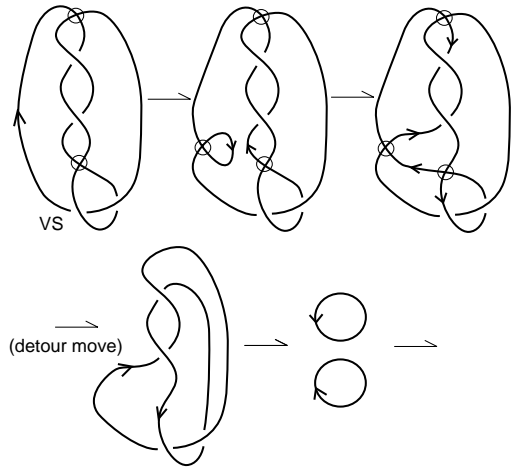}
     \end{tabular}
     \caption{Virtual Stevedore is slice}
     \label{vstevedore}
\end{center}
\end{figure}

The virtual Stevedore ($VS$) is an example that illustrates the viability of this theory. One  can prove that $VS$ is not classical by showing that it is represented on a surface of genus one and no smaller. The technique for this is to use the bracket expansion on a toral representative of $VS$ and examine the structure of the state loops on that surface.  See \cite{VKC}.
\bigbreak

\subsection{Virtual Knot Cobordisms and Seifert Surfaces}

It is a well-known that every oriented classical knot or link bounds an embedded orientable surface in three-space. A representative surface of this kind can be obtained by the algorithm due to Seifert (See \cite{OnKnots}). We have illustrated Seifert's algorithm for a trefoil diagram in Figure~\ref{seifert}. The algorithm proceeds as follows: At each oriented crossing in a given diagram $K,$ smooth that crossing in the oriented manner (reconnecting the arcs locally so that the crossing disappears and the connections respect the orientation). The result of this operation is a collection of oriented simple closed curves in the plane, usually called the \emph{ Seifert circles}. To form the \emph{ Seifert surface} $F(K)$ for the diagram $K,$ attach disjoint discs to each of the Seifert circles, and connect these discs to one another by local half-twisted bands at the sites of the smoothing of the diagram. This process is indicated in the Figure~\ref{seifert}. In that figure we have not completed the illustration of the outer disc.

It is important to observe that we can calculate the genus of the resulting surface quite easily from the combinatorics of the classical knot diagram $K.$ For purposes of simplicity, we shall assume that we are dealing with a knot diagram (one boundary component) and leave the case of links to the reader.

\begin{figure}[h!]
     \begin{center}
     \begin{tabular}{c}
     \includegraphics[width=7cm]{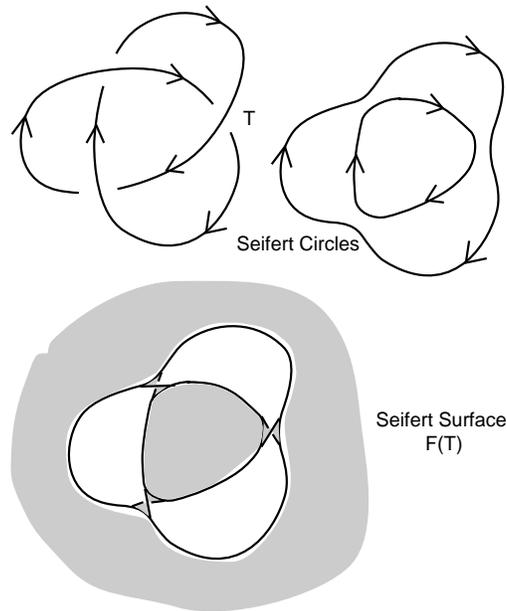}
     \end{tabular}
     \caption{Classical Seifert surface}
     \label{seifert}
\end{center}
\end{figure}

\begin{lem} Let $K$ be a classical knot diagram with $n$ crossings and $r$ Seifert circles.
then the genus of the Seifert Surface $F(K)$ is given by the formula
$$g(F(K)) =\frac{( -r + n +1)}{2}.$$
\end{lem}

\begin{proof}
The surface $F(K),$ as described prior to the statement of the Lemma, retracts to a cell complex consisting of the projected graph of the knot diagram with two-cells attached to each cycle in the graph that corresponds to a Seifert circle. Thus we have that the Euler characteristic of this suface is given the the formula $$\chi(F(K)) = n - e + r$$ where $n,$ the number of crossings in the diagram, is the number of zero-cells, $e$ is the number of one-cells (edges) in the projected diagram (from node to node), and $r$ is the number of Seifert circles as these are in correspondence with the two-cells. However, we know that $4n = 2e$ since there are four edges locally incident to each crossing. Thus,
$$\chi(F(K)) =  - n + r.$$ Furthermore, we have that $\chi(F(K)) = 1 - 2g(F(K)),$ since  this surface has a single boundary component and is orientable. From this it follows that $1-2g(F(K)) = -n + r,$ and hence
$$g(F(K)) =\frac{( -r + n +1)}{2}.$$
\end{proof}

\begin{figure}[h!]
     \begin{center}
     \begin{tabular}{c}
     \includegraphics[width=7cm]{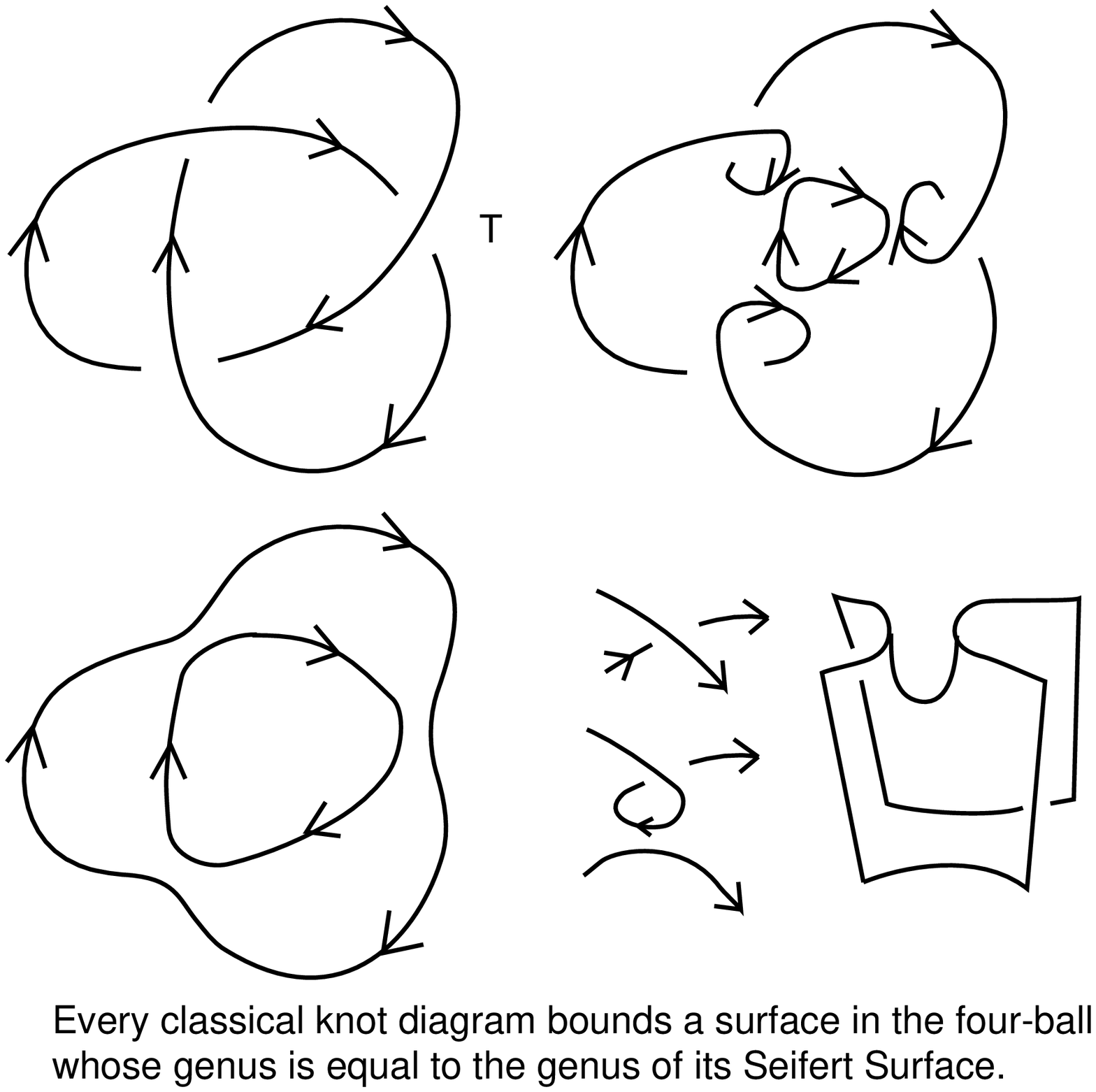}
     \end{tabular}
     \caption{Classical cobordism surface}
     \label{classicalcob}
\end{center}
\end{figure}

We now observe that \emph{for any classical knot $K,$ there is a surface bounding that knot in the four-ball that is homeomorphic to the Seifert surface}.  One can construct this surface by  pushing the Seifert
surface into the four-ball keeping it fixed along the boundary. We will give here a  different description of
this surface as indicated in Figure~\ref{classicalcob}. In that figure we \emph{perform a saddle point transformation at every crossing of the diagram.} The result is a collection of unknotted and unlinked curves. By our interpretation of surfaces in the four-ball obtained by saddle moves and isotopies, we can then bound each of these curves by discs (via deaths of circles) and obtain a surface $S(K)$ embedded in the four-ball with boundary $K.$ As the reader can easily see, the curves produced by the saddle transformations are in one-to-one correspondence with the Seifert circles for $K,$ and it easy to verify that $S(K)$ is homeomorphic with the Seifert surface $F(K).$ Thus we know that $g(S(K)) =\displaystyle\frac{( -r + n +1)}{2}.$ In fact the same argument that we used to analyze the genus of the Seifert surface applies directly to the construction of $S(K)$ via saddles and minima.

\begin{figure}[h!]
     \begin{center}
     \begin{tabular}{c}
     \includegraphics[width=7cm]{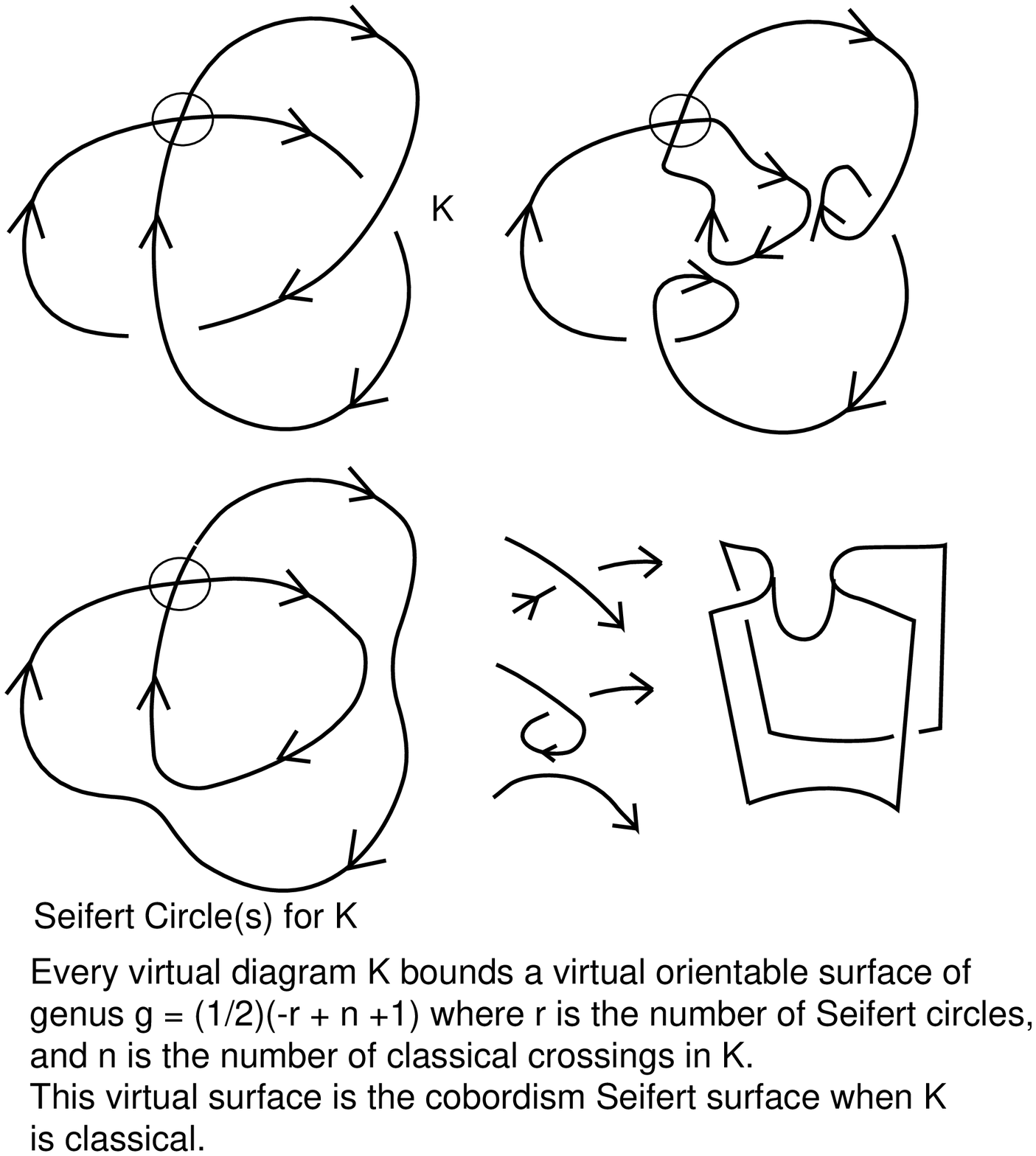}
     \end{tabular}
     \caption{Virtual cobordism Seifert surface}
     \label{virtseifert}
\end{center}
\end{figure}

\begin{figure}[h!]
     \begin{center}
     \begin{tabular}{c}
     \includegraphics[width=7cm]{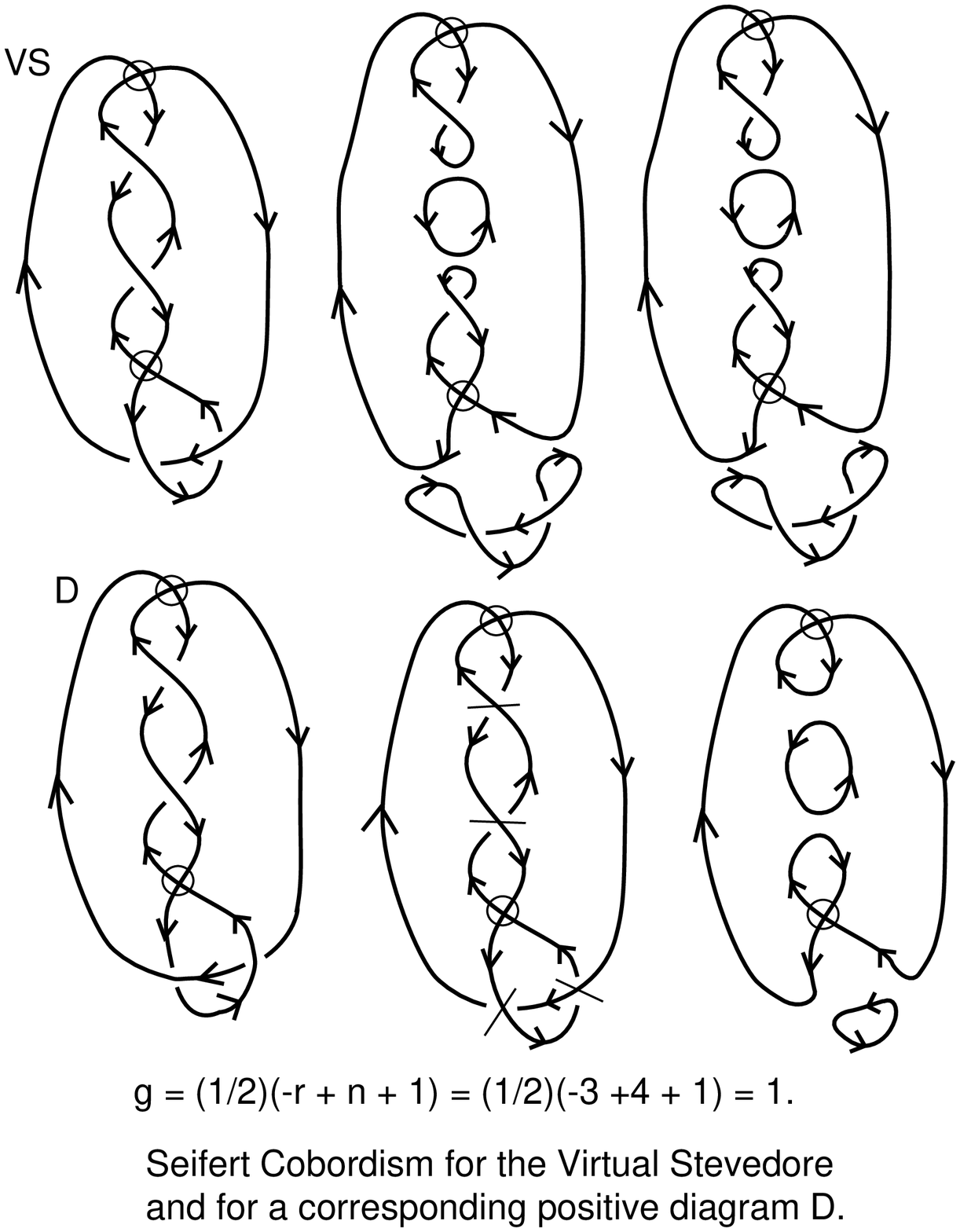}
     \end{tabular}
     \caption{Virtual Stevedore cobordism Seifert surface}
     \label{vstevedoreseifert}
\end{center}
\end{figure}

Now the stage is set for generalizing the Seifert surface to a surface $S(K)$ for virtual knots $K.$
View Figure~\ref{virtseifert} and Figure~\ref{vstevedoreseifert}. In these figures we have performed a saddle transformation at each classical crossing of a virtual knot $K.$ The result is a collection of unknotted curves that are isotopic (by the first classical Reidemeister move) to curves with only virtual crossings. Once the first Reidemeister moves are performed, these curves are identical with the \emph{virtual Seifert circles} obtained from the diagram $K$ by smoothing all of its classical crossings. We can then isotope these circles into a disjoint collection of circles (since they have no classical crossings) and cap them with discs in the four-ball. The result is a virtual surface $S(K)$ whose boundary is the given virtual knot $K.$ We will use the terminology \emph{virtual surface in the four-ball} for this surface schema. In the case of a virtual slice knot, we have that the knot bounds a virtual surface of genus zero. But with this construction we have proved the following lemma.

\begin{lem}\label{lem:SeifertGenus}
 Let $K$ be a virtual knot diagram, then the virtual Seifert surface $S(K)$ constructed above
has genus given by the formula $$g(S(K)) = \frac{( -r + n +1)}{2}$$ where $r$ is the number of virtual Seifert
circles in the diagram $K$ and $n$ is the number of classical crossings in the diagram $K.$
\end{lem}

\begin{proof}
The proof follows by the same argument given in the classical case. Here the projected virtual diagram gives a four-regular graph $G$ (not necessarily planar) whose nodes are in one-to-one correspondence with the classical crossings of $K.$ The edges of $G$ are in one-to-one correspondence with the edges in the diagram that extend from one classical crossing to the next. We regard $G$ as an abstract graph so the the virtual crossings disappear.
The argument then goes over verbatim in the sense that $G$ with two-cells attached to the virtual Seifert circles is a retract of the surface $S(K)$  constructed by cobordism. The counting argument for the genus is identical to the classical case.
\end{proof}

\section{A Rasmussen Invariant for Virtual Knot Cobordisms}
\subsection{Proof of Invariance}

Recall that the Khovanov-Lee homology is no longer a bi-graded homology theory as the maps are not homogeneous on the quantum grading, but can be regarded as a filtered theory with respect to the quantum grading.  As was first shown by Rasmussen \cite{Rasmussen} the information contained in the filtered theory is still quite powerful. In particular it allows one to get a bound on the genus of a cobordism between two knots.

For virtual knots we use the definition of cobordism given in the previous section. It is important to note that the Rasmussen invariant can be viewed as an obstruction to knot concordance (a cobordism of genus 0) for knots in a thickened surface cross an interval, $S_g \times I \times I$, by viewing the knots in $S_g \times I$ as virtual knots. Note that in constructing such a knot concordance the choice of embedding is very important. It is a quick exercise to show that one can find two embeddings of the unknot in the thickened torus which are not concordant in the sense of Turaev \cite{Turaev}.

\begin{figure}[h!]
\centering
    \includegraphics[width=.5\textwidth]{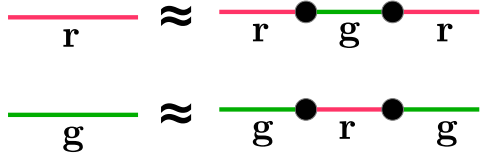}
    \caption{Cut loci cancelation}
\label{fig:CutLociCancelation}
\end{figure}

\begin{lem}\label{lem:cutlocicancel}
Given a canonical generator with alternating colors, one may cancel adjacent cut loci, and the new arc takes the color of the outer arcs. Similarly, one may add pairs of adjacent cut loci to a strand where the arc between the cut loci takes on the opposite color to the original strand. This equivalence is as depicted in Figure \ref{fig:CutLociCancelation}.
\end{lem}

\begin{proof}
Note that this operation does not change the basis of a strand, rather it is simply conjugation. However, if one applies cut loci cancelation to the canonical generator on an ALD with cross cuts, this does effect the checkerboard coloring as in Figure \ref{fig:CutLociCancelationALD}. Furthermore, to apply the cut loci cancelation one may also need to apply an equivalence relation on the ALD as described by Kamada and Kamada \cite{KamadaALD}. Note that this is well-defined since the canonical generator is unique for the knot diagram.
\end{proof}

\begin{figure}[h!]
\centering
    \includegraphics[width=.5\textwidth]{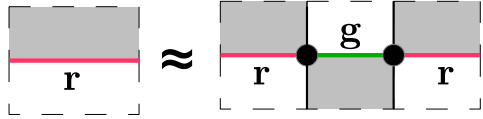}
    \caption{Cut loci cancelation on ALD}
\label{fig:CutLociCancelationALD}
\end{figure}

For example see Figure \ref{fig:ALDCutLociCancelExample}.

\begin{figure}[h!]
\centering
    \includegraphics[width=.95\textwidth]{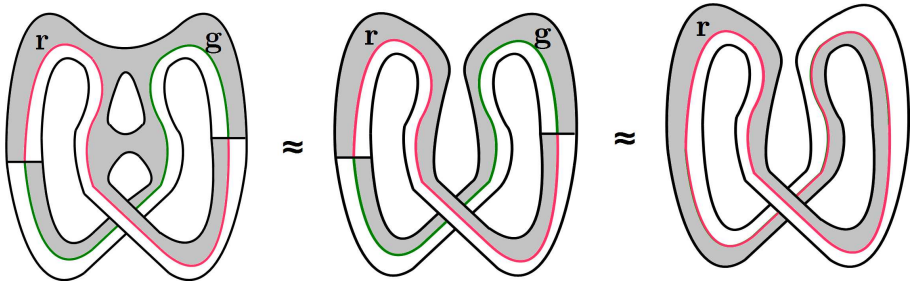}
    \caption{Equivalence for canonical generator of the unknot}
\label{fig:ALDCutLociCancelExample}
\end{figure}

\begin{rem}
While the ALDs with cross cuts are essential in determining the canonical generators we omit most reference to them in the following proofs in order to simplify the exposition. The careful reader will note that they play an important role in the equivalence of the Reidemeister III move as presented below.
\end{rem}

\begin{rem}
We will say that \emph{(canonical) generators are isomorphic (resp. homomorphic)} if there is an isomorphism (resp. homomorphism) of the containing algebra that carries generators to generators.
\end{rem}

\begin{prop}[Generalization of Rasmussen \cite{Rasmussen} Prop 2.3 / See also Caprau \cite{CaprauFiltration} Prop 2]\label{prop:RasCobRM}
Let $\mathfrak{C}$ be an elementary link cobordism from $D$ to $D'$  where $D$ to $D'$ are related by a Reidemeister move. Then $\mathfrak{C}$ induces an isomorphism of the canonical generators of $\widehat{Kh}(D)$
and $\widehat{Kh}(D')$.
\end{prop}

\begin{proof}

As the canonical source-sink orientations are dependant on the orientation of the crossings we must show invariance under the oriented Reidemeister moves. In order to minimize the number of changes in cut loci in the diagrams outside of where the Reidemeister move occurs we chose to work with a set of oriented moves shown by Polyak \cite{PolyakMinimal}, Theorem 1.2, to be minimal.

In the following one must take care when reading orientations. The orientations of the knot or link diagram presented in the Reidemeister moves induce source-sink orientations on the states corresponding to the canonical generators. These orientations and cut loci on an ALD, as defined in Section \ref{sec:LeeALDColoring}, are the orientations appearing in the complexes. As noted above we have suppressed the ALD inside of the complex to simplify the diagrams.

\noindent{Reidemeister I}:

First consider the complexes arising from the Reidemeister move\\ \RIb. Let $D_0$ and $D_1$ represent the left- and right-hand sides of the diagram respectively. Then the associated canonical generators are given by $ \llbracket D_0 \rrbracket  = ( 0 \rightarrow \RIbLeft \{1\} \rightarrow 0)$ and $\llbracket D_1 \rrbracket = ( 0 \rightarrow \RIbRight \rightarrow 0)$ where both canonical generators sit at homological height zero. The chain homotopies  $g: D_0 \rightarrow D_1$ and $f: D_1 \rightarrow D_0$ are given by $g_0 = Id \otimes \epsilon$ where $\epsilon$ occurs on the closed cycle and $f_0 = Id \otimes X \cdot \iota + X \cdot Id \otimes \iota$. Note that by Lemma \ref{lem:redgreenproperties} the $Id$ map preserves the color of the strand and furthermore the multiplication by $X$ in the second summand of $f_0$ does not effect the color of the strand.

The maps for \RId are similar.

\medskip

\noindent{Reidemeister II}:

Next consider the complexes arising from the Reidemeister move\\ \RIIc. As before let $D_0$ and $D_1$ represent the left- and right-hand sides of the diagram respectively. Then the associated canonical generators are given by $ \llbracket D_0 \rrbracket  = ( 0 \rightarrow \RIIcLeft \rightarrow 0)$ and $\llbracket D_1 \rrbracket = ( 0 \rightarrow \RIIcRight \rightarrow 0)$ where both canonical generators sit at homological height zero. Note that by Theorem \ref{thm:GeneratorsforLee} the the strands in the state  $\RIIcLeft$ of $\llbracket D_0 \rrbracket$ must alternate color. Similarly the strands in $\RIIcRight$ of $\llbracket D_1 \rrbracket$ must be of the same color (due to the equivalence of ALDs \cite{KamadaALD} Corollary 4.4). The chain homotopies  $g: D_0 \rightarrow D_1$ and $f: D_1 \rightarrow D_0$ can be described as $g_0 = Saddle \otimes \epsilon$, where $\epsilon$ occurs on the closed cycle and $Saddle$ represents a saddle-morphism (either $m$, $\Delta$ or $\mu$ depending on the connectivity of the endpoints outside of the local diagram) and is non-zero by Lemma \ref{lem:redgreenproperties} on the strands of the same color. Similarly the map $f_0 = Saddle \otimes \iota$ where $\iota$ occurs at the closed cycle.

The maps for \RIId are similar.

\medskip

\noindent{Reidemeister III}:

Finally consider the complexes arising from the Reidemeister move\\ \RIIIb. As before let $D_0$ and $D_1$ represent the left- and right-hand sides of the diagram respectively. Then the associated canonical generators are given by $ \llbracket D_0 \rrbracket  = ( 0 \rightarrow \RIIIbLeft  \rightarrow 0)$ and $\llbracket D_1 \rrbracket = ( 0 \rightarrow \RIIIbRight \rightarrow 0)$ where both canonical generators sit at homological height zero.  Before we give the maps, we first point out the overall change in orientation in the middle strand between $\RIIIbLeft$ and $\RIIIbRight$. This change in orientation implies that the diagrams for the canonical generators differ by the addition (or removal) of cut loci immediately above and below the middle strand since the source-sink orientations are fixed outside of the local picture. (This is an equivalent situation to what happens when one switches the source-sink orientation at a single crossing in a virtual knot diagram.) Hence $g: D_0 \rightarrow D_1$ and $f: D_1 \rightarrow D_0$ can be described as $g_0 =f_0 = Id \otimes \overline{Id} \otimes Id$ where by Lemma \ref{lem:redgreenproperties} the $Id$ preserves and $\overline{Id}$ reverses the color of the strand. The resulting canonical generators are equivalent by Lemma \ref{lem:cutlocicancel}.
\end{proof}

\begin{prop}[Generalization of Rasmussen \cite{Rasmussen} Prop 4.1 / See also Caprau \cite{CaprauFiltration} Prop 3]\label{prop:RasCobBDS}
Let $\mathfrak{C}$ be an elementary link cobordism from $K$ to $K'$ where $K$ and $K'$ are related by a birth, death or saddle, then $\mathfrak{C}$ induces a homomorphism of the canonical generators of $\widehat{Kh}(K)$ and $\widehat{Kh}(K')$ satisfying:
\begin{enumerate}
  \item If $\mathfrak{C}$ corresponds to a birth, then any canonical generator of $\widehat{Kh}(K)$ induces two canonical generators of $\widehat{Kh}(K')$ which agree on all components of $\widehat{Kh}(K)$ other than the new one and take the two possible values $\mathbf{r}$ and $\mathbf{g}$ on the new cycle.
  \item If $\mathfrak{C}$ corresponds to a death then any canonical generator of $\widehat{Kh}(K)$ induces a canonical generator of $\widehat{Kh}(K')$ that agrees on all remaining cycles of the canonical generator of $\widehat{Kh}(K)$.
  \item If $\mathfrak{C}$ corresponds to a saddle then there are two situations to consider.
  \begin{enumerate}
    \item If the two strands of the canonical generator of $\widehat{Kh}(K)$ involved have different values (i.e. one labeled $\mathbf{r}$ and the other $\mathbf{g}$) then the induced map is zero.
    \item If the two strands of the canonical generator of $\widehat{Kh}(K)$ involved have the same values (i.e. either both labeled $\mathbf{r}$ or both labeled $\mathbf{g}$) then the induced map takes the common value on affected component(s) of the canonical generator of $\widehat{Kh}(K')$.
  \end{enumerate}
\end{enumerate}
\end{prop}

\begin{proof}
For the first statement, note that we do not require the ALD for disjoint links to be connected. The ALD corresponding to a birth is an annulus with the new cycle occurring as a central circle. This has two possible checkerboard colorings and so the claim follows. The second statement follows from the definition of $\epsilon$ in $\mathcal{F}_3$. Finally, the third statement is an immediate consequence of Lemma \ref{lem:redgreenproperties}. Hence we have an explicit homomorphism between the canonical generators or $\widehat{Kh}(K)$ and $\widehat{Kh}(K')$ which, up to multiplication by an element in the base field, sends canonical generators to canonical generators.
\end{proof}

\begin{cor}[Generalization of Rasmussen \cite{Rasmussen} Cor 4.2 / See also Caprau \cite{CaprauFiltration} Cor 2]\label{cor:RasCobIso}
Let $\mathfrak{C}$ be a connected virtual knot cobordism from $K$ to $K'$, then $\mathfrak{C}$ induces an isomorphism of the canonical generators of $\widehat{Kh}(K)$ and $\widehat{Kh}(K')$.
\end{cor}

\begin{proof}
Recall that any connected cobordism can be decomposed as the union of elementary link cobordisms. Applying  Propositions \ref{prop:RasCobRM} and \ref{prop:RasCobBDS} along with the compatibility conditions of Lemma \ref{lem:redgreenproperties} gives the desired result.
\end{proof}

\begin{defn}\label{def:RasInvar}
Let $s_{max}(K)$ and $s_{min}(K)$ be the highest and lowest filtration levels of $\widehat{Kh}(K)$. Then the Rasmussen invariant of $K$ is defined by
\[ \overline{s}(K) = \frac{1}{2}[s_{max}(K)+s_{min}(K)]. \]
\end{defn}

 The following proposition is a collection of properties noted and proved by Rasmussen (\cite{Rasmussen} Propositions 3.3, 3.4 and 3.9).

\begin{prop}[Properties of the Rasmussen invariant]\label{prop:PropRasInvar}
Let K be a virtual knot and $s_{min}(K)$, $s_{max}(K)$ and $\overline{s}(K)$ be defined as in Definition \ref{def:RasInvar}. Furthermore let $\overline{K}$ be the mirror image of $K$ defined by changing the signs of all crossings. Then the following properties hold:
\begin{enumerate}
\item $s_{max}(K) = s_{min}(K) + 2$\label{RasInvarProp1}
\item $\overline{s}(K) =  s_{min}(K) + 1$\label{RasInvarProp2}
\item $s_{max}(\overline{K}) = -s_{min}(K)$\label{RasInvarProp3}
\item $s_{min}(\overline{K}) = -s_{max}(K)$\label{RasInvarProp4}
\item $\overline{s}(\overline{K}) = -\overline{s}(K)$\label{RasInvarProp5}
\end{enumerate}
\end{prop}

\begin{proof}
We omit the proof of \hyperref[RasInvarProp2]{\ref*{prop:PropRasInvar}.\ref*{RasInvarProp1}} as the homological algebra in the virtual knot case follows nearly identically to that of the classical knot case as given in Rasmussen. The only change being we must consider orientations on the knot diagrams. \hyperref[RasInvarProp2]{\ref*{prop:PropRasInvar}.\ref*{RasInvarProp3}} and \hyperref[RasInvarProp2]{\ref*{prop:PropRasInvar}.\ref*{RasInvarProp4}} follow from the observation that taking the mirror image $\overline{K}$ of a knot $K$ produces an isomorphism of filtered complex where the elements of filtered degree $C_i$ in the complex for $K$ correspond to elements of filtered degree $C_{-i}$ in the complex for $\overline{K}$.  For knots, Proposition \hyperref[RasInvarProp2]{\ref*{prop:PropRasInvar}.\ref*{RasInvarProp2}} is an immediate consequence of Proposition \hyperref[RasInvarProp2]{\ref*{prop:PropRasInvar}.\ref*{RasInvarProp1}} and Definition \ref{def:RasInvar}. Similarly, \hyperref[RasInvarProp2]{\ref*{prop:PropRasInvar}.\ref*{RasInvarProp5}} follows from \hyperref[RasInvarProp2]{\ref*{prop:PropRasInvar}.\ref*{RasInvarProp3}} and \hyperref[RasInvarProp2]{\ref*{prop:PropRasInvar}.\ref*{RasInvarProp4}}.
\end{proof}

\begin{thm}[Generalization of Rasmussen \cite{Rasmussen} Theorem 1]\label{thm:Ras1}
Let $g_s(K)$ denote the slice genus of the virtual knot K. Then $| \overline{s}(K) | \leq 2g_s(K)$.
\end{thm}

\begin{proof}
Following Rasmussen \cite{Rasmussen}, suppose $\mathfrak{C}$ is a connected cobordism between $K$ and the unknot, $U$, of genus $g$ and let $x$ be a maximal non-zero element of $\widehat{Kh}(K)$.  By Corollary \ref{cor:RasCobIso}, $\mathfrak{C}$  induces an isomorphism, $\phi$, between $x$ and the maximal element of $\widehat{Kh}(U)$. Recall that the filtered degrees of $m$ and $\Delta$ in $\mathcal{F}_3$ are $-1$ while $\eta$ and $\iota$ are 1. Hence the induced isomorphism $\phi$ is of filtered degree -2g and $s(\phi(x)) \geq s(x)-2g$. Since $s_{max}(U)=1$ we also have that $s(\phi(x)) \leq 1$. Combining the two inequalities gives that $s(x) \leq 2g + 1$. Thus  by Propositions \hyperref[RasInvarProp2]{\ref*{prop:PropRasInvar}.\ref*{RasInvarProp1}} and  \hyperref[RasInvarProp2]{\ref*{prop:PropRasInvar}.\ref*{RasInvarProp2}}, $s_{max}(K) \leq 2g+1$ and $\overline{s}(K) \leq 2g$. Finally to show that $\overline{s}(K) \geq -2g$ one can repeat the argument with $\overline{K}$ and apply Proposition \hyperref[RasInvarProp2]{\ref*{prop:PropRasInvar}.\ref*{RasInvarProp5}}.
\end{proof}

\begin{rem}
We will apply Theorem \ref{thm:Ras1} to positive virtual knots in Theorem \ref{thm:PositiveGenus}, generalizing Rasmussen's result which gave a combinatorial proof of a conjecture by Milnor regarding the genus of torus knots.
\end{rem}

\subsection{Examples and Calculations}

When calculating the Rasmussen invariant for a virtual knot one only needs to know the resulting quantum grading on the two remaining non-zero copies of $\mathbb{Q}$ in the Khovanov-Lee homology. This follow directly from the definition in the previous subsection. Furthermore, to compute $s(K)$ we will use Proposition \hyperref[RasInvarProp2]{\ref*{prop:PropRasInvar}.\ref*{RasInvarProp2}}.  Here we consider the class of virtual knots coming from closures of virtual braids of the form $v\sigma^{2n}$ where $v, \sigma \in VB_2$. For more on virtual braids see \cite{VKT} \cite{VirtualBraids}.

Knots of this form are given in Figure \ref{fig:2n1vVirtuals}. The two crossing knot, the closure of $v\sigma^{2}$, whose source-sink diagram with cut loci was given in Figure \ref{fig:examplecutloci} and a canonical generator was given in Figures \ref{fig:ExampleCannonicalStates} and \ref{fig:ALDCutLociCancelExample}, is of this form. Furthermore all closures of braids $v\sigma^{2n}$ have one canonical generator generalizing the form given in Figure \ref{fig:ExampleCannonicalStates} and the other is oppositely colored (i.e. swap $red$ and $green$). Both can be simplified to a single cycle labeled $red$ or $green$ via the same process as is given in Figure \ref{fig:ALDCutLociCancelExample}.

\begin{figure}[h!]
\centering
    \includegraphics[height=.6in]{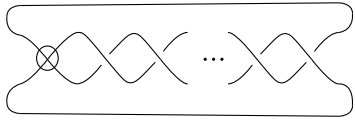}
    \caption{The closure of a 2-braid of the form $v\sigma^{2n}$ }
\label{fig:2n1vVirtuals}
\end{figure}

To calculate $s_{min}(K)$  for knots of the form appearing in Figure \ref{fig:2n1vVirtuals} we can return to the original formula for the quantum grading given for Khovanov homology. Namely for an enhanced state $s$, $q(s) = j(s) + (n_+ - 2n_-)$. Note that we can do so since the canonical generators have only a since cycle.

Let $s_0$ be the canonical state where, instead of decorating by $red$ and $green$, we decorate the single cycle of $s_0$  by $X$. Then $s_{min}(K) = q(s_0)$. Hence if $K$ is of the form given in Figure \ref{fig:2n1vVirtuals} we have $s_{min}(K) = -1 + 2n$ and so by Proposition \hyperref[RasInvarProp2]{\ref*{prop:PropRasInvar}.\ref*{RasInvarProp2}} $\overline{s}(K) =  2n$. Thus no knot of this form is slice and furthermore $n$ is a lower bound on the slice genus.

While we relied on the fact that the canonical generators had only a single cycle we could also note that the previous computation is an example of the case of positive virtual knots. As Rasmussen first noted (\cite{Rasmussen} Section 5.2) when K is a positive knot diagram we can compute the Rasmussen invariant directly from the canonical generators. Here we generalize this to the case of positive virtual knots.

\begin{lem}\label{prop:PosRas}
If K a virtual diagram with only positive crossings (i.e. K represents a positive virtual knot) then $s_{min}(K) = q(s_0) = (-r + n + 1)$ where $r$ is the number of cycles in $s_0$ and $n$ is the number of positive crossings for K.
\end{lem}

\begin{proof}
Let $s_0$ be a canonical generator of K. Since K is a positive virtual knot diagram $s_0$ appears in the A-state of the Khovanov complex. Since there are no generators in $\widehat{CKh}(K)$ with lower homological grading it follows that the only class homologous to $s_0$ is $s_0$ itself. Thus $s_{min}(K) = s([s_0]) = q(s_0)$.

To compute $q(s_0)$ we can, as above, consider $s_0$ where, instead of decorating by $red$ and $green$, we decorate all cycles by $X$.  Then $s_{min}(K) = q(s_0) = (-r + n + 1)$ where $r$ is the number of cycles in $s_0$ and $n$ is the number of positive crossings for K.
\end{proof}

\begin{rem}
If K is not a positive (or negative) virtual knot it can be difficult to compute the invariant directly from the canonical generators. This is not to say that the computation is not possible, but rather (in most cases) one must compute the Rasmussen invariant via the spectral sequence. We plan to return to this topic and consider the Rasmussen invariant for additional classes of virtual knots in a following paper.
\end{rem}

As an example of a calculation for a knot with both positive and negative crossings we return to the virtual Stevedore which was shown to be slice in Figure \ref{vstevedore}. To obtain the Rasmussen invariant one must in principle compute the spectral sequence and determine the remaining highest and lowest non-zero quantum filtration levels. In this case the computation can be simplified by noting that Manturov previously proved that Z-equivalence as pictured in Figure \ref{fig:ZEquiv} induces an isomorphism on the Khovanov complexes (\cite{ArbitraryCoeffs} Lemma 1). Notice that the virtual Stevedore is Z-equivalent to the figure-eight knot. Hence, since the Rasmussen invariant for the figure-eight knot is zero we have that the the Rasmussen invariant for the virtual Stevedore is also zero.

\begin{thm}[Generalization of Rasmussen \cite{Rasmussen} Theorem 4]\label{thm:PositiveGenus}
If K is a positive virtual knot and $g_s(K)$ as in Theorem \ref{thm:Ras1}, then $g_s(K)=\displaystyle\frac{(-r + n + 1)}{2}$ where $r$ is the number of virtual Seifert circles for K and $n$ is the number of classical crossings for K.
\end{thm}

\begin{proof}
By Lemma \ref{prop:PosRas} the Rasmussen invariant gives a lower bound. Lemma \ref{lem:SeifertGenus} shows that the lower bound is realized.
\end{proof}

\begin{figure}[h!]
\centering
    \includegraphics[width=.8\textwidth]{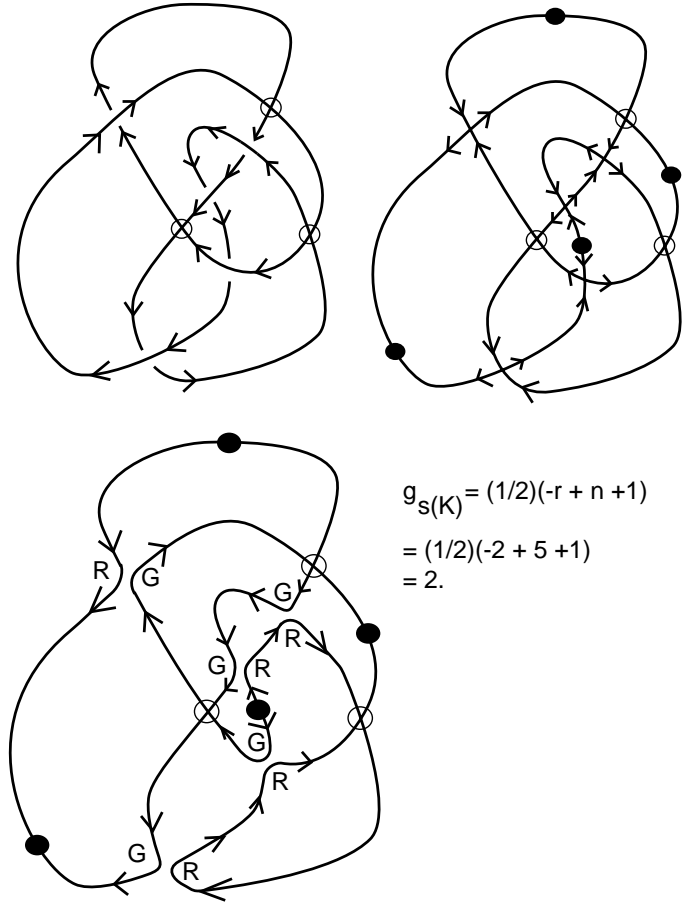}
    \caption{Example calculation for a positive knot}
\label{fig:LeeStateKnot}
\end{figure}

\begin{exa}
The entire process of calculating the genus of a positive knot can be depicted as in Figure \ref{fig:LeeStateKnot}. The upper left diagram depicts the virtual knot and the upper right diagram uses the canonical source-sink orientations to place cut loci on the semi-arcs of the knot diagram. After preforming the oriented smoothing (i.e. the all-A state for the knot diagram) we have one of the canonical generator for Lee's homology as shown in the bottom left, where the second canonical generator comes from swapping labels $\mathbf{r}$ and $\mathbf{g}$. Applying Theorem \ref{thm:PositiveGenus} as in the lower right, we see that $g_s(K)=2$.
\end{exa}

Finally, an interesting question is: ``Can the extension from the category of classical knots to virtual knots lower the slice genus?'' While we cannot give a complete answer, for many classes of knots we can say this is false. In particular, since the Rasmussen invariant presented here agrees with the Rasmussen invariant as originally defined for classical knots, any classical knot whose slice genus equals its Rasmussen invariant has the same slice genus in the virtual category. For instance, any (p,q)-torus knot has slice genus $\displaystyle\frac{(p-1)(q-1)}{2}$ in the virtual category.

\newpage
\appendix
\section{A Spanning Tree Expansion for Global Orders}\label{App:AppendixA}

To calculate the Khovanov homology of a virtual link diagram, we use a method called global order propagation. This insures that the global and local orders are consistent through out the calculation. We summarize this method as follows:
\begin{enumerate}
\item Choose an arbitrary order of the components of the all $A$-state.
\item Choose a spanning tree for the Khovanov complex
\item Use the rules given below to propagate the global order along the spanning tree.
\end{enumerate}

The independence of the above choices is clear. Changing the order on the all $A$-state induces a permutation between the two choices of orders, which in turn determines an isomorphism between the two complexes. Similarly, a change in spanning trees is also related by a permutation. This is inherent in the proof of anti-commutativity for each face.
In the notation below, bracketed numbers, such as $ [a] $, indicate the global order while numbers at smoothings indicate the local order (obtained from the source-sink notation in Figure \ref{fig:CanonicalOrientation}).

The following rules determine the propagation of the global ordering
\begin{figure} [htb] \centering
\begin{subfigure} [b] {0.5\linewidth} \centering
\def\svgwidth{2.5 in}
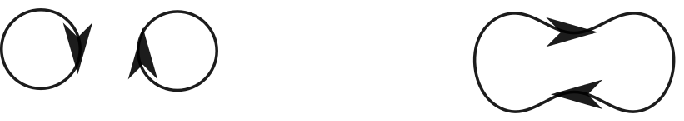
\caption{Under multiplication}
\end{subfigure}
\begin{subfigure} [b]   {0.5\linewidth} \centering
\def\svgwidth{2.5 in}
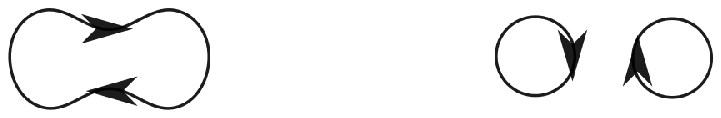
\caption{Under co-multiplication}
\end{subfigure}
\begin{subfigure} [b]  {0.5\linewidth} \centering
\def\svgwidth{2.5 in}
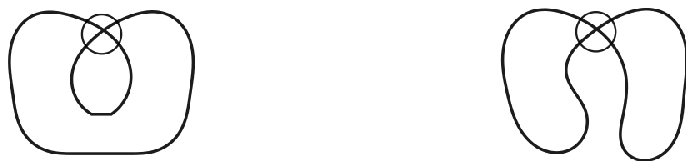
\caption{Under the $ \eta$ map}
\end{subfigure}
\caption{Propagating the basis ordering.}
\label{fig:rules}
\end{figure}

In these rules, we see that multiplication combines two cycles with orders $[a]$ and $[b]$. On the right hand side the order is determined by having the label $[a]$ persist on the new cycle, and all cycles having labels with values greater than $ b $ get reduced by $1$.  (Note that it is possible that the label $[a] $ is also reduced if $a > b$.)
Co-multiplication introduces cycles from a single cycle labeled $[a]$. On the right hand side, global labels greater than $a$ are increased by $1$. Then the cycle locally labeled $1$ receives the global label $[a]$ and the cycle locally labeled $2$ receives the global label $[a+1]$. In the case of the $ \eta $ map, the global label persists and no other change is made to the global order.

To show anti-commutativity and independence of the spanning trees, we consider the spanning trees along both sides of the face of a square. We use the following rules to determine
pre- and post-compositions along the differential.

\begin{figure}[htb]
\centering
\def\svgwidth{1.00 in} 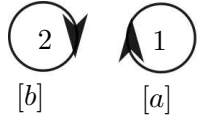
\caption{Two components with local and global basis orders}
\label{fig:preposttwocomponent}
\end{figure}
In the two component case, we compare the order of the global labels $ [a] $ and $[b]$ with the local labels. Here, we assume that the global label $[a]$ (respectively $[b]$) is associated with the local cycle $1$ (respectively the local label $2$). We consider the site $\tau$ in the state $s$.
The pre- or post-composition transposition, denoted $ P_{\tau}(s)$, is:
\begin{equation*}
 P_{\tau}(s) =
 \begin{cases}
(-1)^{a + b +1}    \text{ if } a < b  \\
(-1)^{a + b}   \text{ if } a > b
\end{cases}
\end{equation*}

\begin{figure} [htb] \centering
\def\svgwidth{1.00 in} 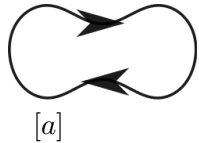
\caption{One component with global label $[a]$}
\label{fig:prepostonecomponent}
\end{figure}

In the single component case, the pre- and post-composition $ P_{\tau}(s) $ depends only on the value of
the global order:
\begin{equation*} P_{\tau}(s) = (-1)^{a +1 }.  \end{equation*}
Hence:
\[ \partial_{\tau}(s) =  P_{\tau}(s) \circ \partial_{\tau} \circ P_{\tau}(s')
\]
 where $\partial_{\tau}$ is the multiplication, comultiplication or single-cycle smoothing map occurring at $\tau$.

We compute two examples of squares obtained from the essential atoms. We leave the remaining cases as an exercise for the reader.

\begin{figure}[htb]
\centering
\def\svgwidth{1.00 in} 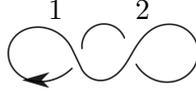
\caption{Example 1}
\label{fig:unknot2}
\end{figure}

We examine the complex (obtained from the two crossing unknot shown in Figure \ref{fig:unknot2}) in Figure \ref{fig:square1} . In this example, we assume that $ a < b$ in the global ordering.
We compute the two compositions
$ \partial_{1}(s') \circ \partial_{2}(s) $ and
$ \partial_{2}(s'') \circ \partial_{1}(s) $
and compare the results.
The global basis labels have been propagated using the rules shown in Figure \ref{fig:rules} and so there are two global basis orderings indicated in the last state.
We must multiply by $(-1)^{a + b +1} $ to move from one global basis to the other; we are required to exchange the basis labels $[a]$ and $[b-1]$ through transpositions until the global orderings are in agreement.

\begin{figure} [htb]
\centering
\[
\xymatrix{ \scalebox{0.8}{ 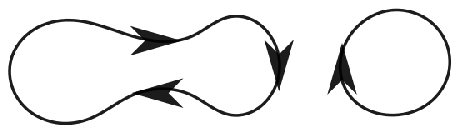} \ar[d]^-*+{\partial_{2}} \ar[r]^-*+{\partial_{1}} &\scalebox{0.8}{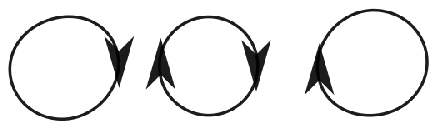}\ar[d]^-*+{\partial_{2}} \\
\scalebox{0.8}{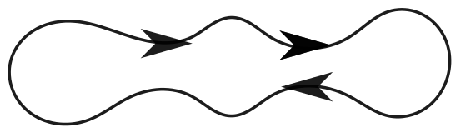} \ar[r]^-*+{\partial_{1}} &\scalebox{0.8}{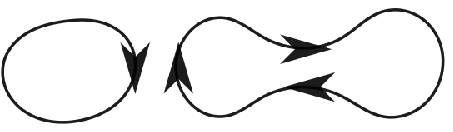}
}
\]
\caption{Anti-commuting square}
\label{fig:square1}
\end{figure}

The map $ \partial_{1}(s) $ indicates a $ \Delta $ map and its associated pre- and post compositions. The $ \Delta $ map acts on the loop with global order $[a]$ so that the pre-composition is $ (-1)^{a+1} $. The post-composition has the form $ (-1)^{ a + (a+1) +1} $ since the loop $[a]$ splits into two loops; giving rise to the labels $[a]$ and $ [a+1]$.  This necessitates an adjust in the global label of the other loop to $ [b+1]$.
Consequently,
\begin{equation} \label{arc1}
 \partial_{1}(s) = (-1)^{a + (a+1) + 1} \Delta (-1)^{a+1} = (-1) ^{a+1} \Delta .
\end{equation}
 We examine the diagram and determine the pre- and post composition for
$ \partial_{2}(s'')$:
\begin{align*}
\mathrm{pre:} (-1)^{(a+1) + (b+1)} \\
\mathrm{post:}(-1)^{b+1}
\end{align*}
Note that $ \partial_{2}(s'')$ involve multiplication, joining the loops labeled
$ [a+1]$ and $[b+1]$. One loop is eliminated and the loop labeled $[b+1]$ has local ordering $1$; the new loop is given ordering $[b]$ and labeling of all loops lower in the order is reduced.
Hence:
\begin{equation} \label{arc2}
 \partial_{2}(s'') = (-1)^{b+ 1} m  (-1)^{(a+1) + (b+1) } = (-1) ^{a+1} \Delta .
\end{equation}
Combining  Equations \ref{arc1} and \ref{arc2}, we determine that
\begin{equation} \label{upper}
\partial_{2}(s'') \circ \partial_{1}(s) = (-1)^0 m \circ \Delta.
\end{equation}
Now, we consider $ \partial_{2}(s) $ which acts on loops labeled $[a]$ and $[b]$.
Since the loops are reverse ordered with regard to the local order, the
pre-composition map is $ (-1)^{a + b}$. The post-composition map is determined by the global label $ [b-1]$ (inherited from the loop with local ordering $1$) and has value
$(-1)^{b} $.
We determine that
\begin{equation} \label{arc3}
\partial_{2}(s) = (-1)^{a+b} m (-1)^{b} = (-1)^a m.
\end{equation}
We evaluate the pre- and post-composition maps for $ \partial_{1}(s') $, noting that the sequence of maps has produced a global basis different from the other composition of maps.
The pre-composition map is $ (-1)^b$ since $\partial_{1}(s') $ involves the co-multiplication map. The post-composition map is $ (-1)^{ (b-1) + b +1} $ since the loops are globally labeled $[b-1]$ and $[b]$ and are not reverse ordered with regard to the local ordering.
We obtain
\begin{equation} \label{arc4}
\partial_{1}(s') = (-1)^{b} \Delta (-1)^{(b-1) + b + 1} = (-1)^b  \Delta
\end{equation}
Combining Equations \ref{arc3} and \ref{arc4}, we see that
\begin{equation}
\partial_{1}(s') \circ \partial_{2}(s)  = (-1)^{a+b}  \Delta \circ m.
\end{equation}

After applying the transition maps required to bring the global orderings into agreement,
we determine
\begin{equation}
(-1)^{a + (b-1) +1} \partial_{1}(s') \circ \partial_{2}(s) = (-1)^1  \Delta \circ m.
\end{equation}
so that the square anti-commutes.

\begin{figure}[htb]
\centering
\def\svgwidth{1.00 in} 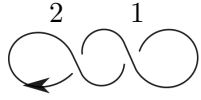
\caption{Example 2}
\label{fig:tripleloop}
\end{figure}

\begin{figure} [htb]
\centering
\[
\xymatrix{ \scalebox{0.8}{ 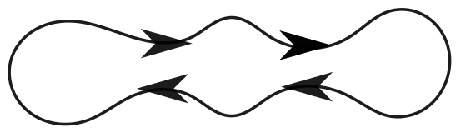} \ar[d]^-*+{\partial_{2}} \ar[r]^-*+{\partial_{1}} &\scalebox{0.8}{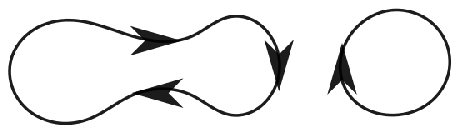}\ar[d]^-*+{\partial_{2}} \\
\scalebox{0.8}{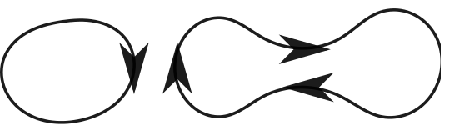} \ar[r]^-*+{\partial_{1}} &\scalebox{0.8}{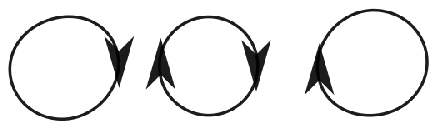}
}
\]
\caption{Anti-commuting square}
\label{fig:square2}
\end{figure}

In Figure \ref{fig:tripleloop}, we see a two crossing diagram that leads to one of the most complicated interactions between the local order and the global order as shown in Figure \ref{fig:square2}. In this square, the final state contains three loops and the labeling obtained by the propagation rules is different for all three loops. Here, we assume that the spanning tree agrees with the upper level and hence the local order on the upper level agrees the global order. As a result, the lower global order will need to be adjusted through a series of transpositions of labels.

In the initial state, there is a single loop with the global order label
$[a]$. The maps $ \partial_{1}(s) $ and $ \partial_{2}(s) $ both involve the co-multiplication,
so that the pre-composition map is $ (-1)^{a+1}$.
We determine that
\begin{equation}\label{2arc1}
\partial_{1}(s) = (-1)^{a + (a+1) +1 }  \Delta (-1)^{a+1}
\end{equation}
The propagation rules induce the illustrated global order.
Next, we compute that
\begin{equation} \label{2arc2}
\partial_{2}(s'') = (-1)^{ (a+1) + (a+2) + 1} \Delta (-1)^{(a+1) + 1}  .
\end{equation}
Combining Equations \ref{2arc1} and \ref{2arc2}, we determine that
\begin{equation*}
\partial_{2}(s'') \circ \partial_{1}(s) = (-1)^1  \Delta \circ \Delta.
\end{equation*}
Next,
\begin{equation}\label{2arc3}
\partial_{2}(s) = (-1)^{a + (a+1)+1}   \Delta (-1)^{a+1}
\end{equation}
and
\begin{equation}\label{2arc4}
\partial_{1}(s') = (-1)^{(a+1) + (a+2) + 1} \Delta (-1)^{(a+1) +1}.
\end{equation}
Using Equations \ref{2arc3} and \ref{2arc4}:
\begin{equation}
\partial_{1}(s') \circ \partial_{2}(s) = (-1) \Delta \circ \Delta.
\end{equation}
We note that the global order on the loops (from left to right) that was propagated by
the upper maps is $ [a], [ a+2],[a+1] $. The lower ordering is: $ [a+1], [a+2],[a] $.
We determine that a sequence of  three transpositions corrects the global order.
Hence the lower maps are
\begin{equation}
(-1)^3 \partial_{1}(s') \circ \partial_{2}(s) = \Delta \circ \Delta
\end{equation}
As a result, the square obtained from the diagram shown in Figure \ref{fig:tripleloop} anti-commutes.

\bibliographystyle{abbrv}
\bibliography{KhovanovLeeVirtuals}

\theendnotes

\end{document}

%% file: khov2.eps_tex
\begingroup%
  \makeatletter%
  \providecommand\color[2][]{%
    \errmessage{(Inkscape) Color is used for the text in Inkscape, but the package 'color.sty' is not loaded}%
    \renewcommand\color[2][]{}%
  }%
  \providecommand\transparent[1]{%
    \errmessage{(Inkscape) Transparency is used (non-zero) for the text in Inkscape, but the package 'transparent.sty' is not loaded}%
    \renewcommand\transparent[1]{}%
  }%
  \providecommand\rotatebox[2]{#2}%
  \ifx\svgwidth\undefined%
    \setlength{\unitlength}{234.67097182bp}%
    \ifx\svgscale\undefined%
      \relax%
    \else%
      \setlength{\unitlength}{\unitlength * \real{\svgscale}}%
    \fi%
  \else%
    \setlength{\unitlength}{\svgwidth}%
  \fi%
  \global\let\svgwidth\undefined%
  \global\let\svgscale\undefined%
  \makeatother%
  \begin{picture}(1,0.99238459)%
    \put(0,0){\includegraphics[width=\unitlength]{khov2.eps}}%
    \put(0.05853548,0.71098371){\color[rgb]{0,0,0}\makebox(0,0)[lb]{\smash{$a$}}}%
    \put(0.88900397,0.90305392){\color[rgb]{0,0,0}\makebox(0,0)[lb]{\smash{$a$}}}%
    \put(0.06985494,0.50857942){\color[rgb]{0,0,0}\makebox(0,0)[lb]{\smash{$\Delta$}}}%
    \put(0.41170371,0.14917075){\color[rgb]{0,0,0}\makebox(0,0)[lb]{\smash{$m$}}}%
    \put(0.15106082,0.11854838){\color[rgb]{0,0,0}\makebox(0,0)[lb]{\smash{$a_2$}}}%
    \put(0.08228973,0.19102313){\color[rgb]{0,0,0}\makebox(0,0)[lb]{\smash{$a_1$}}}%
    \put(0.69942783,0.50116144){\color[rgb]{0,0,0}\makebox(0,0)[lb]{\smash{$\eta$}}}%
    \put(0.41395219,0.75786847){\color[rgb]{0,0,0}\makebox(0,0)[lb]{\smash{$\eta$}}}%
  \end{picture}%
\endgroup%

%% file: knotswith2crossings.eps_tex
\begingroup%
  \makeatletter%
  \providecommand\color[2][]{%
    \errmessage{(Inkscape) Color is used for the text in Inkscape, but the package 'color.sty' is not loaded}%
    \renewcommand\color[2][]{}%
  }%
  \providecommand\transparent[1]{%
    \errmessage{(Inkscape) Transparency is used (non-zero) for the text in Inkscape, but the package 'transparent.sty' is not loaded}%
    \renewcommand\transparent[1]{}%
  }%
  \providecommand\rotatebox[2]{#2}%
  \ifx\svgwidth\undefined%
    \setlength{\unitlength}{300.47243652bp}%
    \ifx\svgscale\undefined%
      \relax%
    \else%
      \setlength{\unitlength}{\unitlength * \real{\svgscale}}%
    \fi%
  \else%
    \setlength{\unitlength}{\svgwidth}%
  \fi%
  \global\let\svgwidth\undefined%
  \global\let\svgscale\undefined%
  \makeatother%
  \begin{picture}(1,0.66037736)%
    \put(0,0){\includegraphics[width=\unitlength]{knotswith2crossings.eps}}%
  \end{picture}%
\endgroup%

%% file: essentialatomsversion2.eps_tex
\begingroup%
  \makeatletter%
  \providecommand\color[2][]{%
    \errmessage{(Inkscape) Color is used for the text in Inkscape, but the package 'color.sty' is not loaded}%
    \renewcommand\color[2][]{}%
  }%
  \providecommand\transparent[1]{%
    \errmessage{(Inkscape) Transparency is used (non-zero) for the text in Inkscape, but the package 'transparent.sty' is not loaded}%
    \renewcommand\transparent[1]{}%
  }%
  \providecommand\rotatebox[2]{#2}%
  \ifx\svgwidth\undefined%
    \setlength{\unitlength}{283.03178711bp}%
    \ifx\svgscale\undefined%
      \relax%
    \else%
      \setlength{\unitlength}{\unitlength * \real{\svgscale}}%
    \fi%
  \else%
    \setlength{\unitlength}{\svgwidth}%
  \fi%
  \global\let\svgwidth\undefined%
  \global\let\svgscale\undefined%
  \makeatother%
  \begin{picture}(1,0.79082569)%
    \put(0,0){\includegraphics[width=\unitlength]{essentialatomsversion2.eps}}%
  \end{picture}%
\endgroup%

%% file: crossing.eps_tex
\begingroup%
  \makeatletter%
  \providecommand\color[2][]{%
    \errmessage{(Inkscape) Color is used for the text in Inkscape, but the package 'color.sty' is not loaded}%
    \renewcommand\color[2][]{}%
  }%
  \providecommand\transparent[1]{%
    \errmessage{(Inkscape) Transparency is used (non-zero) for the text in Inkscape, but the package 'transparent.sty' is not loaded}%
    \renewcommand\transparent[1]{}%
  }%
  \providecommand\rotatebox[2]{#2}%
  \ifx\svgwidth\undefined%
    \setlength{\unitlength}{29.80966435bp}%
    \ifx\svgscale\undefined%
      \relax%
    \else%
      \setlength{\unitlength}{\unitlength * \real{\svgscale}}%
    \fi%
  \else%
    \setlength{\unitlength}{\svgwidth}%
  \fi%
  \global\let\svgwidth\undefined%
  \global\let\svgscale\undefined%
  \makeatother%
  \begin{picture}(1,0.98539666)%
    \put(0,0){\includegraphics[width=\unitlength]{crossing.eps}}%
  \end{picture}%
\endgroup%

%% file: 90degreecrossing.eps_tex
\begingroup%
  \makeatletter%
  \providecommand\color[2][]{%
    \errmessage{(Inkscape) Color is used for the text in Inkscape, but the package 'color.sty' is not loaded}%
    \renewcommand\color[2][]{}%
  }%
  \providecommand\transparent[1]{%
    \errmessage{(Inkscape) Transparency is used (non-zero) for the text in Inkscape, but the package 'transparent.sty' is not loaded}%
    \renewcommand\transparent[1]{}%
  }%
  \providecommand\rotatebox[2]{#2}%
  \ifx\svgwidth\undefined%
    \setlength{\unitlength}{29.37434376bp}%
    \ifx\svgscale\undefined%
      \relax%
    \else%
      \setlength{\unitlength}{\unitlength * \real{\svgscale}}%
    \fi%
  \else%
    \setlength{\unitlength}{\svgwidth}%
  \fi%
  \global\let\svgwidth\undefined%
  \global\let\svgscale\undefined%
  \makeatother%
  \begin{picture}(1,1.01481976)%
    \put(0,0){\includegraphics[width=\unitlength]{90degreecrossing.eps}}%
  \end{picture}%
\endgroup%

%% file: khovcomplex2crossingunknotorientedordered.eps_tex
\begingroup%
  \makeatletter%
  \providecommand\color[2][]{%
    \errmessage{(Inkscape) Color is used for the text in Inkscape, but the package 'color.sty' is not loaded}%
    \renewcommand\color[2][]{}%
  }%
  \providecommand\transparent[1]{%
    \errmessage{(Inkscape) Transparency is used (non-zero) for the text in Inkscape, but the package 'transparent.sty' is not loaded}%
    \renewcommand\transparent[1]{}%
  }%
  \providecommand\rotatebox[2]{#2}%
  \ifx\svgwidth\undefined%
    \setlength{\unitlength}{268.72776182bp}%
    \ifx\svgscale\undefined%
      \relax%
    \else%
      \setlength{\unitlength}{\unitlength * \real{\svgscale}}%
    \fi%
  \else%
    \setlength{\unitlength}{\svgwidth}%
  \fi%
  \global\let\svgwidth\undefined%
  \global\let\svgscale\undefined%
  \makeatother%
  \begin{picture}(1,0.72320996)%
    \put(0,0){\includegraphics[width=\unitlength]{khovcomplex2crossingunknotorientedordered.eps}}%
    \put(0.44662563,0.54989602){\color[rgb]{0,0,0}\makebox(0,0)[lb]{\smash{$ \longrightarrow $
}}}%
    \put(0.52365599,0.49996311){\color[rgb]{0,0,0}\makebox(0,0)[lb]{\smash{$ \eta $}}}%
    \put(0.43596599,0.15759264){\color[rgb]{0,0,0}\makebox(0,0)[lb]{\smash{$ \longrightarrow $}}}%
    \put(0.52109791,0.10379461){\color[rgb]{0,0,0}\makebox(0,0)[lb]{\smash{$ m $ }}}%
    \put(0.20599497,0.35664314){\color[rgb]{0,0,0}\makebox(0,0)[lb]{\smash{$ \downarrow $ }}}%
    \put(0.14335294,0.3577383){\color[rgb]{0,0,0}\makebox(0,0)[lb]{\smash{$ \Delta$}}}%
    \put(0.73305186,0.36266568){\color[rgb]{0,0,0}\makebox(0,0)[lb]{\smash{$ \downarrow$}}}%
    \put(0.6640045,0.36437336){\color[rgb]{0,0,0}\makebox(0,0)[lb]{\smash{$ \eta $}}}%
    \put(0.30648649,0.17305288){\color[rgb]{0,0,0}\makebox(0,0)[lb]{\smash{$[a]$}}}%
    \put(-0.00271825,0.213636){\color[rgb]{0,0,0}\makebox(0,0)[lb]{\smash{$[a+1]$}}}%
    \put(0.56544547,0.213636){\color[rgb]{0,0,0}\makebox(0,0)[lb]{\smash{$[a]$}}}%
    \put(0.54998525,0.64845523){\color[rgb]{0,0,0}\makebox(0,0)[lb]{\smash{$[a]$}}}%
    \put(0.21642994,0.69832999){\color[rgb]{0,0,0}\makebox(0,0)[lb]{\smash{$1$}}}%
    \put(0.1805247,0.54179506){\color[rgb]{0,0,0}\makebox(0,0)[lb]{\smash{$1$}}}%
    \put(0.22883794,0.27317345){\color[rgb]{0,0,0}\makebox(0,0)[lb]{\smash{$1$}}}%
    \put(0.20564759,0.1340313){\color[rgb]{0,0,0}\makebox(0,0)[lb]{\smash{$1$}}}%
    \put(0.65206192,0.62489387){\color[rgb]{0,0,0}\makebox(0,0)[lb]{\smash{$1$}}}%
    \put(0.65979207,0.54372781){\color[rgb]{0,0,0}\makebox(0,0)[lb]{\smash{$1$}}}%
    \put(0.67138724,0.2094){\color[rgb]{0,0,0}\makebox(0,0)[lb]{\smash{$1$}}}%
    \put(0.69651014,0.13789626){\color[rgb]{0,0,0}\makebox(0,0)[lb]{\smash{$1$}}}%
    \put(0.06376009,0.63262397){\color[rgb]{0,0,0}\makebox(0,0)[lb]{\smash{$[a]$}}}%
    \put(0.22690542,0.6210289){\color[rgb]{0,0,0}\makebox(0,0)[lb]{\smash{$2$}}}%
    \put(0.29840904,0.55145774){\color[rgb]{0,0,0}\makebox(0,0)[lb]{\smash{$2$}}}%
    \put(0.232703,0.21519772){\color[rgb]{0,0,0}\makebox(0,0)[lb]{\smash{$2$}}}%
    \put(0.28874638,0.09924572){\color[rgb]{0,0,0}\makebox(0,0)[lb]{\smash{$2$}}}%
    \put(0.79120411,0.09731333){\color[rgb]{0,0,0}\makebox(0,0)[lb]{\smash{$2$}}}%
    \put(0.78733905,0.62875901){\color[rgb]{0,0,0}\makebox(0,0)[lb]{\smash{$2$}}}%
    \put(0.78540657,0.5456602){\color[rgb]{0,0,0}\makebox(0,0)[lb]{\smash{$2$}}}%
    \put(0.79313659,0.20746743){\color[rgb]{0,0,0}\makebox(0,0)[lb]{\smash{$2$}}}%
  \end{picture}%
\endgroup%

%% file: rulesmult.eps_tex
\begingroup%
  \makeatletter%
  \providecommand\color[2][]{%
    \errmessage{(Inkscape) Color is used for the text in Inkscape, but the package 'color.sty' is not loaded}%
    \renewcommand\color[2][]{}%
  }%
  \providecommand\transparent[1]{%
    \errmessage{(Inkscape) Transparency is used (non-zero) for the text in Inkscape, but the package 'transparent.sty' is not loaded}%
    \renewcommand\transparent[1]{}%
  }%
  \providecommand\rotatebox[2]{#2}%
  \ifx\svgwidth\undefined%
    \setlength{\unitlength}{194.45394363bp}%
    \ifx\svgscale\undefined%
      \relax%
    \else%
      \setlength{\unitlength}{\unitlength * \real{\svgscale}}%
    \fi%
  \else%
    \setlength{\unitlength}{\svgwidth}%
  \fi%
  \global\let\svgwidth\undefined%
  \global\let\svgscale\undefined%
  \makeatother%
  \begin{picture}(1,0.21295471)%
    \put(0,0){\includegraphics[width=\unitlength]{rulesmult.eps}}%
    \put(0.40874206,0.12260323){\color[rgb]{0,0,0}\makebox(0,0)[lb]{\smash{$ \longrightarrow $}}}%
    \put(0.01348132,0.03981214){\color[rgb]{0,0,0}\makebox(0,0)[lb]{\smash{$[b]$}}}%
    \put(0.22714222,0.03580615){\color[rgb]{0,0,0}\makebox(0,0)[lb]{\smash{$[a]$}}}%
    \put(0.78931411,0.00665124){\color[rgb]{0,0,0}\makebox(0,0)[lb]{\smash{$[a]$}}}%
    \put(0.05087085,0.13574056){\color[rgb]{0,0,0}\makebox(0,0)[lb]{\smash{2
}}}%
    \put(0.24583058,0.13350222){\color[rgb]{0,0,0}\makebox(0,0)[lb]{\smash{1
}}}%
  \end{picture}%
\endgroup%

%% file: rulesdelta.eps_tex
\begingroup%
  \makeatletter%
  \providecommand\color[2][]{%
    \errmessage{(Inkscape) Color is used for the text in Inkscape, but the package 'color.sty' is not loaded}%
    \renewcommand\color[2][]{}%
  }%
  \providecommand\transparent[1]{%
    \errmessage{(Inkscape) Transparency is used (non-zero) for the text in Inkscape, but the package 'transparent.sty' is not loaded}%
    \renewcommand\transparent[1]{}%
  }%
  \providecommand\rotatebox[2]{#2}%
  \ifx\svgwidth\undefined%
    \setlength{\unitlength}{219.30655768bp}%
    \ifx\svgscale\undefined%
      \relax%
    \else%
      \setlength{\unitlength}{\unitlength * \real{\svgscale}}%
    \fi%
  \else%
    \setlength{\unitlength}{\svgwidth}%
  \fi%
  \global\let\svgwidth\undefined%
  \global\let\svgscale\undefined%
  \makeatother%
  \begin{picture}(1,0.17412962)%
    \put(0,0){\includegraphics[width=\unitlength]{rulesdelta.eps}}%
    \put(-0.00143029,0.04644918){\color[rgb]{0,0,0}\makebox(0,0)[lb]{\smash{$[a]$}}}%
    \put(0.43803654,0.09459924){\color[rgb]{0,0,0}\makebox(0,0)[lb]{\smash{$ \longrightarrow$}}}%
    \put(0.91263478,0.00797435){\color[rgb]{0,0,0}\makebox(0,0)[lb]{\smash{$[a]$}}}%
    \put(0.75633908,0.09658392){\color[rgb]{0,0,0}\makebox(0,0)[lb]{\smash{2
}}}%
    \put(0.92920529,0.09459924){\color[rgb]{0,0,0}\makebox(0,0)[lb]{\smash{1
}}}%
    \put(0.69674376,0.0144809){\color[rgb]{0,0,0}\makebox(0,0)[lb]{\smash{$[a+1]$}}}%
  \end{picture}%
\endgroup%

%% file: ruleseta.eps_tex
\begingroup%
  \makeatletter%
  \providecommand\color[2][]{%
    \errmessage{(Inkscape) Color is used for the text in Inkscape, but the package 'color.sty' is not loaded}%
    \renewcommand\color[2][]{}%
  }%
  \providecommand\transparent[1]{%
    \errmessage{(Inkscape) Transparency is used (non-zero) for the text in Inkscape, but the package 'transparent.sty' is not loaded}%
    \renewcommand\transparent[1]{}%
  }%
  \providecommand\rotatebox[2]{#2}%
  \ifx\svgwidth\undefined%
    \setlength{\unitlength}{211.73225565bp}%
    \ifx\svgscale\undefined%
      \relax%
    \else%
      \setlength{\unitlength}{\unitlength * \real{\svgscale}}%
    \fi%
  \else%
    \setlength{\unitlength}{\svgwidth}%
  \fi%
  \global\let\svgwidth\undefined%
  \global\let\svgscale\undefined%
  \makeatother%
  \begin{picture}(1,0.20975164)%
    \put(0,0){\includegraphics[width=\unitlength]{ruleseta.eps}}%
    \put(-0.00148146,0.03484557){\color[rgb]{0,0,0}\makebox(0,0)[lb]{\smash{$[a]$}}}%
    \put(0.4817085,0.0937113){\color[rgb]{0,0,0}\makebox(0,0)[lb]{\smash{$\longrightarrow$}}}%
    \put(0.67056967,0.03811591){\color[rgb]{0,0,0}\makebox(0,0)[lb]{\smash{$[a]$}}}%
  \end{picture}%
\endgroup%

%% file: preposttwocomponent.eps_tex
\begingroup%
  \makeatletter%
  \providecommand\color[2][]{%
    \errmessage{(Inkscape) Color is used for the text in Inkscape, but the package 'color.sty' is not loaded}%
    \renewcommand\color[2][]{}%
  }%
  \providecommand\transparent[1]{%
    \errmessage{(Inkscape) Transparency is used (non-zero) for the text in Inkscape, but the package 'transparent.sty' is not loaded}%
    \renewcommand\transparent[1]{}%
  }%
  \providecommand\rotatebox[2]{#2}%
  \ifx\svgwidth\undefined%
    \setlength{\unitlength}{63.32846837bp}%
    \ifx\svgscale\undefined%
      \relax%
    \else%
      \setlength{\unitlength}{\unitlength * \real{\svgscale}}%
    \fi%
  \else%
    \setlength{\unitlength}{\svgwidth}%
  \fi%
  \global\let\svgwidth\undefined%
  \global\let\svgscale\undefined%
  \makeatother%
  \begin{picture}(1,0.57156068)%
    \put(0,0){\includegraphics[width=\unitlength]{preposttwocomponent.eps}}%
    \put(0.04139522,0.03991605){\color[rgb]{0,0,0}\makebox(0,0)[lb]{\smash{$[b]$}}}%
    \put(0.69745414,0.0276152){\color[rgb]{0,0,0}\makebox(0,0)[lb]{\smash{$[a]$}}}%
    \put(0.15620211,0.33447018){\color[rgb]{0,0,0}\makebox(0,0)[lb]{\smash{2
}}}%
    \put(0.75483794,0.32759722){\color[rgb]{0,0,0}\makebox(0,0)[lb]{\smash{1
}}}%
  \end{picture}%
\endgroup%

%% file: prepostonecomponent.eps_tex
\begingroup%
  \makeatletter%
  \providecommand\color[2][]{%
    \errmessage{(Inkscape) Color is used for the text in Inkscape, but the package 'color.sty' is not loaded}%
    \renewcommand\color[2][]{}%
  }%
  \providecommand\transparent[1]{%
    \errmessage{(Inkscape) Transparency is used (non-zero) for the text in Inkscape, but the package 'transparent.sty' is not loaded}%
    \renewcommand\transparent[1]{}%
  }%
  \providecommand\rotatebox[2]{#2}%
  \ifx\svgwidth\undefined%
    \setlength{\unitlength}{58.21826379bp}%
    \ifx\svgscale\undefined%
      \relax%
    \else%
      \setlength{\unitlength}{\unitlength * \real{\svgscale}}%
    \fi%
  \else%
    \setlength{\unitlength}{\svgwidth}%
  \fi%
  \global\let\svgwidth\undefined%
  \global\let\svgscale\undefined%
  \makeatother%
  \begin{picture}(1,0.69833395)%
    \put(0,0){\includegraphics[width=\unitlength]{prepostonecomponent.eps}}%
    \put(0.13126559,0.03003917){\color[rgb]{0,0,0}\makebox(0,0)[lb]{\smash{$[a]$
}}}%
  \end{picture}%
\endgroup%

%% file: unknot2.eps_tex
\begingroup%
  \makeatletter%
  \providecommand\color[2][]{%
    \errmessage{(Inkscape) Color is used for the text in Inkscape, but the package 'color.sty' is not loaded}%
    \renewcommand\color[2][]{}%
  }%
  \providecommand\transparent[1]{%
    \errmessage{(Inkscape) Transparency is used (non-zero) for the text in Inkscape, but the package 'transparent.sty' is not loaded}%
    \renewcommand\transparent[1]{}%
  }%
  \providecommand\rotatebox[2]{#2}%
  \ifx\svgwidth\undefined%
    \setlength{\unitlength}{103.31009973bp}%
    \ifx\svgscale\undefined%
      \relax%
    \else%
      \setlength{\unitlength}{\unitlength * \real{\svgscale}}%
    \fi%
  \else%
    \setlength{\unitlength}{\svgwidth}%
  \fi%
  \global\let\svgwidth\undefined%
  \global\let\svgscale\undefined%
  \makeatother%
  \begin{picture}(1,0.43429516)%
    \put(0,0){\includegraphics[width=\unitlength]{unknot2.eps}}%
    \put(0.2107493,0.36957798){\color[rgb]{0,0,0}\makebox(0,0)[lb]{\smash{$1$}}}%
    \put(0.6652482,0.36957798){\color[rgb]{0,0,0}\makebox(0,0)[lb]{\smash{$2$}}}%
  \end{picture}%
\endgroup%

%% file: unknot2aa.eps_tex
\begingroup%
  \makeatletter%
  \providecommand\color[2][]{%
    \errmessage{(Inkscape) Color is used for the text in Inkscape, but the package 'color.sty' is not loaded}%
    \renewcommand\color[2][]{}%
  }%
  \providecommand\transparent[1]{%
    \errmessage{(Inkscape) Transparency is used (non-zero) for the text in Inkscape, but the package 'transparent.sty' is not loaded}%
    \renewcommand\transparent[1]{}%
  }%
  \providecommand\rotatebox[2]{#2}%
  \ifx\svgwidth\undefined%
    \setlength{\unitlength}{145.04191828bp}%
    \ifx\svgscale\undefined%
      \relax%
    \else%
      \setlength{\unitlength}{\unitlength * \real{\svgscale}}%
    \fi%
  \else%
    \setlength{\unitlength}{\svgwidth}%
  \fi%
  \global\let\svgwidth\undefined%
  \global\let\svgscale\undefined%
  \makeatother%
  \begin{picture}(1,0.3052266)%
    \put(0,0){\includegraphics[width=\unitlength]{unknot2aa.eps}}%
    \put(0.46330398,0.11457638){\color[rgb]{0,0,0}\makebox(0,0)[lb]{\smash{2}}}%
    \put(0.74258393,0.11815689){\color[rgb]{0,0,0}\makebox(0,0)[lb]{\smash{1}}}%
    \put(0.2699563,0.24347483){\color[rgb]{0,0,0}\makebox(0,0)[lb]{\smash{1}}}%
    \put(0.26637579,-0){\color[rgb]{0,0,0}\makebox(0,0)[lb]{\smash{2}}}%
    \put(-0.00216263,0.25779688){\color[rgb]{0,0,0}\makebox(0,0)[lb]{\smash{$[a]$}}}%
    \put(0.86790187,0.25063585){\color[rgb]{0,0,0}\makebox(0,0)[lb]{\smash{$[b]$}}}%
  \end{picture}%
\endgroup%

%% file: unknot2ab.eps_tex
\begingroup%
  \makeatletter%
  \providecommand\color[2][]{%
    \errmessage{(Inkscape) Color is used for the text in Inkscape, but the package 'color.sty' is not loaded}%
    \renewcommand\color[2][]{}%
  }%
  \providecommand\transparent[1]{%
    \errmessage{(Inkscape) Transparency is used (non-zero) for the text in Inkscape, but the package 'transparent.sty' is not loaded}%
    \renewcommand\transparent[1]{}%
  }%
  \providecommand\rotatebox[2]{#2}%
  \ifx\svgwidth\undefined%
    \setlength{\unitlength}{175.74803467bp}%
    \ifx\svgscale\undefined%
      \relax%
    \else%
      \setlength{\unitlength}{\unitlength * \real{\svgscale}}%
    \fi%
  \else%
    \setlength{\unitlength}{\svgwidth}%
  \fi%
  \global\let\svgwidth\undefined%
  \global\let\svgscale\undefined%
  \makeatother%
  \begin{picture}(1,0.41935483)%
    \put(0,0){\includegraphics[width=\unitlength]{unknot2ab.eps}}%
    \put(0.48387094,0.19354835){\color[rgb]{0,0,0}\makebox(0,0)[lb]{\smash{2}}}%
    \put(0.68209796,0.19379459){\color[rgb]{0,0,0}\makebox(0,0)[lb]{\smash{1}}}%
    \put(0.22580645,0.19354835){\color[rgb]{0,0,0}\makebox(0,0)[lb]{\smash{1}}}%
    \put(0.41935483,0.19354835){\color[rgb]{0,0,0}\makebox(0,0)[lb]{\smash{2}}}%
    \put(0.09677419,0.30017239){\color[rgb]{0,0,0}\makebox(0,0)[lb]{\smash{$[a]$}}}%
    \put(0.78552077,0.30312726){\color[rgb]{0,0,0}\makebox(0,0)[lb]{\smash{$[b+1]$}}}%
    \put(0.38660428,0.30903715){\color[rgb]{0,0,0}\makebox(0,0)[lb]{\smash{$[a+1]$}}}%
  \end{picture}%
\endgroup%

%% file: unknot2ba.eps_tex
\begingroup%
  \makeatletter%
  \providecommand\color[2][]{%
    \errmessage{(Inkscape) Color is used for the text in Inkscape, but the package 'color.sty' is not loaded}%
    \renewcommand\color[2][]{}%
  }%
  \providecommand\transparent[1]{%
    \errmessage{(Inkscape) Transparency is used (non-zero) for the text in Inkscape, but the package 'transparent.sty' is not loaded}%
    \renewcommand\transparent[1]{}%
  }%
  \providecommand\rotatebox[2]{#2}%
  \ifx\svgwidth\undefined%
    \setlength{\unitlength}{136.06298828bp}%
    \ifx\svgscale\undefined%
      \relax%
    \else%
      \setlength{\unitlength}{\unitlength * \real{\svgscale}}%
    \fi%
  \else%
    \setlength{\unitlength}{\svgwidth}%
  \fi%
  \global\let\svgwidth\undefined%
  \global\let\svgscale\undefined%
  \makeatother%
  \begin{picture}(1,0.40744271)%
    \put(0,0){\includegraphics[width=\unitlength]{unknot2ba.eps}}%
    \put(0.66666669,0.30438935){\color[rgb]{0,0,0}\makebox(0,0)[lb]{\smash{2}}}%
    \put(0.33555979,0.30057248){\color[rgb]{0,0,0}\makebox(0,0)[lb]{\smash{1}}}%
    \put(0.331743,0.04103051){\color[rgb]{0,0,0}\makebox(0,0)[lb]{\smash{2}}}%
    \put(0.04548346,0.31583966){\color[rgb]{0,0,0}\makebox(0,0)[lb]{\smash{$[b-1]$}}}%
    \put(0.66666669,0.04166672){\color[rgb]{0,0,0}\makebox(0,0)[lb]{\smash{1}}}%
  \end{picture}%
\endgroup%

%% file: unknot2bb.eps_tex
\begingroup%
  \makeatletter%
  \providecommand\color[2][]{%
    \errmessage{(Inkscape) Color is used for the text in Inkscape, but the package 'color.sty' is not loaded}%
    \renewcommand\color[2][]{}%
  }%
  \providecommand\transparent[1]{%
    \errmessage{(Inkscape) Transparency is used (non-zero) for the text in Inkscape, but the package 'transparent.sty' is not loaded}%
    \renewcommand\transparent[1]{}%
  }%
  \providecommand\rotatebox[2]{#2}%
  \ifx\svgwidth\undefined%
    \setlength{\unitlength}{128.00471662bp}%
    \ifx\svgscale\undefined%
      \relax%
    \else%
      \setlength{\unitlength}{\unitlength * \real{\svgscale}}%
    \fi%
  \else%
    \setlength{\unitlength}{\svgwidth}%
  \fi%
  \global\let\svgwidth\undefined%
  \global\let\svgscale\undefined%
  \makeatother%
  \begin{picture}(1,0.43659835)%
    \put(0,0){\includegraphics[width=\unitlength]{unknot2bb.eps}}%
    \put(0.66189513,0.3382281){\color[rgb]{0,0,0}\makebox(0,0)[lb]{\smash{2}}}%
    \put(0.17470836,0.19183533){\color[rgb]{0,0,0}\makebox(0,0)[lb]{\smash{1}}}%
    \put(0.44044658,0.19183533){\color[rgb]{0,0,0}\makebox(0,0)[lb]{\smash{2}}}%
    \put(-0.00245047,0.01467664){\color[rgb]{0,0,0}\makebox(0,0)[lb]{\smash{$[b-1]$}}}%
    \put(0.66189513,0.05896631){\color[rgb]{0,0,0}\makebox(0,0)[lb]{\smash{1}}}%
    \put(0.44788455,0.01366222){\color[rgb]{0,0,0}\makebox(0,0)[lb]{\smash{$[b]$}}}%
    \put(0.02189196,0.3706846){\color[rgb]{0,0,0}\makebox(0,0)[lb]{\smash{$[a]$}}}%
    \put(0.43977042,0.38285583){\color[rgb]{0,0,0}\makebox(0,0)[lb]{\smash{$[b]$}}}%
  \end{picture}%
\endgroup%

%% file: tripleloop.eps_tex
\begingroup%
  \makeatletter%
  \providecommand\color[2][]{%
    \errmessage{(Inkscape) Color is used for the text in Inkscape, but the package 'color.sty' is not loaded}%
    \renewcommand\color[2][]{}%
  }%
  \providecommand\transparent[1]{%
    \errmessage{(Inkscape) Transparency is used (non-zero) for the text in Inkscape, but the package 'transparent.sty' is not loaded}%
    \renewcommand\transparent[1]{}%
  }%
  \providecommand\rotatebox[2]{#2}%
  \ifx\svgwidth\undefined%
    \setlength{\unitlength}{103.30297097bp}%
    \ifx\svgscale\undefined%
      \relax%
    \else%
      \setlength{\unitlength}{\unitlength * \real{\svgscale}}%
    \fi%
  \else%
    \setlength{\unitlength}{\svgwidth}%
  \fi%
  \global\let\svgwidth\undefined%
  \global\let\svgscale\undefined%
  \makeatother%
  \begin{picture}(1,0.44266774)%
    \put(0,0){\includegraphics[width=\unitlength]{tripleloop.eps}}%
    \put(0.63863215,0.3779461){\color[rgb]{0,0,0}\makebox(0,0)[lb]{\smash{$1$}}}%
    \put(0.21299611,0.3779461){\color[rgb]{0,0,0}\makebox(0,0)[lb]{\smash{$2$}}}%
  \end{picture}%
\endgroup%

%% file: tripleloopaa.eps_tex
\begingroup%
  \makeatletter%
  \providecommand\color[2][]{%
    \errmessage{(Inkscape) Color is used for the text in Inkscape, but the package 'color.sty' is not loaded}%
    \renewcommand\color[2][]{}%
  }%
  \providecommand\transparent[1]{%
    \errmessage{(Inkscape) Transparency is used (non-zero) for the text in Inkscape, but the package 'transparent.sty' is not loaded}%
    \renewcommand\transparent[1]{}%
  }%
  \providecommand\rotatebox[2]{#2}%
  \ifx\svgwidth\undefined%
    \setlength{\unitlength}{136.06298828bp}%
    \ifx\svgscale\undefined%
      \relax%
    \else%
      \setlength{\unitlength}{\unitlength * \real{\svgscale}}%
    \fi%
  \else%
    \setlength{\unitlength}{\svgwidth}%
  \fi%
  \global\let\svgwidth\undefined%
  \global\let\svgscale\undefined%
  \makeatother%
  \begin{picture}(1,0.40744271)%
    \put(0,0){\includegraphics[width=\unitlength]{tripleloopaa.eps}}%
    \put(0.66666669,0.30438935){\color[rgb]{0,0,0}\makebox(0,0)[lb]{\smash{2}}}%
    \put(0.33555979,0.30057248){\color[rgb]{0,0,0}\makebox(0,0)[lb]{\smash{1}}}%
    \put(0.331743,0.04103051){\color[rgb]{0,0,0}\makebox(0,0)[lb]{\smash{2}}}%
    \put(0.04548346,0.31583966){\color[rgb]{0,0,0}\makebox(0,0)[lb]{\smash{$[a]$}}}%
    \put(0.66666669,0.04166672){\color[rgb]{0,0,0}\makebox(0,0)[lb]{\smash{1}}}%
  \end{picture}%
\endgroup%

%% file: tripleloopab.eps_tex
\begingroup%
  \makeatletter%
  \providecommand\color[2][]{%
    \errmessage{(Inkscape) Color is used for the text in Inkscape, but the package 'color.sty' is not loaded}%
    \renewcommand\color[2][]{}%
  }%
  \providecommand\transparent[1]{%
    \errmessage{(Inkscape) Transparency is used (non-zero) for the text in Inkscape, but the package 'transparent.sty' is not loaded}%
    \renewcommand\transparent[1]{}%
  }%
  \providecommand\rotatebox[2]{#2}%
  \ifx\svgwidth\undefined%
    \setlength{\unitlength}{145.04191828bp}%
    \ifx\svgscale\undefined%
      \relax%
    \else%
      \setlength{\unitlength}{\unitlength * \real{\svgscale}}%
    \fi%
  \else%
    \setlength{\unitlength}{\svgwidth}%
  \fi%
  \global\let\svgwidth\undefined%
  \global\let\svgscale\undefined%
  \makeatother%
  \begin{picture}(1,0.3052266)%
    \put(0,0){\includegraphics[width=\unitlength]{tripleloopab.eps}}%
    \put(0.46330398,0.11457638){\color[rgb]{0,0,0}\makebox(0,0)[lb]{\smash{2}}}%
    \put(0.74258393,0.11815689){\color[rgb]{0,0,0}\makebox(0,0)[lb]{\smash{1}}}%
    \put(0.2699563,0.24347483){\color[rgb]{0,0,0}\makebox(0,0)[lb]{\smash{1}}}%
    \put(0.26637579,-0){\color[rgb]{0,0,0}\makebox(0,0)[lb]{\smash{2}}}%
    \put(-0.00216263,0.25779688){\color[rgb]{0,0,0}\makebox(0,0)[lb]{\smash{$[a+1]$}}}%
    \put(0.86790187,0.25063585){\color[rgb]{0,0,0}\makebox(0,0)[lb]{\smash{$[a]$}}}%
  \end{picture}%
\endgroup%

%% file: tripleloopba.eps_tex
\begingroup%
  \makeatletter%
  \providecommand\color[2][]{%
    \errmessage{(Inkscape) Color is used for the text in Inkscape, but the package 'color.sty' is not loaded}%
    \renewcommand\color[2][]{}%
  }%
  \providecommand\transparent[1]{%
    \errmessage{(Inkscape) Transparency is used (non-zero) for the text in Inkscape, but the package 'transparent.sty' is not loaded}%
    \renewcommand\transparent[1]{}%
  }%
  \providecommand\rotatebox[2]{#2}%
  \ifx\svgwidth\undefined%
    \setlength{\unitlength}{127.67908368bp}%
    \ifx\svgscale\undefined%
      \relax%
    \else%
      \setlength{\unitlength}{\unitlength * \real{\svgscale}}%
    \fi%
  \else%
    \setlength{\unitlength}{\svgwidth}%
  \fi%
  \global\let\svgwidth\undefined%
  \global\let\svgscale\undefined%
  \makeatother%
  \begin{picture}(1,0.37859515)%
    \put(0,0){\includegraphics[width=\unitlength]{tripleloopba.eps}}%
    \put(0.66103283,0.27997402){\color[rgb]{0,0,0}\makebox(0,0)[lb]{\smash{2}}}%
    \put(0.17260353,0.13320789){\color[rgb]{0,0,0}\makebox(0,0)[lb]{\smash{1}}}%
    \put(0.43901949,0.13320789){\color[rgb]{0,0,0}\makebox(0,0)[lb]{\smash{2}}}%
    \put(0.66103283,0){\color[rgb]{0,0,0}\makebox(0,0)[lb]{\smash{1}}}%
    \put(0.01939739,0.3125133){\color[rgb]{0,0,0}\makebox(0,0)[lb]{\smash{$[a]$}}}%
    \put(0.43834161,0.32471556){\color[rgb]{0,0,0}\makebox(0,0)[lb]{\smash{$[a+1]$}}}%
  \end{picture}%
\endgroup%

%% file: tripleloopbb.eps_tex
\begingroup%
  \makeatletter%
  \providecommand\color[2][]{%
    \errmessage{(Inkscape) Color is used for the text in Inkscape, but the package 'color.sty' is not loaded}%
    \renewcommand\color[2][]{}%
  }%
  \providecommand\transparent[1]{%
    \errmessage{(Inkscape) Transparency is used (non-zero) for the text in Inkscape, but the package 'transparent.sty' is not loaded}%
    \renewcommand\transparent[1]{}%
  }%
  \providecommand\rotatebox[2]{#2}%
  \ifx\svgwidth\undefined%
    \setlength{\unitlength}{175.74803467bp}%
    \ifx\svgscale\undefined%
      \relax%
    \else%
      \setlength{\unitlength}{\unitlength * \real{\svgscale}}%
    \fi%
  \else%
    \setlength{\unitlength}{\svgwidth}%
  \fi%
  \global\let\svgwidth\undefined%
  \global\let\svgscale\undefined%
  \makeatother%
  \begin{picture}(1,0.41935483)%
    \put(0,0){\includegraphics[width=\unitlength]{tripleloopbb.eps}}%
    \put(0.48387094,0.19354835){\color[rgb]{0,0,0}\makebox(0,0)[lb]{\smash{2}}}%
    \put(0.68209796,0.19379459){\color[rgb]{0,0,0}\makebox(0,0)[lb]{\smash{1}}}%
    \put(0.22580645,0.19354835){\color[rgb]{0,0,0}\makebox(0,0)[lb]{\smash{1}}}%
    \put(0.41935483,0.19354835){\color[rgb]{0,0,0}\makebox(0,0)[lb]{\smash{2}}}%
    \put(0.09677419,0.30017239){\color[rgb]{0,0,0}\makebox(0,0)[lb]{\smash{$[a]$}}}%
    \put(0.78552077,0.30312726){\color[rgb]{0,0,0}\makebox(0,0)[lb]{\smash{$[a+1]$}}}%
    \put(0.38660428,0.30903715){\color[rgb]{0,0,0}\makebox(0,0)[lb]{\smash{$[a+2]$}}}%
    \put(0.09997537,0.06082245){\color[rgb]{0,0,0}\makebox(0,0)[lb]{\smash{$[a+1]$}}}%
    \put(0.38660428,0.06082245){\color[rgb]{0,0,0}\makebox(0,0)[lb]{\smash{$[a+2]$}}}%
    \put(0.6495937,0.06673231){\color[rgb]{0,0,0}\makebox(0,0)[lb]{\smash{$[a]$}}}%
  \end{picture}%
\endgroup%